\documentclass[12pt,reqno]{amsart}
\usepackage{tikz,dsfont}
\usepackage{amsmath, amssymb, graphicx, mathrsfs, hyperref}
\usepackage{pdfpages}
\usepackage{float}
\usepackage{verbatim} 
\usepackage{enumerate}
\newcommand{\ex}{\mathbb{E}}

\newtheoremstyle{dotless}{}{}{\itshape}{}{\bfseries}{}{ }{} 

\newtheorem{lemma}{Lemma}

\newtheorem{prop}{Proposition}

\newtheorem*{remark}{Remark}

\makeatletter
\def\nfootnote{\xdef\@thefnmark{}\@footnotetext}
\makeatother

\begin{document} 

\title{Primes in short intervals: Heuristics and calculations}

\author{ Andrew Granville}

\address{D{\'e}partment  de Math{\'e}matiques et Statistique,   Universit{\'e} de Montr{\'e}al, CP 6128 succ Centre-Ville, Montr{\'e}al, QC  H3C 3J7, Canada.}
   \email{andrew@dms.umontreal.ca}  

\author{Allysa Lumley}
\address{Centre recherche math{\'e}matiques,   Universit{\'e} de Montr{\'e}al, CP 6128 succ Centre-Ville, Montr{\'e}al, QC  H3C 3J7, Canada.}
   \email{ }  

\email{lumley@crm.umontreal.ca}

\dedicatory{Dedicated to the memory of Lord Cherwell $^\dag$}\nfootnote{$^\dag$\ Lord Cherwell's scientific advice to Winston Churchill during the second world war led to the development support and subsequent unveiling of several extraordinary military innovations.
Partly as a consequence of Cherwell's status, scientific research, even in pure mathematics, was never so encouraged as after the war. In 1956, Cherwell returned to Oxford University to pursue his earlier interests, writing a paper with E.M.~Wright on conjectures about the distribution of prime tuples, and another on primes in short intervals, before his death in 1957.}

\thanks{Thanks are due to James Maynard for some helpful remarks on both the content and the exposition,   to Kevin Ford and Drew Sutherland for making various data available, as well as to the three anonymous referees for their helpful comments.}

\begin{abstract}
We formulate, using heuristic reasoning,  conjectures for the range of the number of primes in intervals of length $y$ around $x$, where $y\ll (\log x)^2$. 
In particular we conjecture that the maximum grows surprisingly slowly as $y$ ranges from $\log x$ to $(\log x)^2$.
We will exhibit the available data, showing that it   somewhat supports  our conjectures, though not so well that there may not be room for some modifications.
\end{abstract}

\date{}

\maketitle

\section{Introduction}

We are interested in estimating the maximum and minimum number of primes in a length $y$ sub-interval of $(x,2x]$, denoted by
\[
M(x,y):=\max_{X\in (x,2x]} \pi(X+y)-\pi(X)  \text{ and } m(x,y):=\min_{X\in (x,2x]} \pi(X+y)-\pi(X) ,
\]
respectively, so that
\[
m(x,y)\leq \pi(X+y)-\pi(X)  \leq M(x,y) \text{ whenever } x<X\leq 2x,
\]
and these bounds cannot be improved (by definition).
  It is widely believed that $m(x,y)=0$ for $y\ll (\log x)^2$ though we do not know the precise value of the implicit constant. However there has been little study of how $m(x,y)$ subsequently grows, or of how $M(x,y)$ behaves for $y\ll (\log x)^{2+o(1)}$. In this article we will conjecture a series of guesstimates for $M(x,y)$ and $m(x,y)$ in different ranges, comparing these  estimates to what relevant data we can compute, and discussing some of the issues that prevent us from being too confident of these guesses.

The starting point for our investigations came from a comparison of two known observations:

Based on the (conjectured) size of admissible sets we believe that there exists a constant $c>0$ such that 
\[
M(x,y) \sim \frac y{\log y}
\]
for  $y\leq c \log x$, as long as $y\to \infty$ as $x\to \infty$ (see sections 1.1, 4.1, 8.1 and 9.1).  On the other hand, based on a modification of Cram\'er's probabilistic model \cite{Cra} for the distribution of primes (which in turn is based on Gauss's observation that the primes have density $\frac 1{\log x}$ around $x$), we believe that 
\[
M(x,y) \sim \sigma_+(A) \frac y{\log x}
\]
for $y=(\log x)^A$ with $A>2$, for some constant $\sigma_+(A)>1$, for which $\sigma_+(A)\to 1^+$ as $A\to \infty$ (see sections 1.5, 3.1, and 7.2). 

Therefore it seems that in both ranges,  $M(x,y)$ is roughly linear in $y$: In particular,
 \[
 M(x,y) \sim \frac y{\log\log x} \text{ for } y \text{ a little smaller than }\log x,
 \]
 whereas, if $c_+:=\sigma_+(2)$ then
 \[
 M(x,y)\sim c_+\frac y{\log x}\text{ for } y \text{ a little bigger than } (\log x)^2. 
 \]
If true then $ M(x,y)$ has  quite different slopes, $\frac 1{\log\log x}$ vs. $\frac {c_+}{\log x}$, in these two different ranges, and so there is a substantial change in behaviour
of $M(x,y)$ as $y$ grows from around $ \log x$ to slightly beyond $(\log x)^{2}$. 
Our main goal is  to investigate what happens in-between, though also to give heuristic support for the claims above.

At the end-points of this in-between interval, the above claims suggest that 
\[
M(x,\log x)\sim \frac {\log x}{\log\log x} \text{ whereas } M(x,(\log x)^2) \asymp \log x,
\]
so $M(x,y)$ does not seem to get much bigger as $y$ grows from $\log x$ to $(\log x)^2$; indeed it grows by only a factor of $\log\log x$. This is  very different from before and after this interval: As $y$ goes from $1$ to $\log x$ we expect $M(x,y)$  to grow by a factor of   $\asymp \frac {\log x}{\log\log x} $, and as $y$ goes
from $(\log x)^2$ to $(\log x)^3$ to grow by a similar factor of $\asymp \log x$ (and indeed for any subsequent interval of multiplicative length $\log x$). This does not seem to have been previously observed.

Based on an appropriate heuristic we conjecture that if $1<A<2$ then 
\[
M(x,(\log x)^A)\sim \frac 1{2-A}  \cdot \frac {\log x}{\log\log x};
\]
 more precisely that if $\log x\leq y=o( (\log x)^2)$ then 
\begin{equation} \label{eq: Intermediate intervals}
M(x,y) \sim \frac {\log x}{\log \left(\tfrac{(\log x)^2} y\right) } .
\end{equation}
We will provide data with $x$ up to $10^{12}$ to support this claim, though it should be noted that although this is as far as we have been able to compute, these $x$ are still small enough that   secondary terms are likely to have a significant impact (see sections 1.2, 8.3, 9.2). For this reason we also look at
\[
M(x,2y)/M(x,y)
\] 
because we expect that, as $x\to \infty$ this looks much like $1$ in this range, and $2$ outside this range. However we will compare the data for this ratio to a more precise conjecture.

In this article we will argue that there are four ranges of $y$ in each of which we expect  different behaviour for   $M(x,y)$, namely:
\[
y\ll \log x;\ \log x\ll y =o(\log x)^2;\ y\asymp (\log x)^2;\ \text{ and } y/(\log x)^2\to \infty \text{ with } y\leq x.
\]
We will present these separately in the introduction though there is significant overlap in the theory; and when it comes to presenting data for a given value of $x$ up to which we can compute, it is often unclear where one $y$-interval should end and the next begin.

\subsection{Guesstimates for very short intervals: $y\ll \log x$}

We believe that if $y\leq \log x$ then 
\begin{equation} \label{eq: Very short interval}
M(x,y) \sim \frac y{\log y}
\end{equation}
provided $x,y\to \infty$.
We will now formulate a more precise conjecture than this for $y\leq (1-o(1)) \log x$:
A set of integers $A$   is \emph{admissible} if for every prime $p$ there is a residue class mod $p$ that does not contain any element from the set (otherwise $A$ is \emph{inadmissible}).  Let $S(y)$ denote the maximum size of an admissible set $A$  which is a subset of $[1,y]$,\footnote{We say that $A$, and any translate of $A$, has \emph{length} $\leq y$.} so that
\[
M(x,y) \leq  S(y) \text{ if } x\geq y
\]
(for if $X<p_1<\dots<p_k\leq X+y$ are primes then $\{ p_1-X,\ldots, p_k-X\}$ is an admissible set).
 We believe that if $y\leq (1-o(1))\log x$ then\footnote{The ``$o(1)$'' here can be interpreted as saying that for any fixed $\epsilon>0$, if $x$ is sufficiently large then \eqref{eq: Very short interval2} holds for all $y\leq (1-\epsilon)\log x$.}
\begin{equation} \label{eq: Very short interval2}
M(x,y) = S(y).
\end{equation}
These two conjectures are consistent since it is believed that $S(y)\sim \frac y{\log y}$.  
The data seems to confirm the conjecture \eqref{eq: Very short interval2} for $x=10^k$ for $k=9,10,11$ and $12$:

\begin{figure}[H]\centering{
\includegraphics[scale=.5]{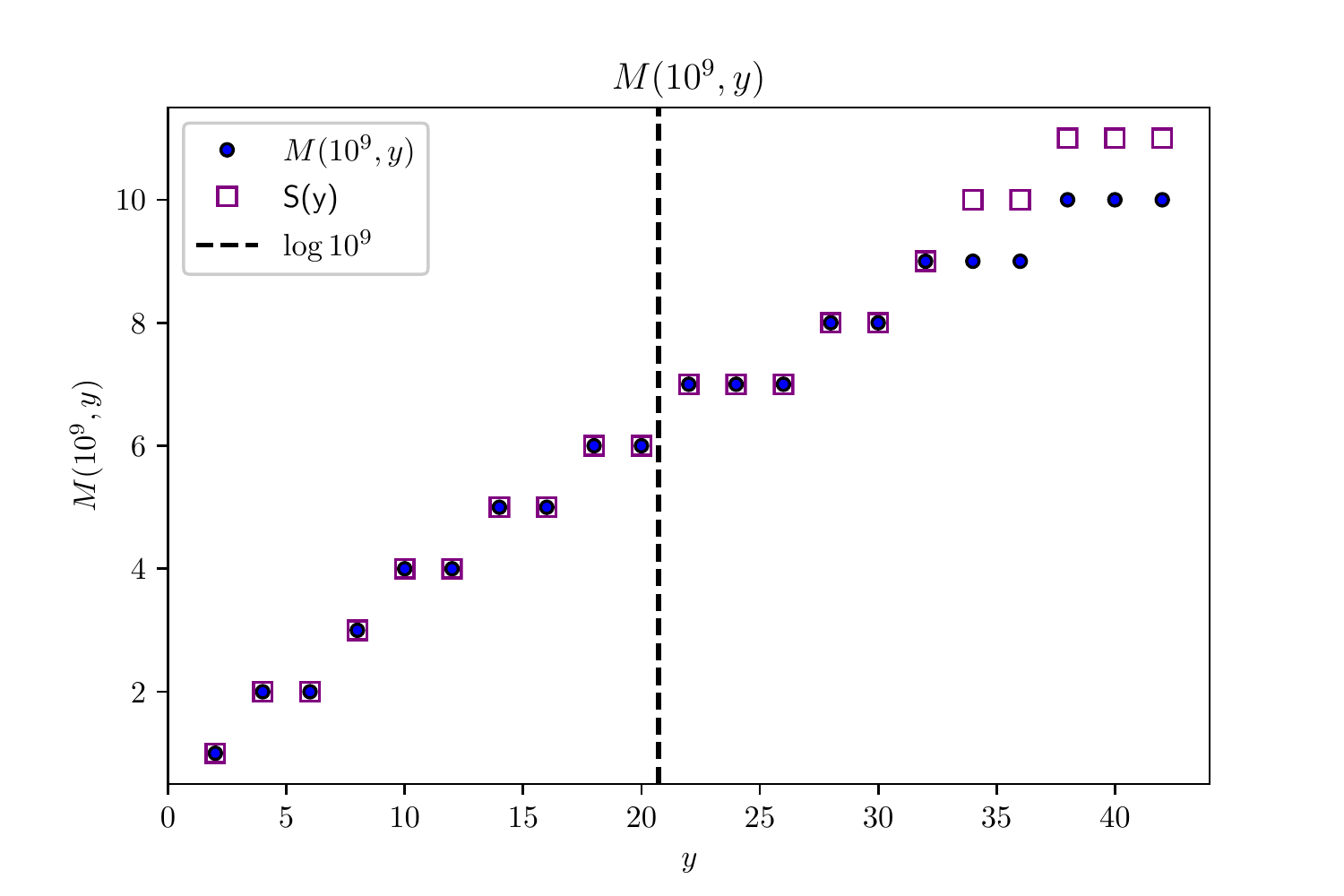}
\includegraphics[scale=.5]{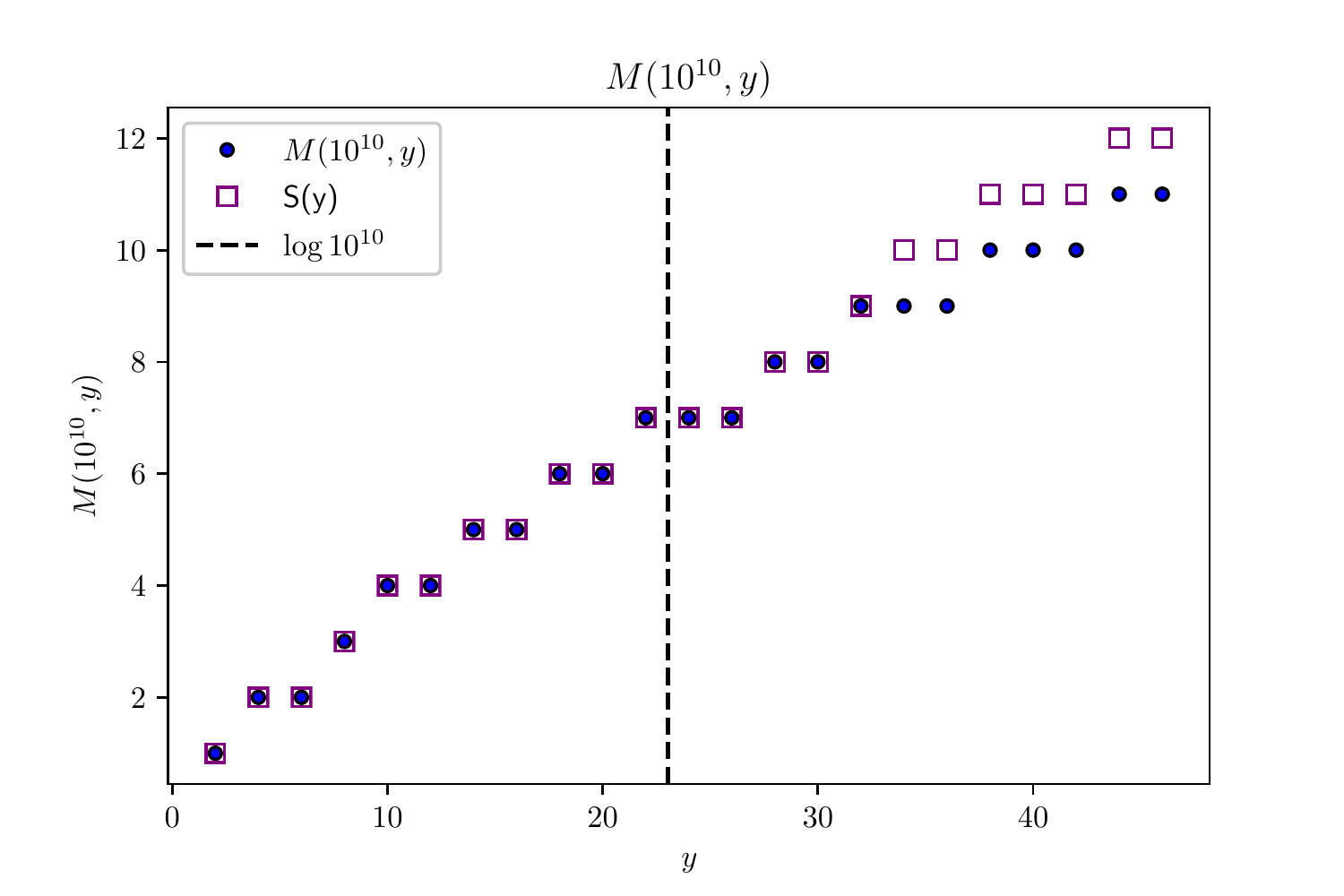}
\newline

\includegraphics[scale=.5]{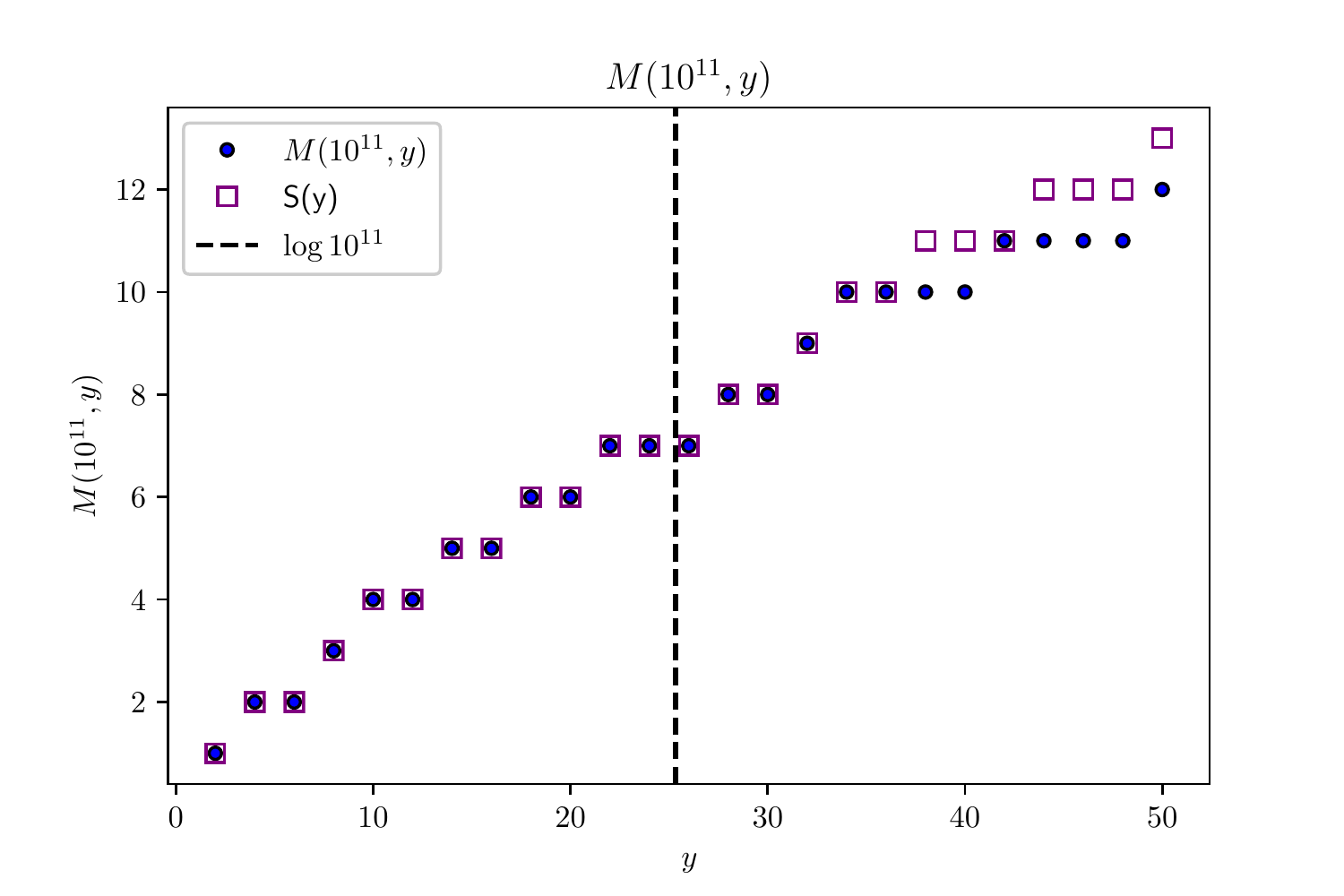}
\includegraphics[scale=.5]{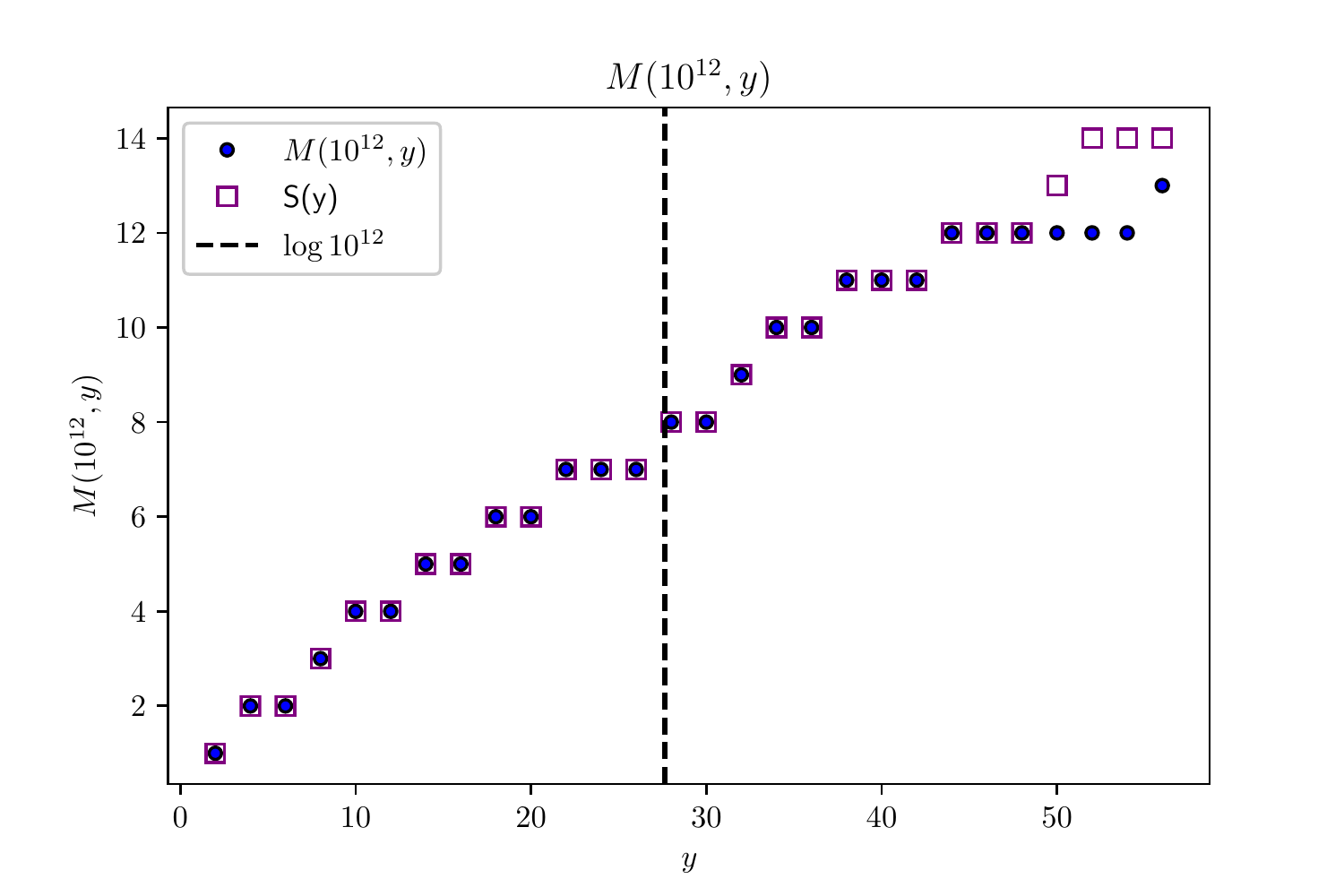} }
\caption{$M(x,y)$ vs. $S(y)$ for $x=10^k, k=9,\dots,12$ and $y\leq 2\log x$.\newline  We observe that $M(x,y)=S(y)$ up to the dashed line   at $y=\log x$}
\end{figure} 

\noindent In these graphs, for each $y$ (the horizontal axis), a  colored-in dot   represents
$M(x,y)$, and an empty box represents the value of $S(y)$. In this data, it appears that $M(x,y) = S(y)$ for $y$ up to about $\tfrac 32 \log x$, 
and then $M(x,y)$ is at worst a little less than $S(y)$ for $y$ between  $\tfrac 32 \log x$ and $2 \log x$, for these values of $x$.  Although we do believe that 
$M(x,y) = S(y)$ for all $y\leq (1-\epsilon) \log x$, for all sufficiently large $x$, and perhaps even for all $y\leq \log x$ for all $x$, we do not believe that this should be so for 
$y>(1+\epsilon) \log x$ and that the data we see here is an artifice of the relatively small values of $x$ we can compute with.  Indeed, if we are wrong about this, if 
$M(x,y) = S(y)$ for a sequence of $x,y$ with  $y>(1+\epsilon) \log x$ and $x$ arbitrarily large, then this would contradict the key conjecture in section 1.2.

More discussion of this heuristic in section 4, as well as in sections 8.1 and 9.1

\subsection{Intermediate length intervals: $ \log x\leq y=o( (\log x)^2)$}
In this range we believe that \eqref{eq: Intermediate intervals} holds:
\[
M(x,y) \sim L(x,y) \text{ where } L(x,y):=\frac {\log x}{\log \left(\tfrac{(\log x)^2} y\right) } .
\]
However, when comparing this prediction to the data,  it is not obvious how to interpret ``$o( (\log x)^2)$'' for a given $x$-value. We have made the rather arbitrary choice of $\tfrac 12 (\log x)^2$ as the upper bound for the $y$-range. We have also taken
$\tfrac 12 \log x$ as a lower bound which reflects our uncertainty as to whether things can really be predicted so precisely, though we have marked $\log x$ with a dashed line.

\begin{figure}[H]\centering{
\includegraphics[scale=.5]{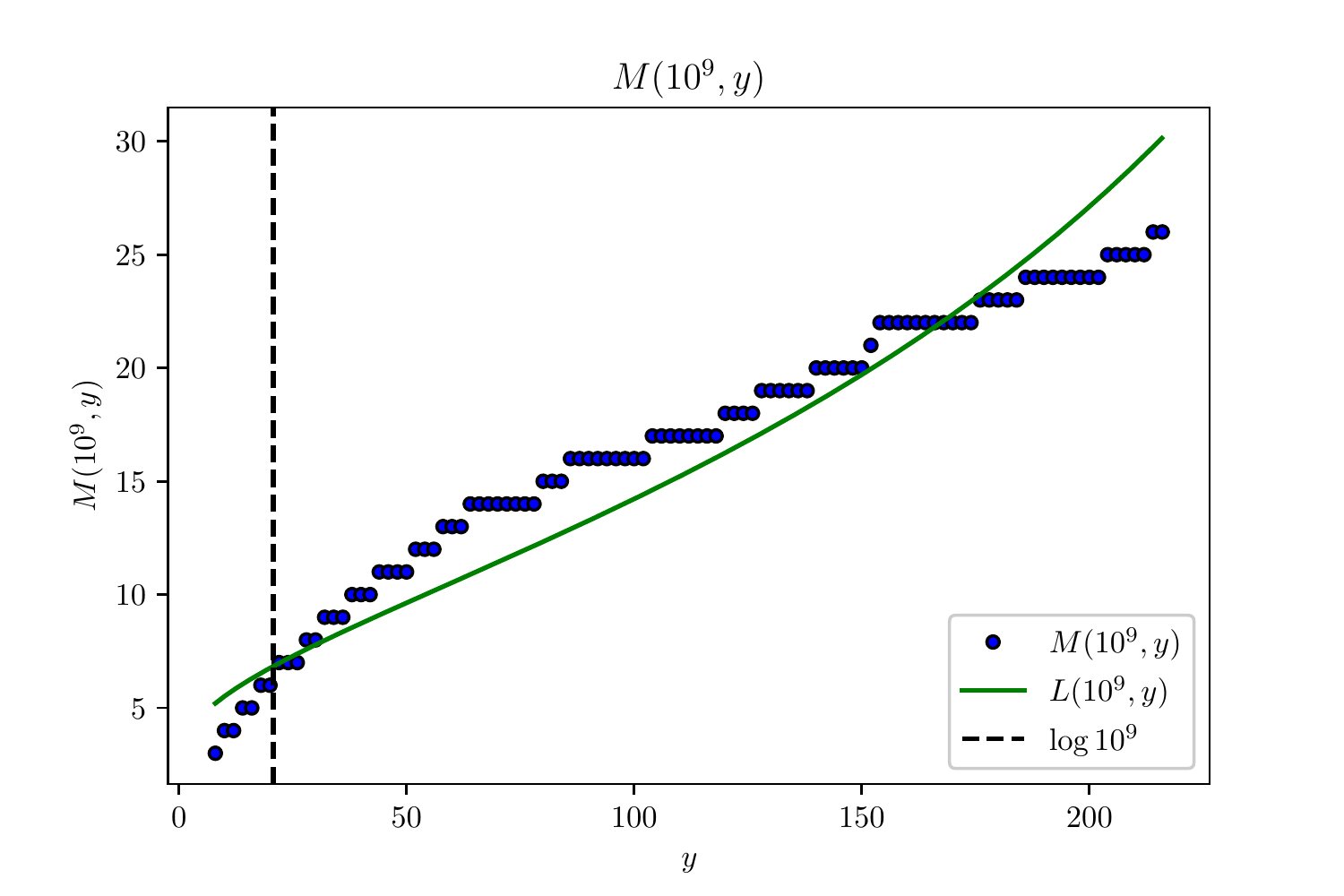}
\includegraphics[scale=.5]{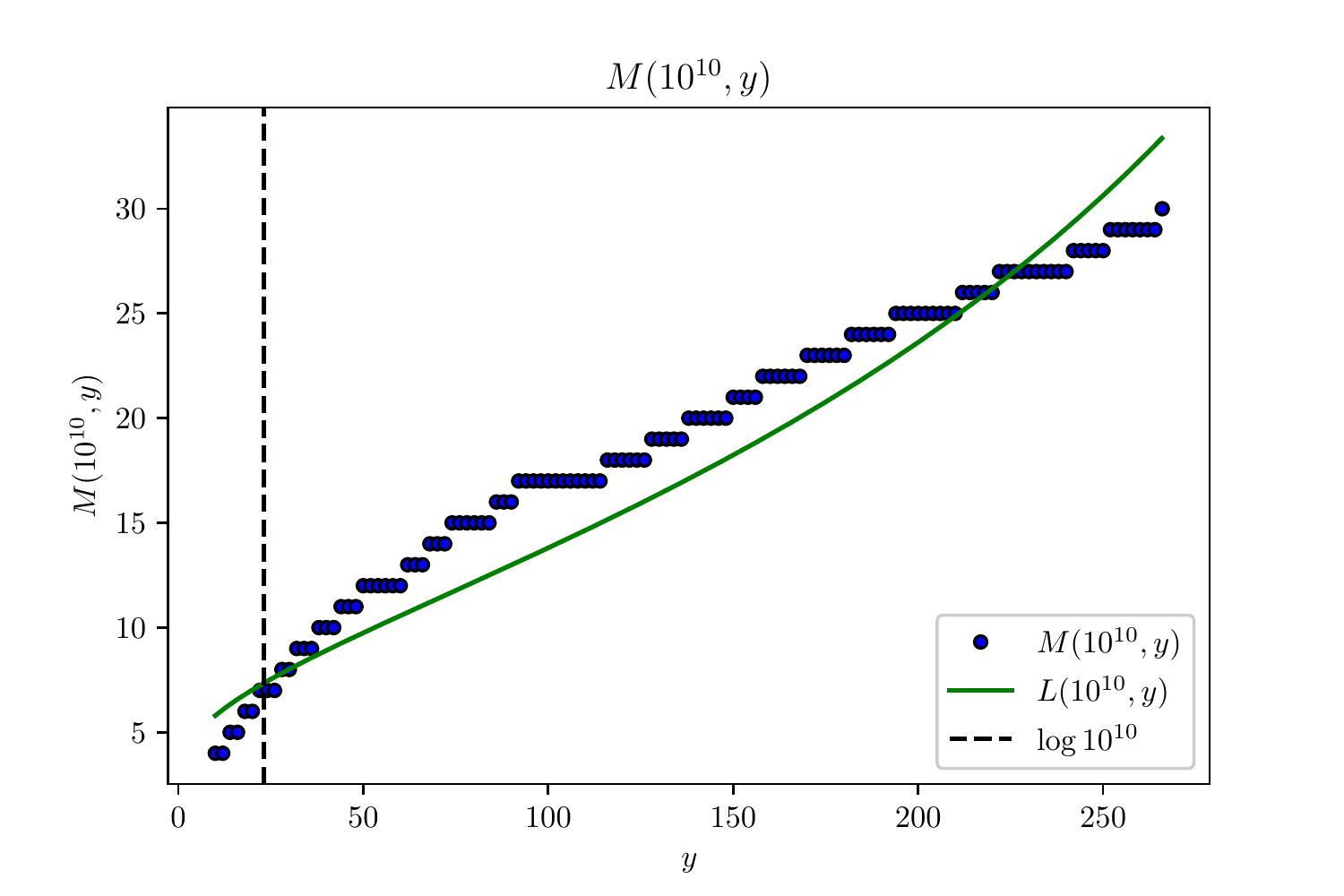}
\newline

\includegraphics[scale=.5]{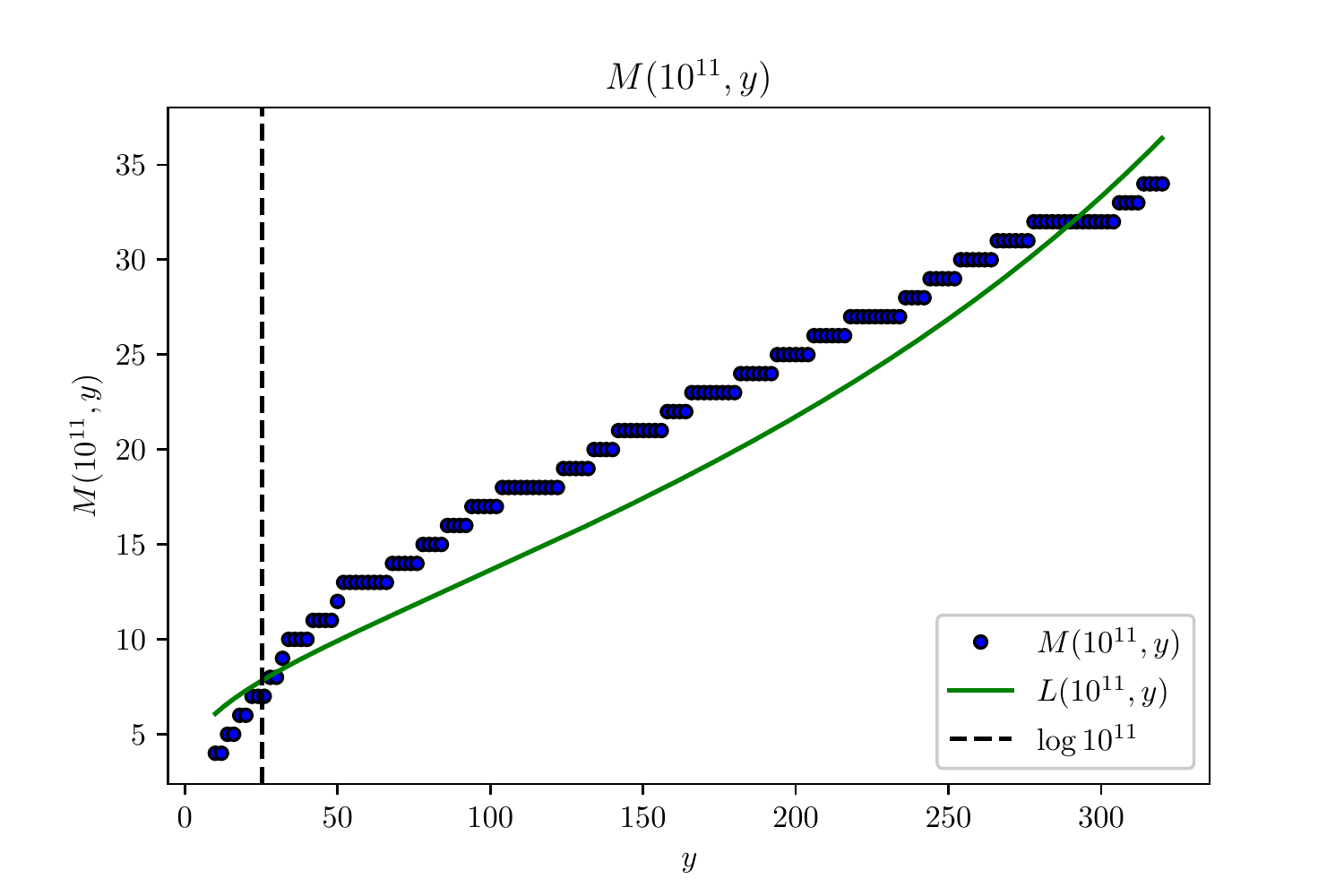}
\includegraphics[scale=.5]{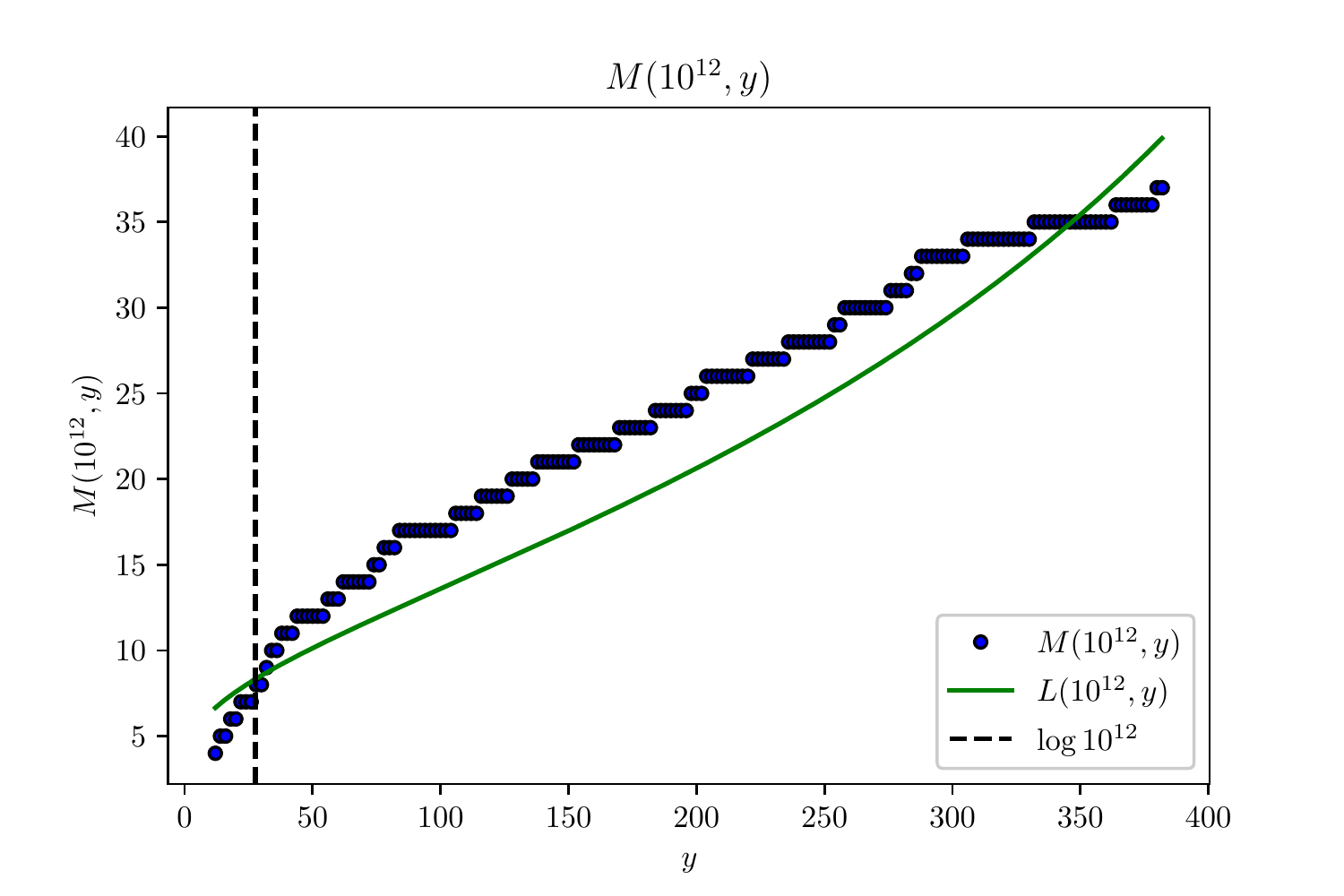} }
\caption{$M(x,y)$ vs.$L(x,y)$ for $x=10^k, k=9,\dots,12$  and $\tfrac 12\log x\leq   y\leq \tfrac 12(\log x)^2$.  Dashed line   at $y=\log x$, which is the end of the range of the $M(x,y)=S(y)$ conjecture.}
\end{figure} 

\noindent Here, for each $y$ (the horizontal axis), a  colored-in dot   represents $M(x,y)$, and the 
  continuous curve  $L(x,y)$ (our prediction in  \eqref{eq: Intermediate intervals}). Our prediction and the data seem to co-incide at $y=\log x$
  (where the dashed line is), and again at a point that seems to be slowly increasing (towards $\tfrac 12(\log x)^2$) as $x$ grows.
 The graph indicates that our prediction provides a pretty good approximation to the data in the whole range, though it is concave up whereas the data itself appears to yield a curve that is concave down. We have no explanation for that.

\subsection{The maximum on longer intervals: $y\asymp (\log x)^2$}
Here we mean that $y=t(\log x)^2$ for some fixed value of $t$.
In this range we will need to define two implicit functions to formulate our conjectures for $m(x,y)$ and $M(x,y)$:
For every given $t>0$ consider the equation
\[
u(\log u-\log t-1) +t = 1.
\]
We will show that for every $t>0$ there is a unique solution $u_+(t)$ with $u_+(t)>t$.
If $0<t<1$ there is no solution in $u\in (0,t)$, so we let $u_-(t)=0$.
If $t>1$ then there is a unique solution $u_-(t)$ with $0<u_-(t)<t$.  
We believe that there exist constants $c_-,c_+>0$ such that if
 $y=t(\log x)^2$ then
\begin{equation} \label{eq: LogSquared intervals}
m(x,y)\sim u_-(c_-t) \log x \text{ and } M(x,y)\sim u_+(c_+t) \log x.
\end{equation}
We will see at the end of section 3  that $c_\pm$ are constants that can be defined in terms of sieving intervals.
We know that $c_+\geq 1.015\dots$ and $c_-\leq \frac {e^\gamma}2 = 0.890536\dots$, and perhaps both of these inequalities should be equalities.\footnote{We will assume that $c_+ =1.015\dots$ and $c_-  = 0.8905\dots$ throughout for the purpose of comparing our conjectures to our data. We will explain   the significance of $1.015\dots$ at the end of section 3.}
Here is the   data for $M(x,y)$ in this range:

\begin{figure}[H]\centering{
\includegraphics[scale=.5]{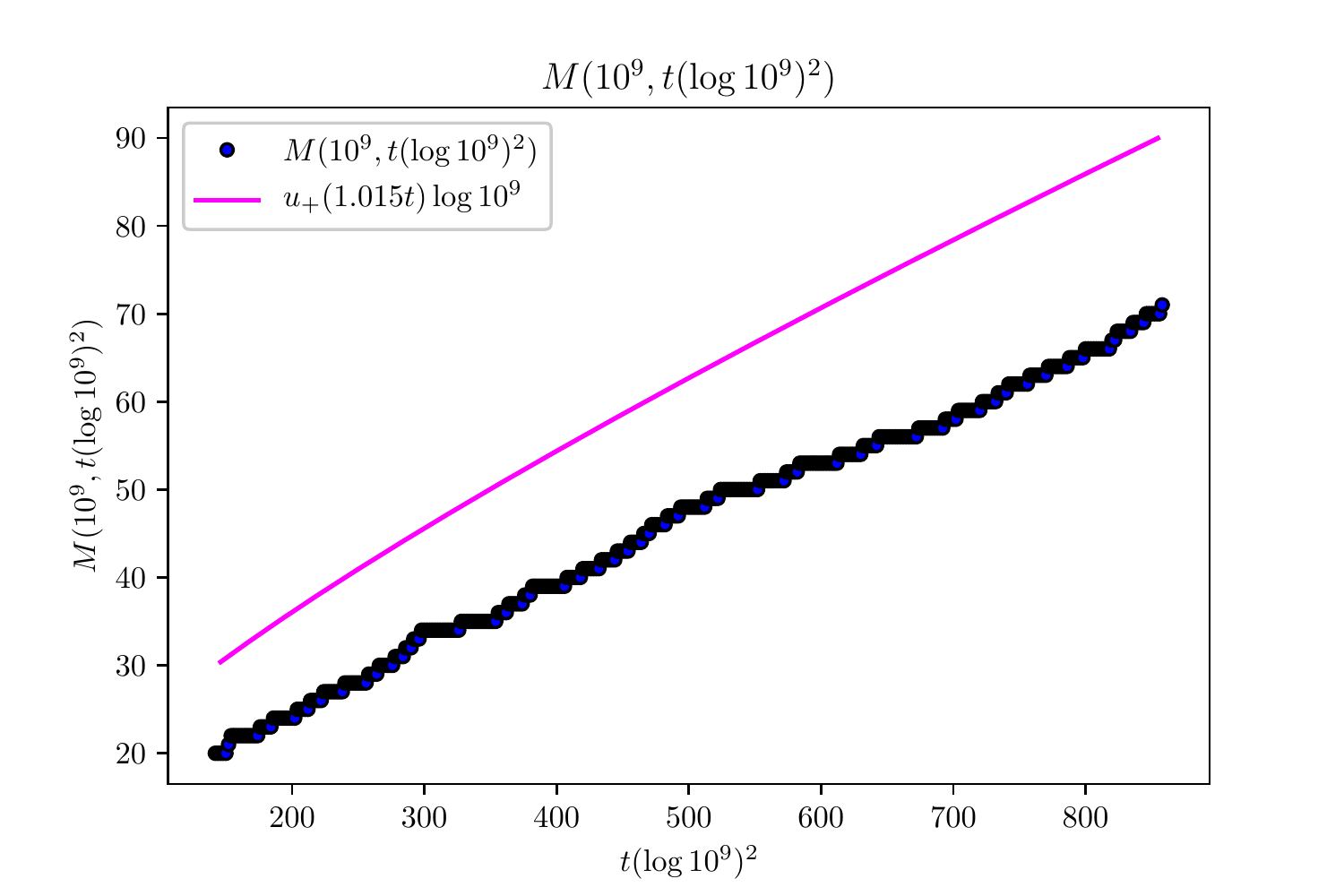}
\includegraphics[scale=.5]{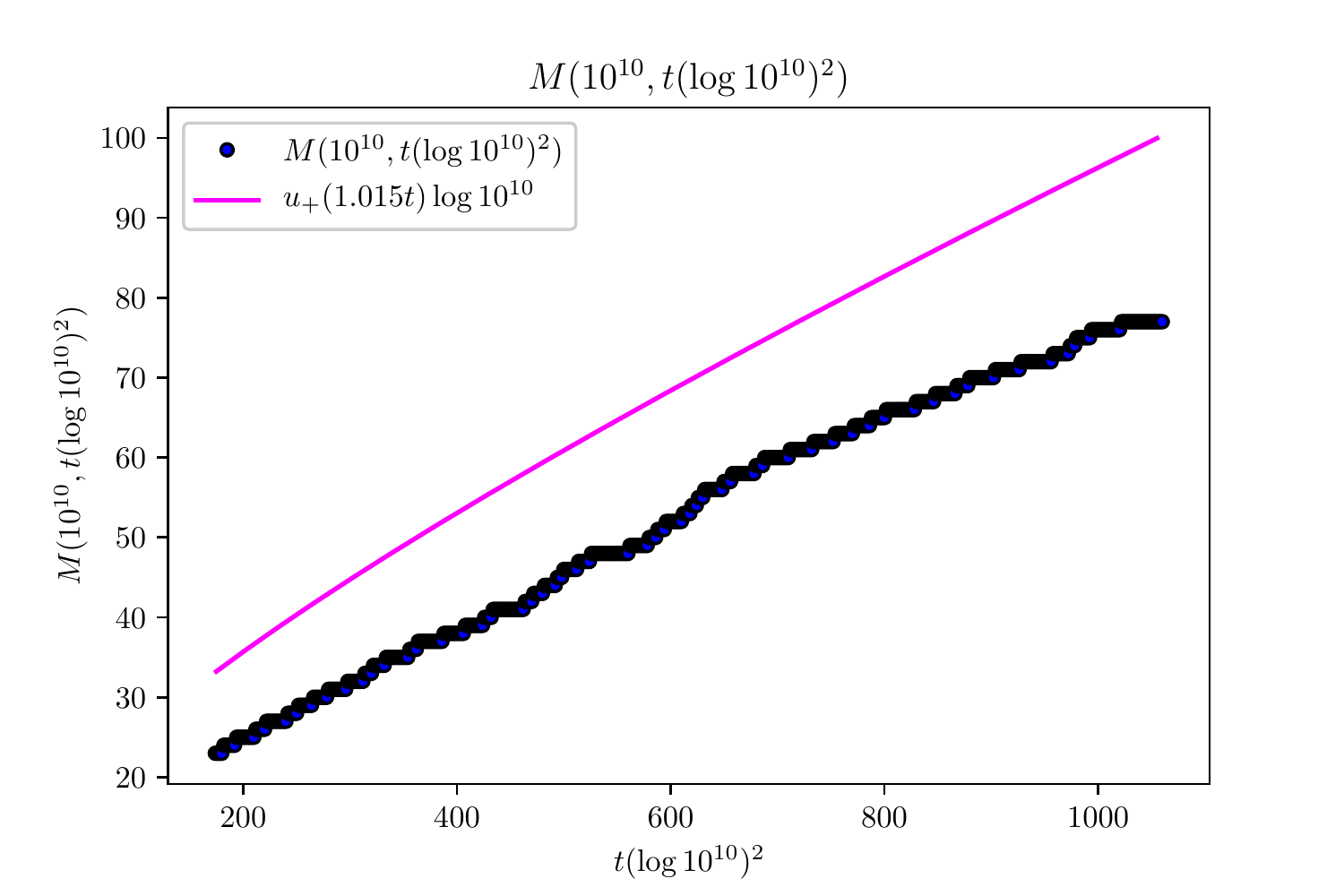}
\newline

\includegraphics[scale=.5]{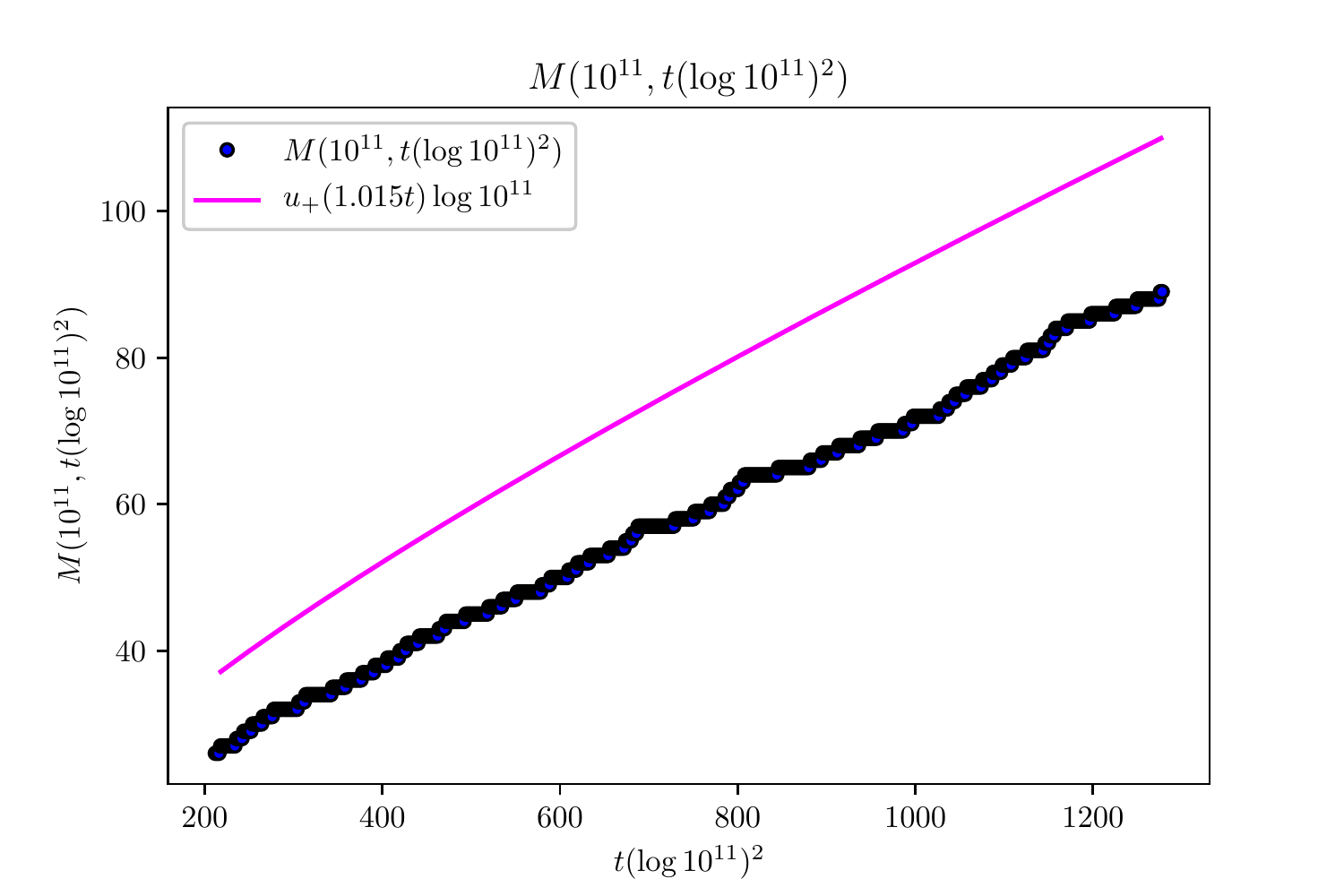}
\includegraphics[scale=.5]{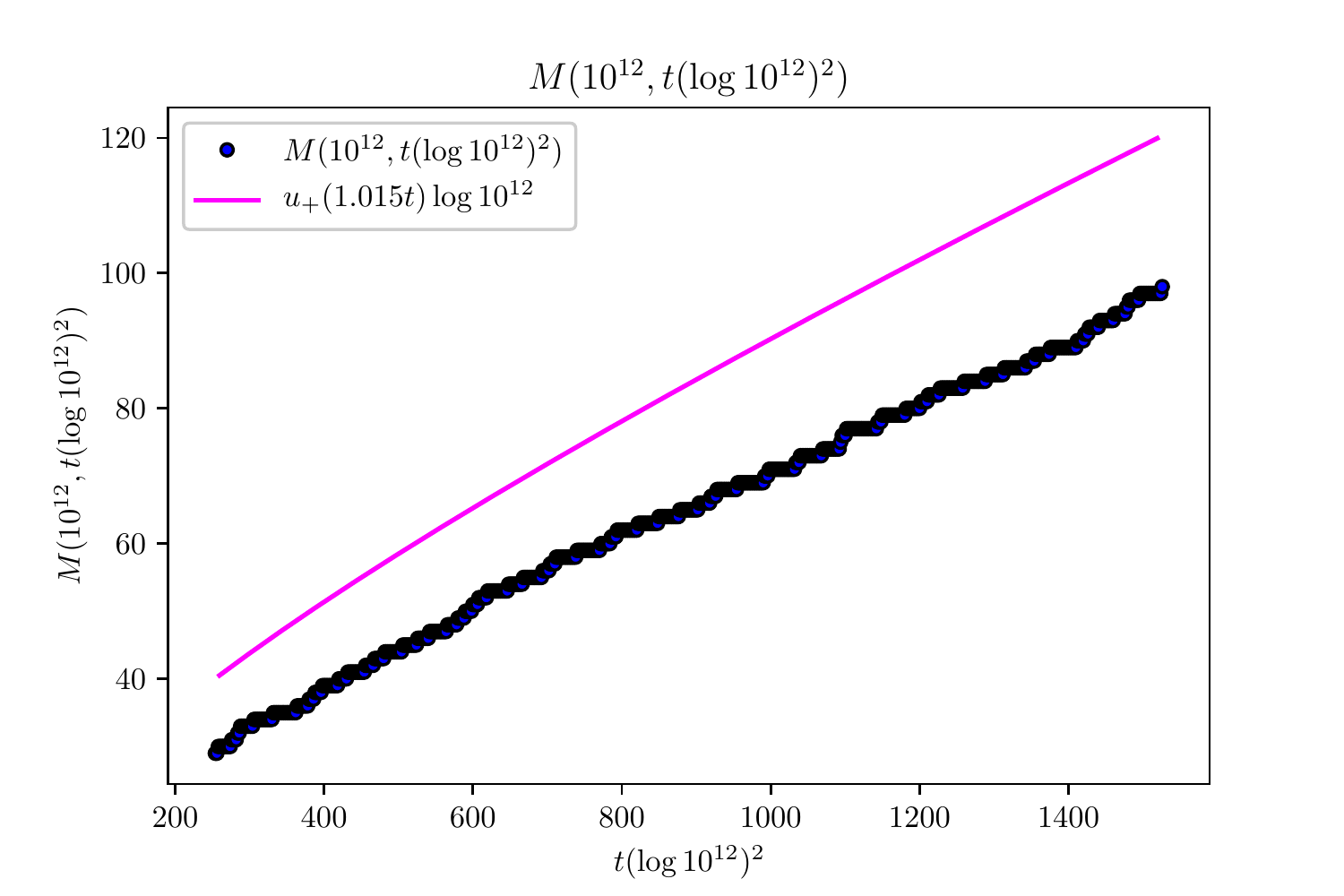} }
\caption{$M(x,y)$ vs. $u_+(1.015 t)  \log x$ where $y=t(\log x)^2$
\newline . \hskip .7in for $x=10^k, k=9,\dots,12$   and $\tfrac 13(\log x)^2\leq   y\leq 2(\log x)^2$.}
\end{figure} 

Here, for each $y$ (the horizontal axis), a  colored-in dot   represents $M(x,y)$, and the
   red curve represents our prediction $u_+(1.015 t)  \log x$ where $y=t(\log x)^2$.
 It appears that this prediction is  too large  by a factor of about $35\%$  (and if $c_+ $ is larger than $1.015$ then the red curve will be even further above the data). However we   believe this is  a consequence of only calculating up to $x=10^{12}$ and hopefully the data will get closer to our curve the larger $x$ gets.\footnote{One referee asks whether we expect that $u_+(c_+t)  \log x\geq M(x,t(\log x)^2)$ will persist for larger $x$; we have no idea how to make predictions that are this precise, and doubt the  value of trying to do so given how far out our predictions currently are from the data!}
 In this range for $y$ it is  already well-known that  data for  the minimum does not yet satisfy the standard conjectures:

\subsection{The minimum on longer intervals: $y\asymp (\log x)^2$}

The prediction \eqref{eq: LogSquared intervals} implies that if $c_-t<1$ then $m(x,t(\log x)^2)=0$ but not if $c_-t>1$. That is, we conjecture the following lower bound for the maximal gap between consecutive primes: 
 \[
 \max_{x<p_n\leq 2x} p_{n+1}-p_n \sim c_-^{-1} ( \log x)^2 \geq 2e^{-\gamma}( \log x)^2;
 \]
 and it is feasible that we have equality here.  This is larger   than Cram\'er's original conjecture (that this maximal gap is $\sim ( \log x)^2$).
 As we will discuss, Cram\'er's reasoning is  flawed by  failing to take account of divisibility by small primes  (a point originally made by the first author  back in \cite{Gr1} and recently re-iterated by the in-depth analysis of Banks, Ford and Tao in \cite{BFT}.)  However the data does not really support either conjecture, as the 
largest gap between consecutive primes that has been found is about $.9206 (\log x)^2$ (a shortfall of around $22\% $ from $2e^{-\gamma}\approx 1.1229\cdots$).  
\medskip

\begin{figure}[H]\centering{

\centerline{\vbox{\offinterlineskip \halign{\vrule #&\ \ #\ \hfill
&& \hfill\vrule #&\ \ \hfill # \hfill\ \ \cr \noalign{\hrule} \cr
height5pt&\omit && \omit && \omit & \cr &$ \ \ p_n $&& $p_{n+1}-p_n$ &&
 \hfill$(p_{n+1}-p_n)/\log^2 p_n$\hfill & \cr
height5pt&\omit && \omit && \omit & \cr \noalign{\hrule} \cr
height3pt&\omit && \omit && \omit & \cr
& 113   && \    14  && .6264 & \cr
height3pt&\omit && \omit && \omit & \cr
& 1327  && \  34  && .6576 & \cr
height3pt&\omit && \omit && \omit & \cr
& 31397 && \ 72 && .6715 & \cr
height3pt&\omit && \omit && \omit & \cr
& 370261 && 112 && .6812 & \cr
height3pt&\omit && \omit && \omit & \cr
& 2010733 && 148 && .7026 & \cr
height3pt&\omit && \omit && \omit & \cr
& 20831323 && 210 && .7395 & \cr
height3pt&\omit && \omit && \omit & \cr
& 25056082087 && 456 && .7953 & \cr
height3pt&\omit && \omit && \omit & \cr
& 2614941710599 && 652 && .7975 & \cr
height3pt&\omit && \omit && \omit & \cr
& 19581334192423 && 766 && .8178 & \cr
height3pt&\omit && \omit && \omit & \cr
& 218209405436543 &&   906 &&   .8311 & \cr
height3pt&\omit && \omit && \omit & \cr
&  1693182318746371 && 1132  && \hskip -.05in .9206& \cr
height3pt&\omit && \omit && \omit & \cr
 \noalign{\hrule}}} }
 
\caption{ (Known) record-breaking gaps between primes}}
\end{figure} 

\smallskip

In \cite{BFT} Banks, Ford and Tao graphed how the maximal gap between primes grows, as compared to the proposed asymptotics $2e^{-\gamma}( \log x)^2, (\log x)^2$ and the more precise $(\log x)(\log x-\log\log x)$.   In section 9.5 we discuss the heuristic justification for these conjectures and variants. All such heuristics seem to suggest that the maximal gap between consecutive primes up to $x$ should grow like $ \log x (a\log x+b\log\log x+c)$ for some constants $a,b,c$. The only possibilities for $a$ seem to be $a=1$ or $2e^{-\gamma}$, though there are many possible guesses for $b$ and $c$. Here we graph $2e^{-\gamma}( \log x)^2$ and $(\log x)^2$ as well as the best fit functions of the form $ \log x (a\log x+b\log\log x+c)$ where $a=1$ or $2e^{-\gamma}$.\footnote{One referee correctly feels that it is inappropriate to try to fit a justification to the data but, who knows, perhaps some enterprising future researcher will see a clearly good reason for our favourite candidate,
$ \log x (2e^{-\gamma}\log x-5\log\log x+6)$.}

\begin{figure}[H]\centering{
\includegraphics[scale=1]{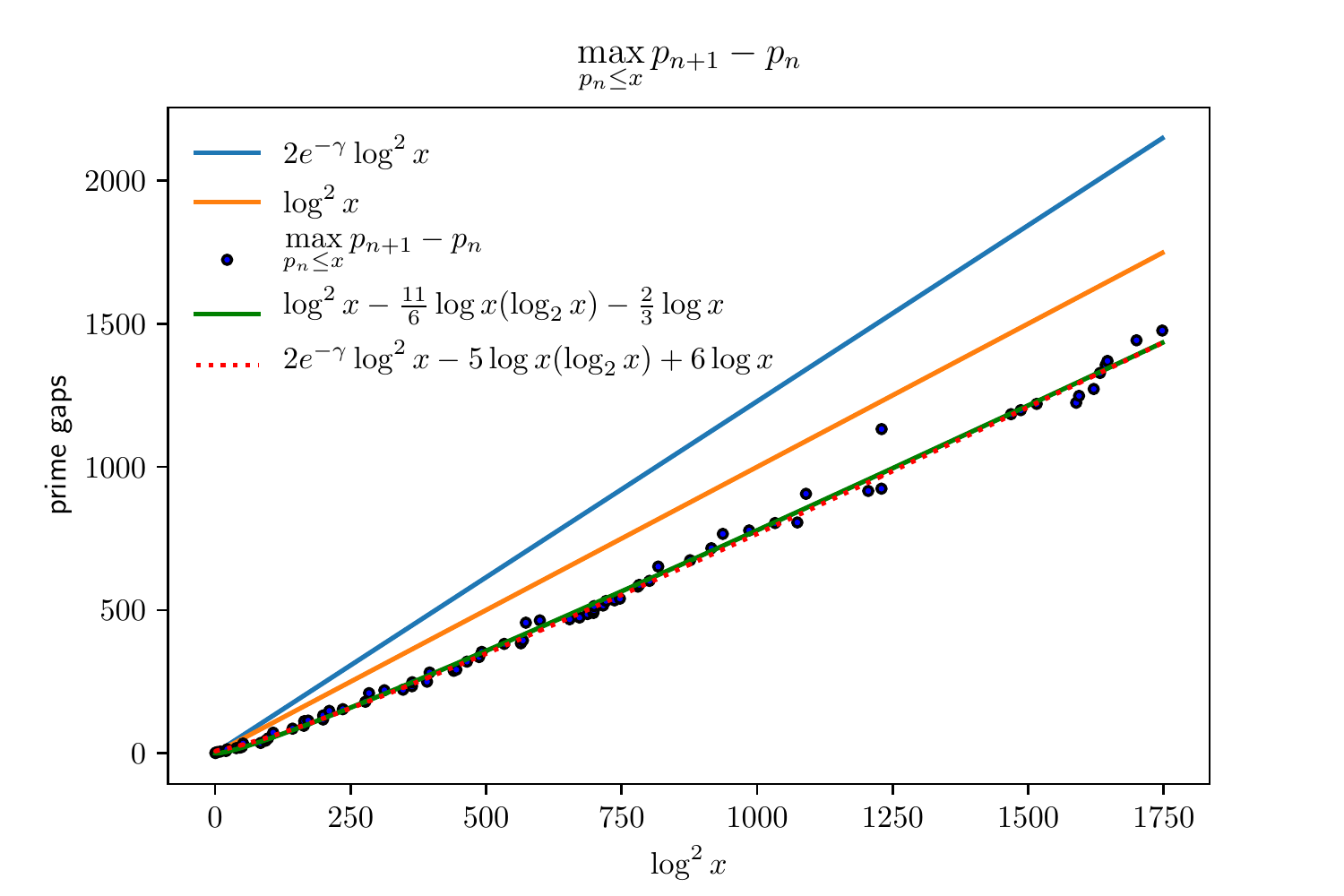}
\caption{$\displaystyle\max_{p_n\leq x} (p_{n+1}-p_n)$  vs. Conjectured approximations  } }
\end{figure} 

\noindent The data for the largest gap between consecutive primes is substantially smaller than our two  predictions. No one has suggested a good reason for this shortfall, though in appendix A we explain how at least some of this shortfall is due to the use of asymptotic estimates for primes and sieves, for relatively small values.

In Figure 5, we have also graphed  the best fit to the data of curves of the form $ \log x (a\log x+b\log\log x+c)$ with $a=1$ and $2e^{-\gamma}$, and the fit is tight. This suggests that we should be looking harder at possible secondary terms and reasons why they might occur.
\bigskip

If \eqref{eq: LogSquared intervals} really does hold then $m(x,y)\sim u_-(c_-t) \log x$ for $y=t(\log x)^2$, where 
$u_-(c_-t)=0$ when $c_-t\leq 1$, but $u_-(c_-t)>0$ for $c_-t>1$. It is of interest to compare this prediction for $m(x,y)$ to the data, and we will assume that $c_-=\frac {e^\gamma}2 = 0.8905 \dots$ for the purpose of comparison:


\begin{figure}[H]\centering{
\includegraphics[scale=.5]{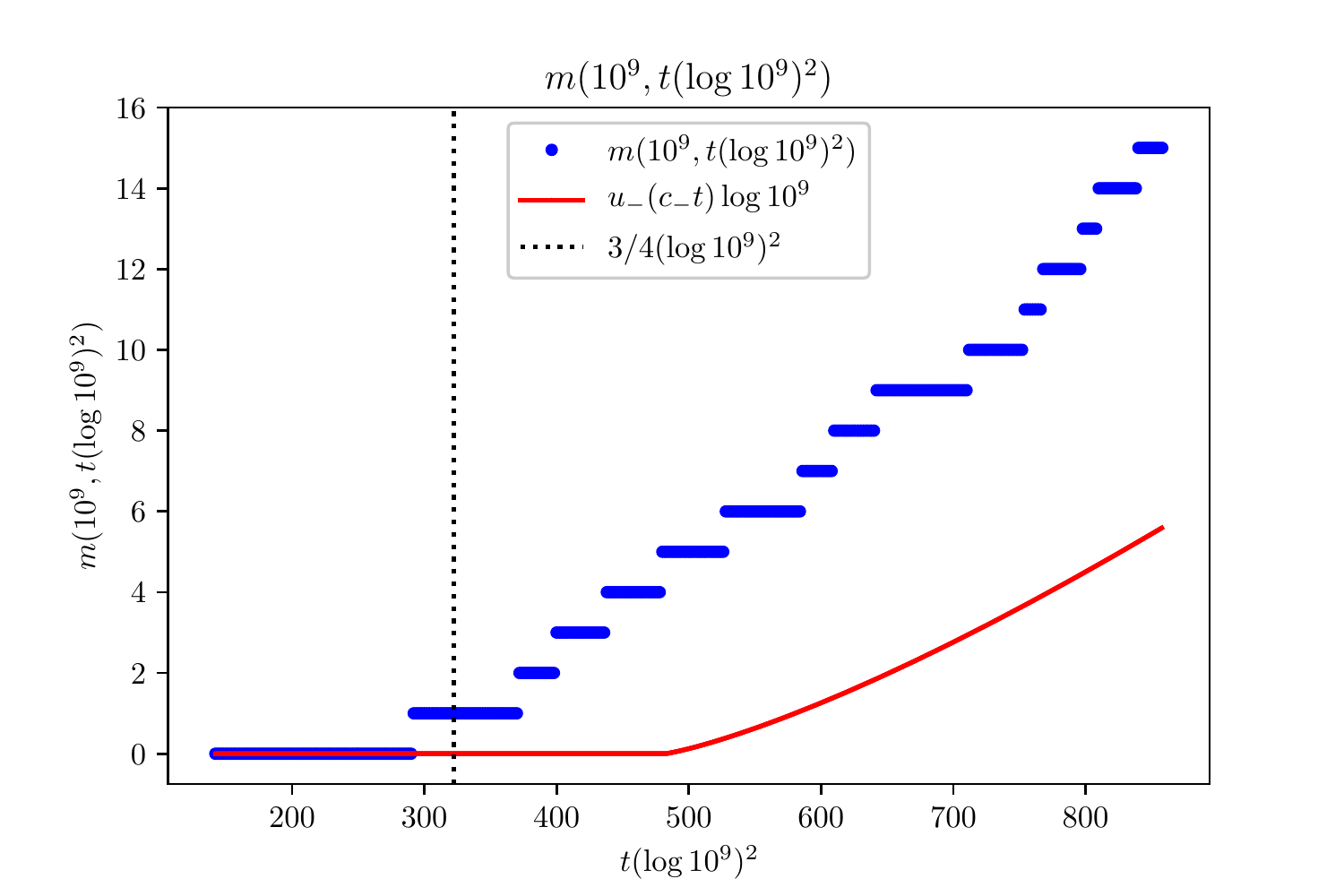}
\includegraphics[scale=.5]{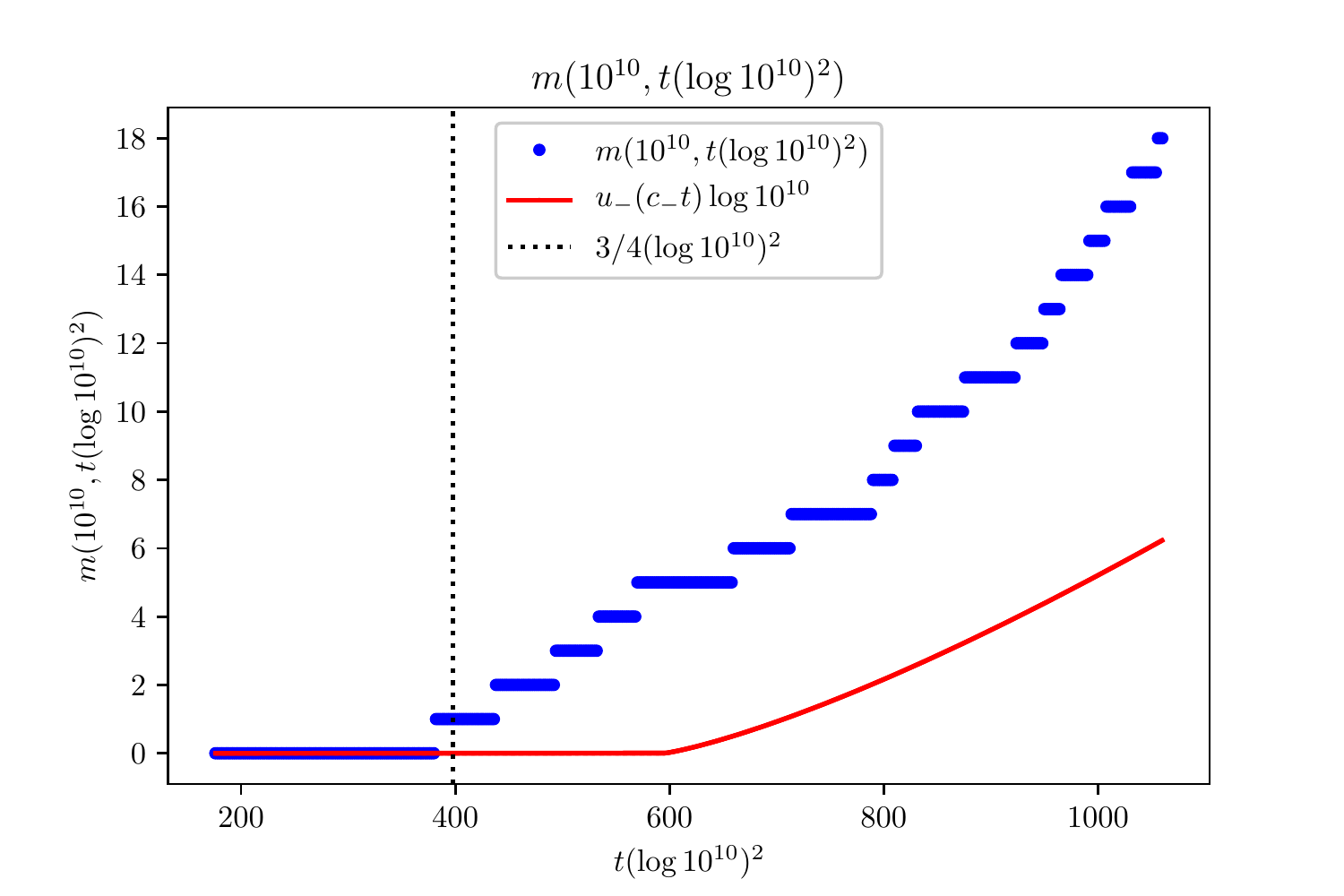}
\newline

\includegraphics[scale=.5]{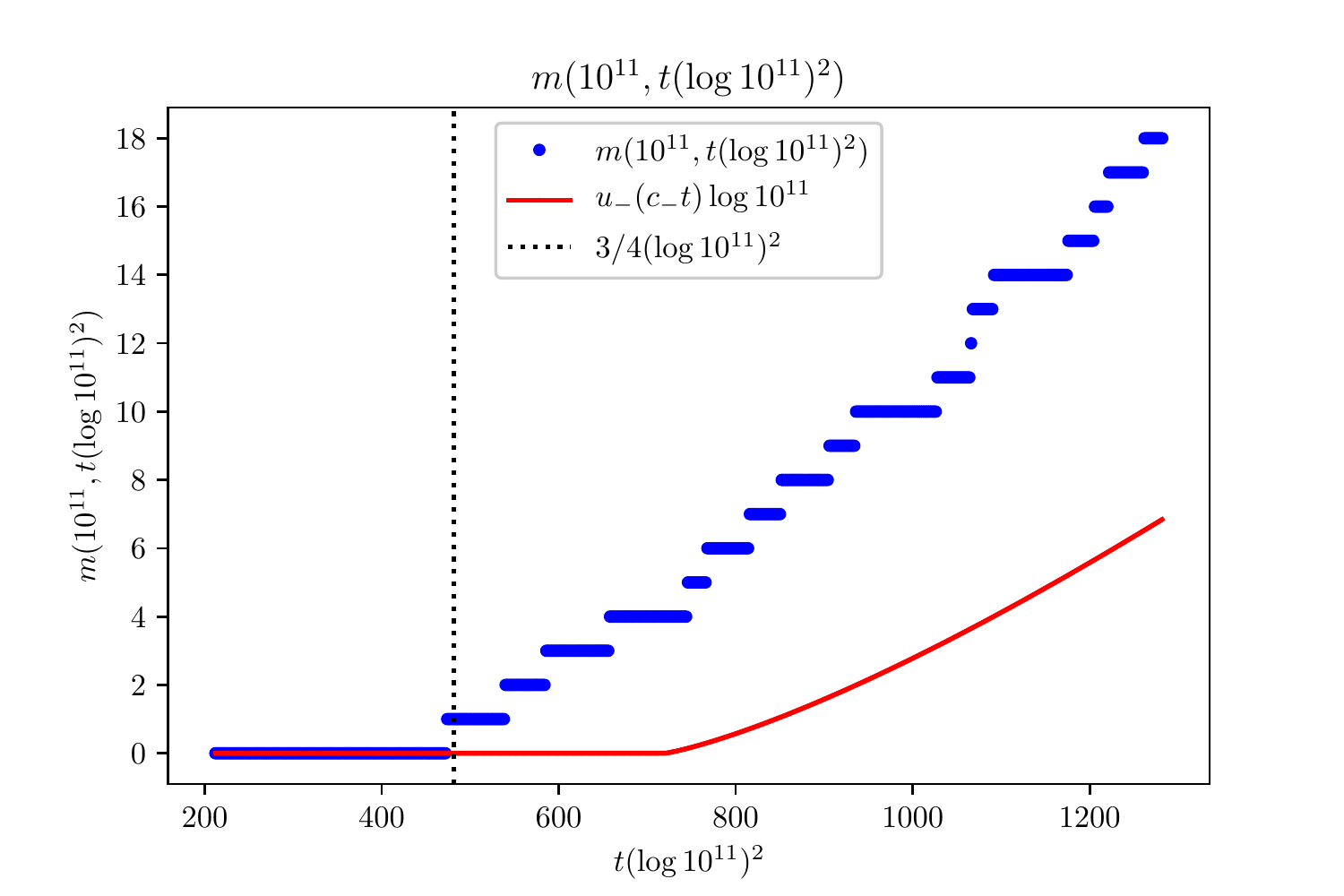}
\includegraphics[scale=.5]{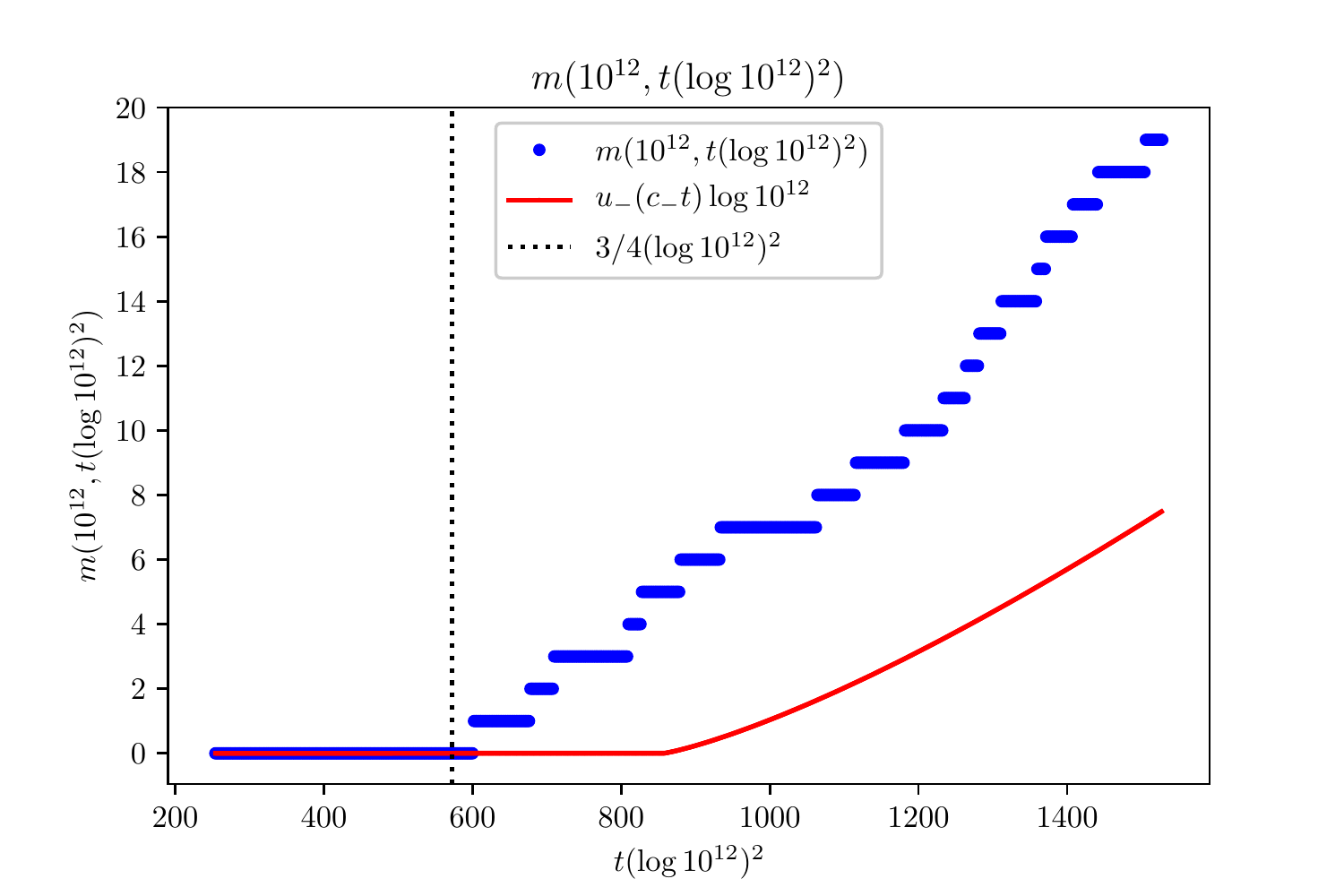}
 }
 \caption{$m(x,y)$ vs. $u_-(0.8905 t)  \log x$ where $y=t(\log x)^2$
 \newline . \hskip .7in for $x=10^k, k=9,\dots,12$ and $\tfrac 13(\log x)^2\leq   y\leq 2(\log x)^2$.}
\end{figure} 

For these values of $x$ it appears that the smallest $y$ for which $m(x,y)>0$ is at about $y=\tfrac 34 (\log x)^2$, which is significantly smaller   than in the prediction (though the ratio $y/(\log x)^2$ appears to be growing slowly with $x$).  This confirms what we saw in the previous two figures when studying $\displaystyle\max_{p_n\leq x} (p_{n+1}-p_n)$. We plotted the maximum $M(x,y)$ vs our prediction in this same range in  Figure 3 and that data there appears to have a similar shape to our prediction. However it is not obvious here whether the data for the minimum, $m(x,y)$, has a similar shape to our prediction.

We now compare our predictions for both the maxima and the minima with the data in the range $ \tfrac 13(\log x)^2\leq   y\leq 2(\log x)^2$, on the same graph, to get  a better sense of how well these fit:

\begin{figure}[H]\centering{
\includegraphics[scale=.5]{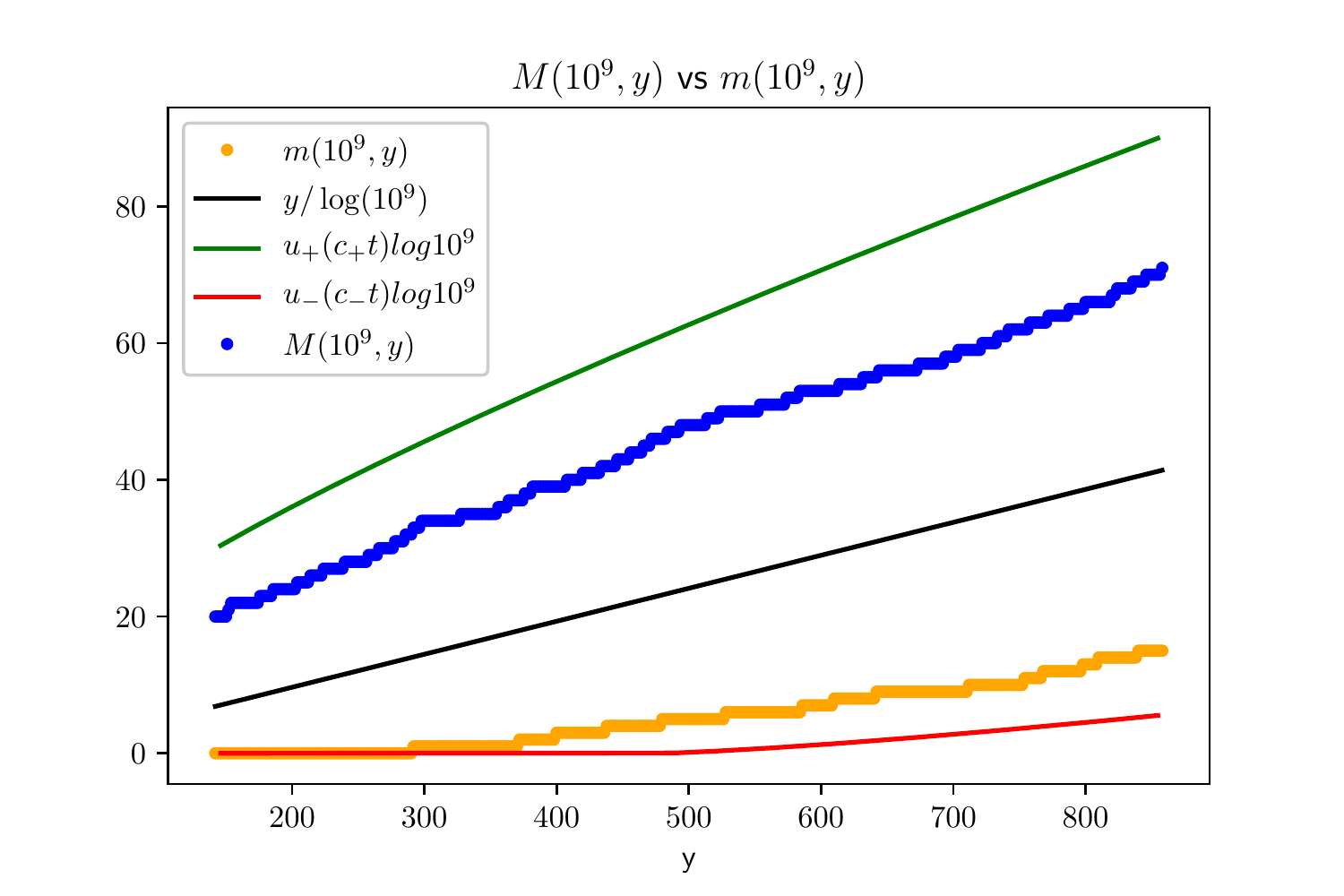}
\includegraphics[scale=.5]{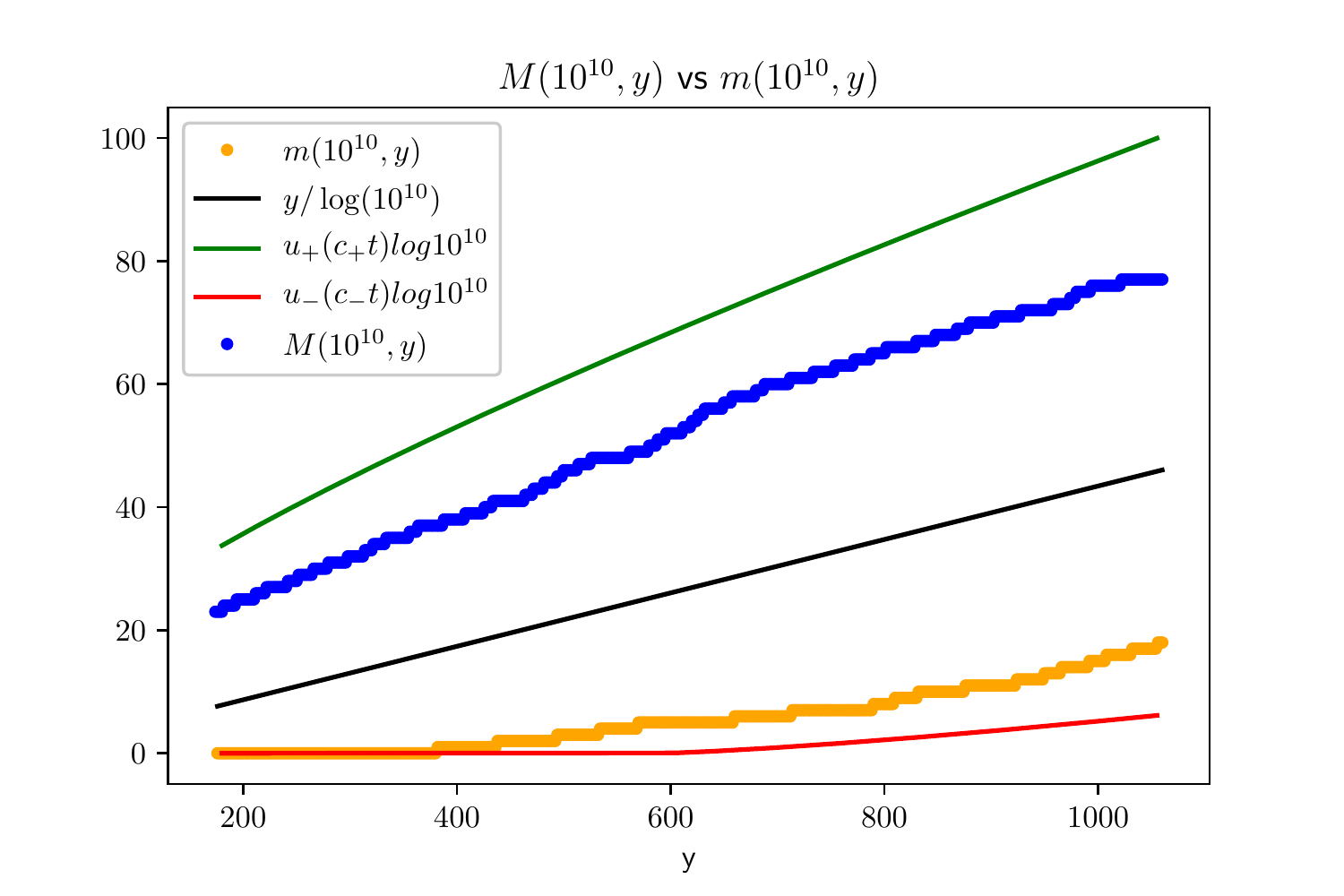}
\newline

\includegraphics[scale=.5]{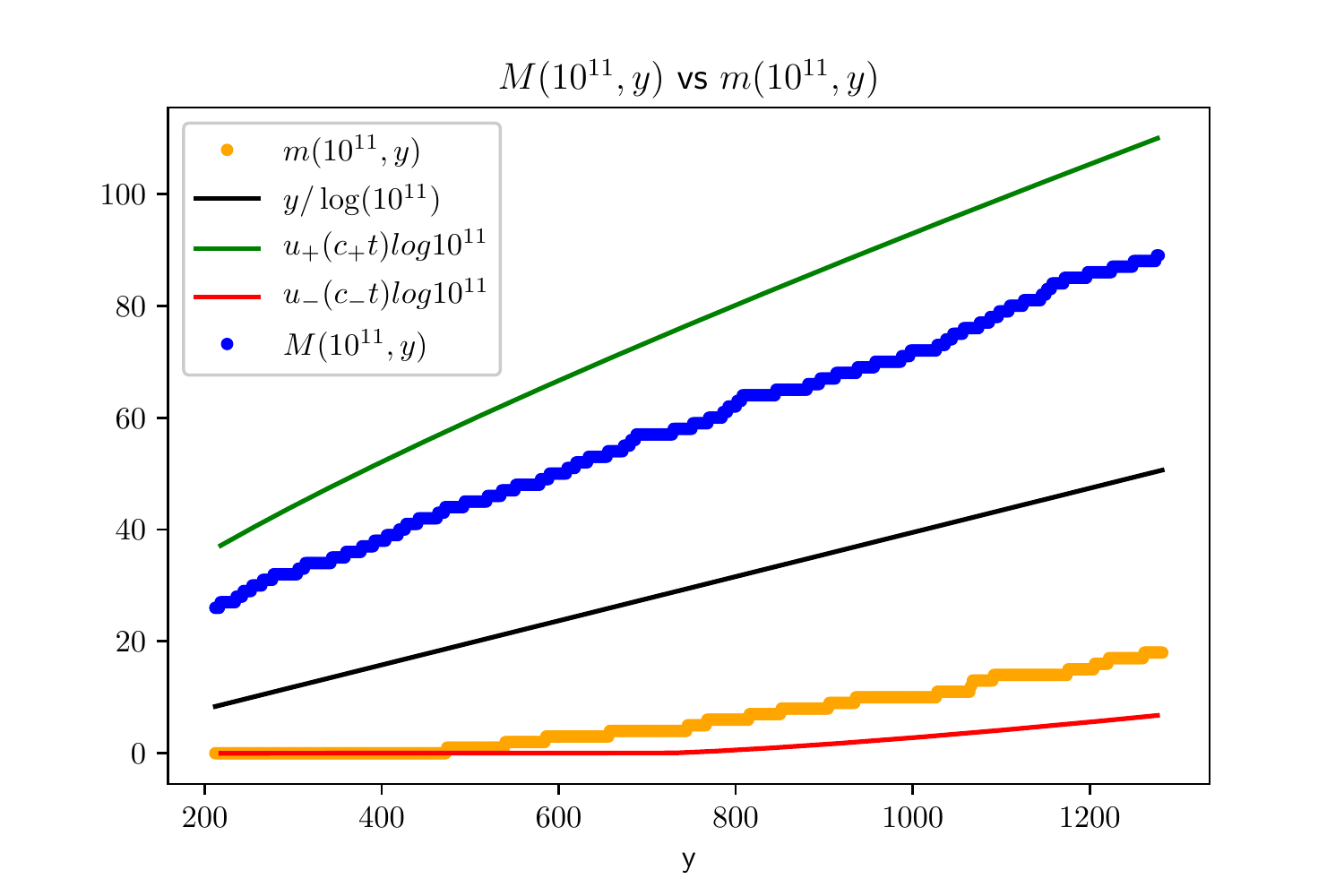}
\includegraphics[scale=.5]{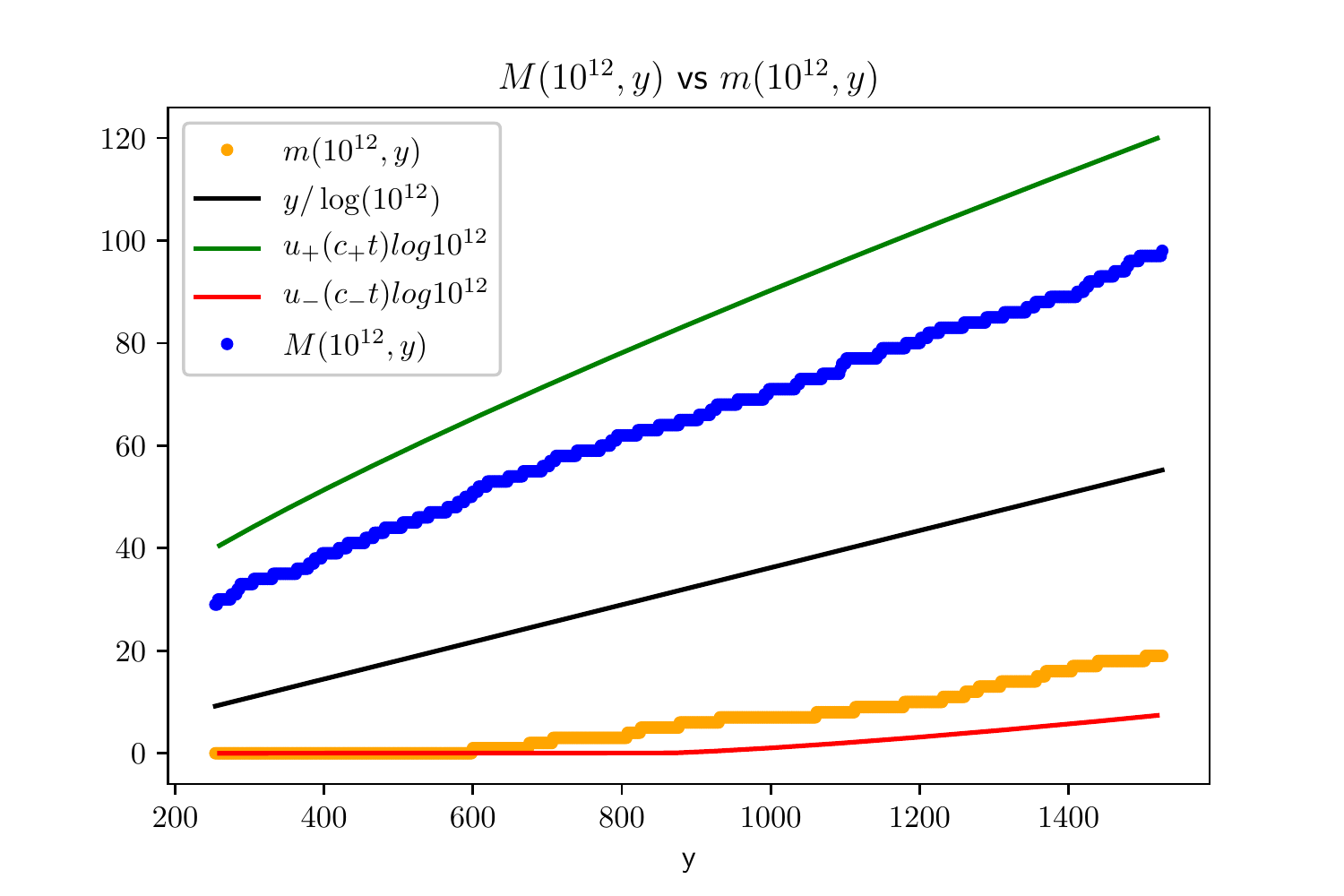}
 }
 \caption{ $u_-(c_-t)\log x$ vs $m(x,y)$ vs. $\tfrac y{\log x}$ vs. $M(x,y)$ vs $u_+(c_+t)\log x$
 \newline in ascending order,  where $y=t(\log x)^2$ for $x=10^k, k=9,\dots,12$  and
  \newline . \hskip1.6in $ \tfrac 13(\log x)^2 \leq   y\leq 2(\log x)^2$.}
\end{figure} 

We do not know what conclusions to draw from this data!

  
\subsection{Long intervals: $y/(\log x)^2\to \infty$}

We believe that there exist continuous functions $0<\sigma_-(A)<1<\sigma_+(A)$ for which
 $\sigma_-(A),\sigma_+(A)\to 1$ as $A\to \infty$, such that if $y/(\log x)^2\to \infty$ then
\begin{equation} \label{eq: logtotheAPredictions}
 m(x,y) \sim  \sigma_-(A) \frac y{\log x} \text{ and } M(x,y) \sim  \sigma_+(A) \frac y{\log x}
 \end{equation}
 writing $y=(\log x)^A$. Moreover we should take
 \[
 c_-=\sigma_-(2)  \text{ and } c_+=\sigma_+(2)
 \]
 above.
 We will obtain these conjectures from a discussion of sieve theory.  
 
 At first sight these conjectures seem  to be inconsistent with Selberg's result that 
 \[
 \pi(x+y)-\pi(x) \sim \frac y{\log x}
 \]
for almost all $x$, assuming that $y/(\log x)^2\to \infty$ (which he proved assuming the Riemann Hypothesis). However the 
``almost all'' in the statement allows for exceptions and in 1984, Maier \cite{Mai} exhibited, for all $A>2$,  constants $\delta_+(A) ,\delta_-(A)>0$ for which there is an infinite sequence of integers $x_+$ and $x_-$ with
\[
 m(x_-,y_-) \lesssim  \delta_-(A) \frac {y_-}{\log x_-} \text{ and } M(x_+,y_+) \gtrsim  \delta_+(A) \frac {y_+}{\log x_+} 
 \]
where $y_\pm=(\log x_\pm)^A$. As far as we know it could be that 
$ \sigma_-(A)=\delta_-(A) $ and $ \sigma_+(A)=\delta_+(A)$ for each $A$, as we will discuss in sections 2.2 and 3.

 \subsection{Another statistic}
 
 The data in sections 1.1 and 1.2 seem to support our conjectures for $M(x,y)$ in the range $y=o((\log x)^2)$, but the data in sections 1.3 and 1.4 for larger $y$ are less encouraging.  For this reason it seems appropriate to return to the question of how $M(x,y)$ grows as a function of $y$ in the range $y\asymp (\log x)^2$, and so we examine the ratio
 \[
r_+(x,y):= M(x,2y)/M(x,y).
 \]
 Our \emph{asymptotic predictions} suggest that this looks like $2+o(1)$ if $y\leq \frac 12 \log x$ and if $y/(\log x)^2\to \infty$, and $1+o(1)$ if
 $\log x\leq y=o((\log x)^2)$. For $y\asymp (\log x)^2$ our prediction for $M(x,y)$ is more complicated; indeed if $y=t(\log x)^2$ then we predict that this looks like
 \[
\rho_+(t):=  u_+(2c_+t)/u_+(c_+t)
 \]
and  we now compare this new statistic to the data:

\begin{figure}[H]\centering{
\includegraphics[scale=.5]{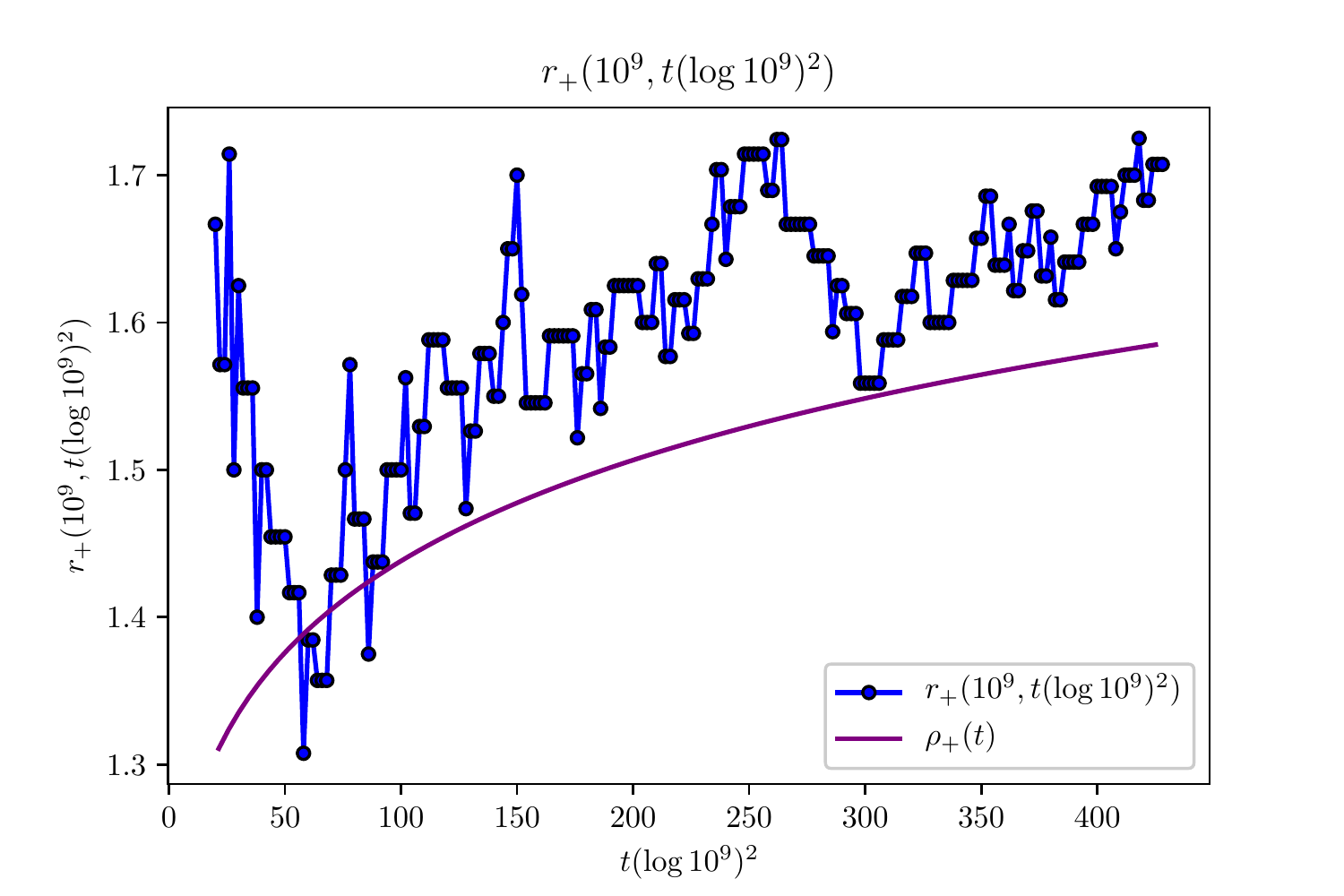}
\includegraphics[scale=.5]{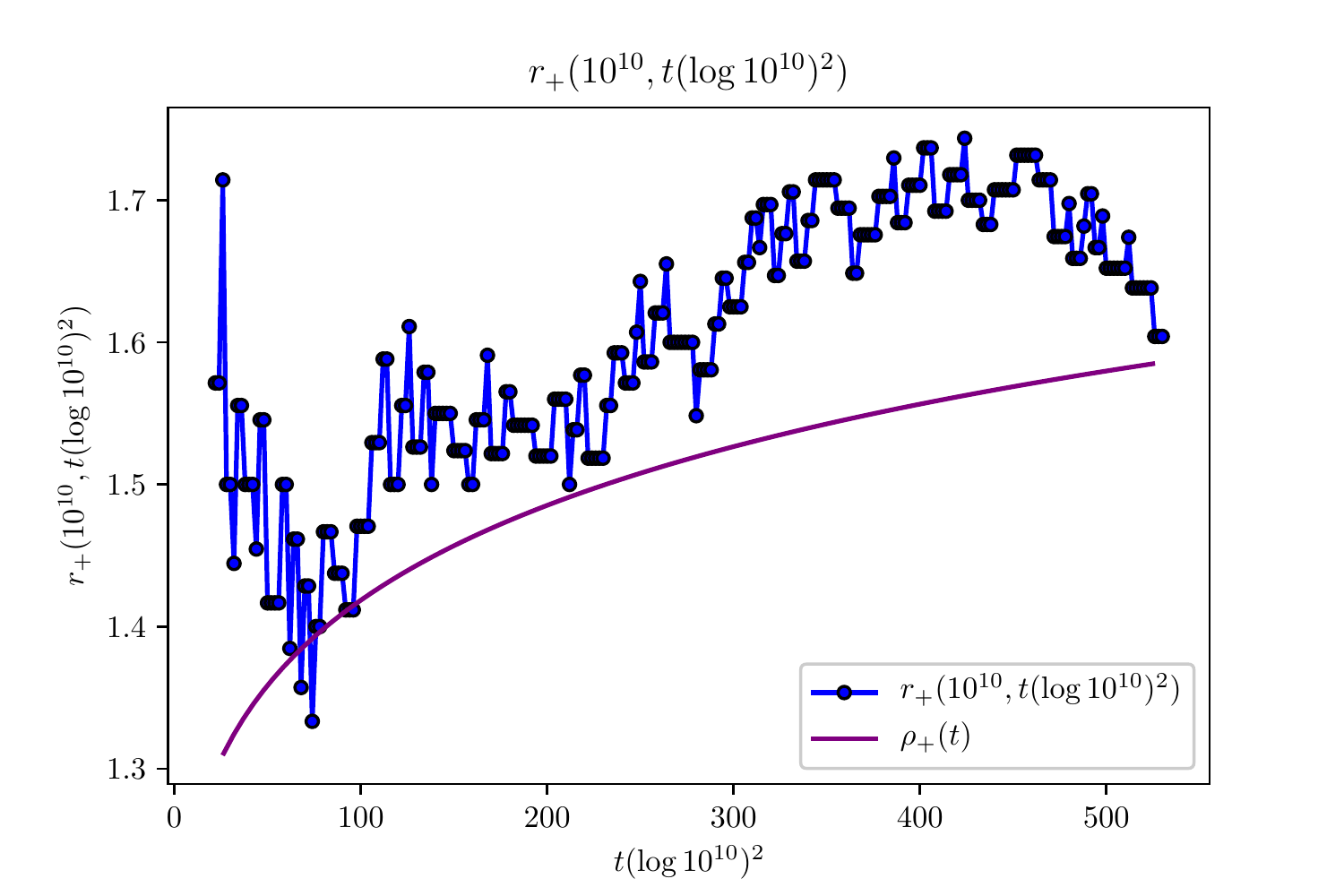}
\newline

\includegraphics[scale=.5]{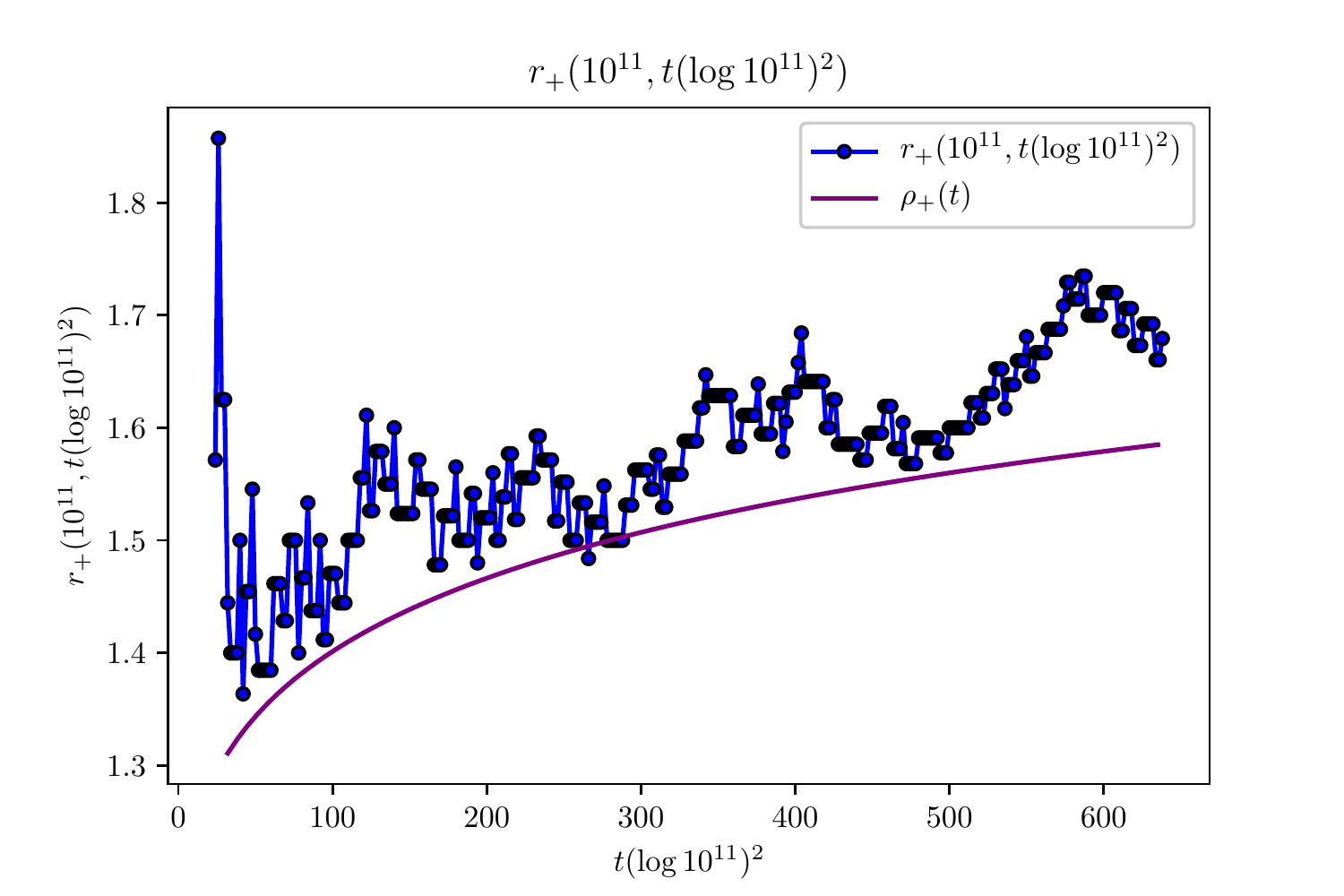}
\includegraphics[scale=.5]{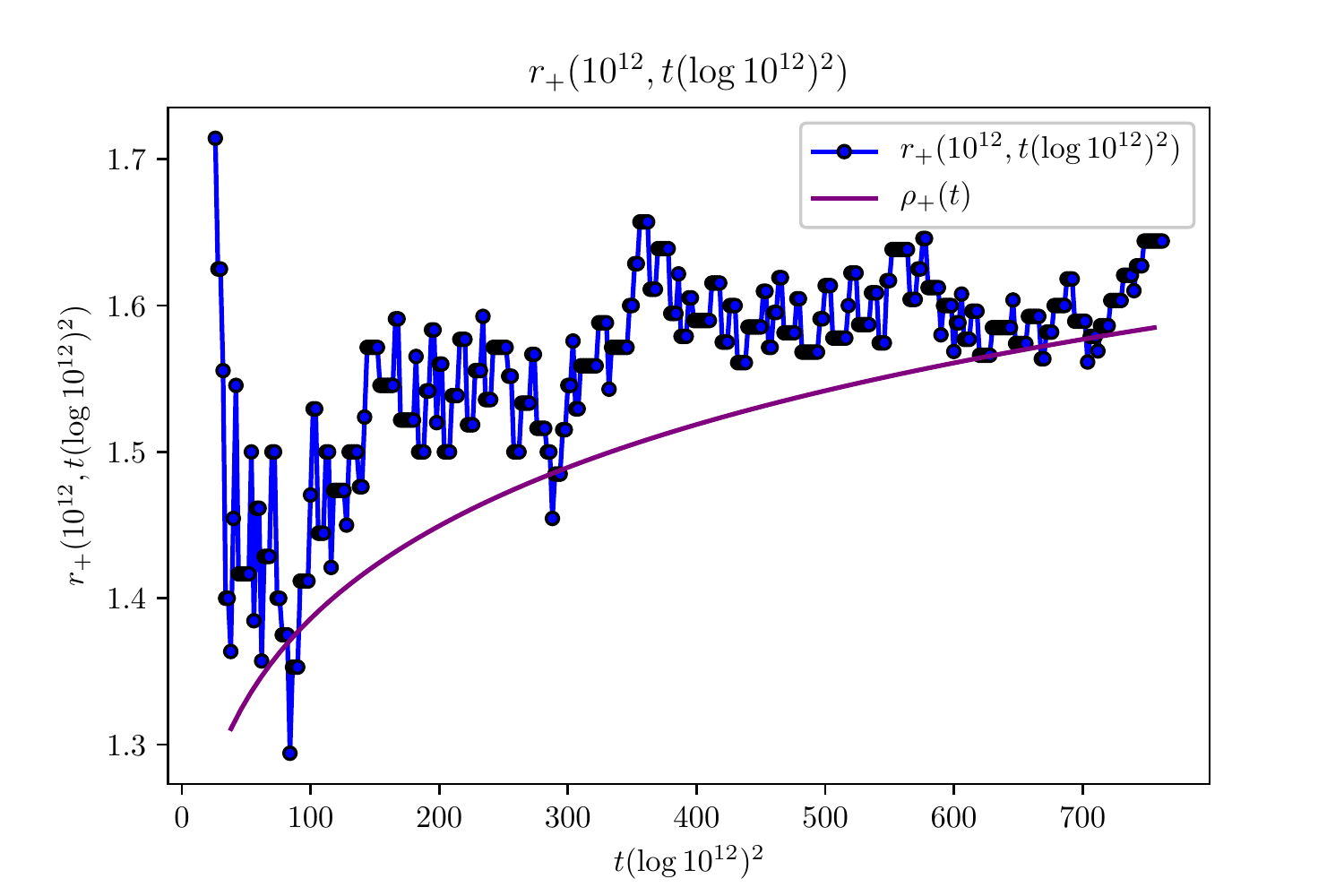} }
\caption{$M(10^k,2y)/M(10^k.y)$ for  $k=9,\dots,12$  and $ y\leq (\log (10^k))^2$.}
\end{figure} 


We can see the shape of our prediction looks correct but it is a little on the low side. What is encouraging is that the fit seems to get better as $k$ grows.

 \subsection{Summary of conjectures} We now recall in one place the conjectures given above:
 
 Fix $\epsilon>0$. If $x$ is sufficiently large and $y\leq (1-\epsilon)\log x$ then
 \[
 M(x,y)=S(y).
 \]
 A weaker conjecture claims if $y\leq (1-o(1))\log x$ and $y\to \infty$ as $x\to \infty$ then 
 \[
 M(x,y) \sim \frac y{\log y} .
 \]
If  $ \log x\leq y=o( (\log x)^2)$ then
\[
M(x,y) \sim L(x,y) \text{ where } L(x,y):=\frac {\log x}{\log \left(\tfrac{(\log x)^2} y\right) } .
\]
We conjecture that there exist constants $c_-,c_+>0$ such that if
 $y=t(\log x)^2$ then
\[
m(x,y)\sim u_-(c_-t) \log x \text{ and } M(x,y)\sim u_+(c_+t) \log x,
\]
and we even have tentative guesses about the values of $c_-$ and $c_+$. Moreover this suggests that 
\[
 \max_{x<p_n\leq 2x} p_{n+1}-p_n \sim c_-^{-1} ( \log x)^2.
 \]
Finally for any fixed $A>2$ we believe  that there exist continuous functions $ \sigma_-(A)<1<\sigma_+(A)$  such that if $y=(\log x)^A$  then
\[
 m(x,y) \sim  \sigma_-(A) \frac y{\log x} \text{ and } M(x,y) \sim  \sigma_+(A) \frac y{\log x}.
 \]
 
\section{Some historical comparisons}

\subsection{Best results known for small and large gaps between consecutive primes}
Following up the 2013 breakthrough by Yitang Zhang \cite{Zha} on small gaps between primes, Maynard \cite{May1} and Tao \cite{Tao}  proved that there are shortish intervals that contain $m$ primes for any fixed $m$. Their remarkable work implies that there exists a constant $c>0$ such that for each $y\geq 2$ we have
\[
M(x,y) \geq c \log y \text{ if } x \text{ is sufficiently large},
\]
which unfortunately is far smaller than what is conjectured here, in all ranges of $y$. However, before Zhang's work we could only say, for $y\ll \log x$, that $M(x,y)\geq 1$,  and after Zhang only that $M(x,y)\geq 2$, so these latest efforts are significant leap forward in our understanding.
\footnote{In \cite{May3} Maynard asks similar questions for integers that are the sum of two squares. He proved unconditionally the remarkably strong result that there are 
intervals $(X,X+y]$ which contain $\gg  \frac{y}{(\log x)^{1/2}} + y^{1/10}$ integers that are the sum of two squares for all $y\geq 1$.
 This is still  much smaller than what is probably the truth  for $y\ll (\log x)^c$  but it is at least  a power of $y$, as we might conjecture, so far closer to the truth than what is known unconditionally about primes.}

Similarly   Ford,   Green,  Konyagin,  Maynard and Tao \cite{FGM}, following up on \cite{FGK, May2}, recently showed that 
\[
m(x,y)=0 \text{ for some } y \gg \frac{\log x \log\log x \log\log\log\log x}{\log\log\log x} ,
\]
and they believe their technique (which consists of looking only at divisibility by small primes) can be extended no further than 
$y$ as large as $(\log x)(\log\log x)^{2+o(1)}$ which is far smaller than what is conjectured (here and previously).

\subsection{Unusual distribution of primes in intervals}
As discussed in section 1.5, Maier \cite{Mai} proved that there can be surprisingly few or many primes in an interval of length $(\log x)^A$ with $A>2$.  His proof can be easily modified to express his result in terms of certain sieving constants:
Define
\[
S(x,y,z) := \#\{ n\in (x,x+y]:\ (n,P(z))=1\}
\]
where $P(z):=\prod_{p\leq z} p$, and let  
\begin{align*}
S^+(y,z):= \max_{x}  S(x,y,z)
\text{ and } S^-(y,z): = \min_{x}  S(x,y,z).
\end{align*}
For each fixed $u\geq 1$ we define
\begin{align*}
\sigma_+(u) :&= \limsup_{z\to \infty} S^+(z^u,z)\bigg/  \bigg\{ \prod_{p\leq z}\bigg( 1-\frac 1p \bigg) \cdot  z^u \bigg\}  \\
\text{ and } \sigma_-(u) :&=  \liminf_{z\to \infty} S^-(z^u,z)\bigg/ \bigg\{  \prod_{p\leq z}\bigg( 1-\frac 1p \bigg) \cdot  z^u \bigg\}  .
\end{align*}
We will discuss what we know about the constants $\sigma_-(u)$ and $\sigma_+(u)$ in the next section, although we state here that we believe that both the limsup's and the liminf's are actually limits so that
\begin{equation} \label{eq: asymp for S's}
S^+(z^u,z) \sim \sigma_+(u)\prod_{p\leq z}\bigg( 1-\frac 1p \bigg) \cdot  z^u  \text{ and } 
S^-(z^u,z) \sim \sigma_-(u)\prod_{p\leq z}\bigg( 1-\frac 1p \bigg) \cdot  z^u.
\end{equation}

Maier's proof in \cite{Mai} can be modified to show that for $y=(\log x)^A$ and $z=\epsilon \log x$ we have 
\[
M(x,y) \geq \{ 1+o_{x\to \infty}(1)\}    S^+(y,z)  \cdot \frac {e^\gamma \log z}{\log x}
\]
which implies that there exist  arbitrarily large $x$ ($=x_+$) for which
\[
M(x,y)  \geq \{ 1+o(1)\}\sigma_+(A) \frac{y}{\log x} .
\] 
Analogously that there  are arbitrarily large $x$ ($=x_-$) for which
\[
m(x,y) \leq \{ 1+o(1)\} \sigma_-(A)  \frac{y}{\log x} .
\] 
If, as we believe, \eqref{eq: asymp for S's} holds then these estimates are  true for all $x$.
In \eqref{eq: logtotheAPredictions} we have conjectured that these bounds are ``best possible''; paraphrasing, we are postulating that Maier's observation about the effect of small prime factors is the key issue in estimating the extreme number of primes in intervals with lengths significantly longer than $(\log x)^2$. In fact our conjectures come from firstly sieving by small primes, and secondly looking at the tail probabilities of the binomial distribution that comes from a probabilistic model which takes account of divisibility by small primes.

  We will study in Appendix B how well some (relatively small) data for the full distribution compares to reality.


\section{Sieve methods and their limitations}

Let $\mathcal A$ be a set of integers  (of size $y$) to be \emph{sieved} (in our case the integers in the interval $(X,X+y]$), such that 
\[
\# \{ a\in \mathcal A:\ d|a\} = \frac{g(d)}d X + r(\mathcal A,d)
\]
where $g(d)$ is a multiplicative function, which is more-or-less $1$ on average over primes $p$  in short intervals (in our case each $g(p)=1$), and the error terms $r(\mathcal A,d)$ are small on average (in our case each   $|r(\mathcal A,d)|\leq 1$).  The goal in sieve theory is to give upper and lower bounds for
\[
S(\mathcal A,z) := \{ n\in \mathcal A:\ (n,P(z))=1\}.
\]
This equals $G(z)y$ ``on average'' where $G(z):= \prod_{p\leq z} (1 -\frac{g(p)}p)$.
In 1965, Jurkat and Richert \cite{JR} showed  that if $y=z^u$ then 
\begin{equation} \label{eq: JR theorem}
(f(u)+o(1))  \cdot G(z)y \leq S(\mathcal A,z)  \lesssim F(u) \cdot G(z)y    ,
\end{equation}
where $f(u)= e^\gamma(\omega(u) -\frac{\rho(u)}u)$ and $F(u)= e^\gamma(\omega(u) +\frac{\rho(u)}u)$, and
$\rho(u)$ and $\omega(u)$ are the Dickman-de Bruijn and Buchstab functions, respectively.
One can define these functions directly by
\[
f(u)=0 \text{ and }  F(u)=\frac{2e^\gamma}u \text{ for } 0<u\leq 2
\]
(in fact $F(u)=\frac{2e^\gamma}u$ also for $ 2<u\leq 3$) and
\[
 f(u)= \frac 1u \int_1^{u-1} F(t) dt   \text{ and }  F(u)=\frac {2e^\gamma} u  + \frac 1u \int_2^{u-1} f(t) dt   \text{ for all } u\geq 2.
\]
Iwaniec  \cite{Iw1} and Selberg \cite{Se1} showed that this result is ``best possible'' by noting that the sets
\[
\mathcal A^\pm =\{ n\leq x: \lambda(n)=\mp 1\}  
\]
 where $\lambda(n)$ is Liouville's function (so that $\lambda(\prod_p p^{e_p})= (-1)^{\sum_p e_p}$) satisfy the above hypotheses, with 
 \begin{equation} \label{eq: Selbergexample}
 S(\mathcal A^-,z) \sim f(u) \cdot G(z) \# \mathcal A^- \text{ and } S(\mathcal A^+,z) \sim F(u) \cdot G(z) \# \mathcal A^+.
 \end{equation}
 Since our question (bounding $S(x,y,z)$) is an example of this \emph{linear sieve} we deduce that 
 \[
 f(u) \leq \sigma_-(u) \leq 1 \leq \sigma_+(u) \leq F(u),
 \]
 and we expect that all of these inequalities are strict. However, in \cite{Gr3}, it is shown that if there are infinitely many ``Siegel zeros'',\footnote{That is, putative counterexamples to the Generalized Riemann Hypothesis, the most egregious that cannot be ruled out by current methods.} then,  in fact,
 \[
  \sigma_-(u) = f(u) \text{ and }  \sigma_+(u) =F(u) \text{ for all } u\geq 1.
 \]
 Given that eliminating Siegel zeros seems like an intractable problem for now, we are stuck. However in this paper we are allowed to guess at the truth, though we  know too  few interesting examples to even take an educated guess as to the true values of $\sigma_-(u) $ and $\sigma_+(u)$. It is useful   to note the following:

\begin{lemma}\label{lem: SieveBehaviour}
  $\sigma_+(u)$ is non-increasing, $\sigma_-(u)$ is non-decreasing,  and $\sigma_+(u), \sigma_-(u)\to 1$ as $u\to \infty$ 
\end{lemma}

\begin{proof} [Proof of Lemma \ref{lem: SieveBehaviour}]
Select $x$ so that $S(x,z^B,z)=S^+(z^B,z)$   is attained. For $A<B$, partition the interval $(x,x+z^B]$ into 
$z^{B-A}$ disjoint subintervals of length $z^A$, and select the subinterval with $ \#\{ n\in (X,X+z^A]:  (n,P(z))=1\}$ maximal. Therefore
\begin{align*}
S^+(z^A,z)&\geq  \max_{\substack{X=x+jz^A\\ 0\leq j\leq z^{B-A}-1}}   \#\{ n\in (X,X+z^A]:  (n,P(z))=1\} \\ &\geq 
\frac 1{z^{B-A}} \#\{ n\in (x,x+z^B]:  (n,P(z))=1\} = \frac {S^+(z^B,z)}{z^{B-A}},
\end{align*}
so that $\sigma_+(A)\geq \sigma_+(B)$. The analogous proof, with the inequalities reversed, yields the result for $\sigma_-$.

The fundamental lemma of the small sieve (see, eg, \cite{FI}) gives that 
\[
S(x,z^u,z) =\{ 1+O(u^{-u}) \} \prod_{p\leq z}\bigg( 1-\frac 1p \bigg) \cdot  z^u
\]
so that $\sigma_+(u), \sigma_-(u)=1+O(u^{-u})=1+o_{u\to \infty}(1)$.
\end{proof}

\subsection{Best bounds known}
In Maier's paper he used the  well-known fact that for all $u\geq 1$,
 \[
    \#\{  n\leq  z^u:\ (n,P(z))=1\} \sim  \omega(u) \frac{z^u}{\log z}  
 \]
where $\omega(u)$ is the \emph{Buchstab function}, defined by $\omega(u)=\frac 1u$ for $1\leq u\leq 2$, and
$(u\omega(u))'=\omega(u-1)$ for all $u\geq 2$.  By Lemma \ref{lem: SieveBehaviour} we have
\[
\sigma_+(A) =\max_{B\geq A}\sigma_+(B) \geq e^{\gamma} \max_{B\geq A} \omega(B), 
\]
 and, similarly, $\sigma_-(A)  \leq e^{\gamma} \min_{B\geq A} \omega(B)$.   
For all we know, it could be that  
\[
\sigma_+(A) = e^{\gamma} \max_{B\geq A} \omega(B)  .
\]
 That is, it could be that the most extreme example of sieving an interval, $S(x,z^A,z)$, occurs where $|x|< z^{O(1)}$, that is when $x$ is very small, but  there is little evidence that there are no other intervals with even more extreme behaviour. 
 
 In \cite{MS}, Maier and Stewart noted one could obtain smaller upper bounds for $\sigma_-(A)$ for small $A$. Their idea was to construct a sieve based on the ideas used to prove that there are long gaps between primes: Fix $2>u>1$. One first sieves the interval $[1,x]$ where $x=z^u$ with the primes in $(z^{1/v},z]$ where $1\leq v\leq \frac 1{u-1}$. The integers left are the $z^{1/v}$-smooth integers up to $x$, and the integers 
 of the form $mp\leq x$ for some prime $p\in (z,x]$ (note that $m\leq x/p<x/z=z^{u-1}\leq z^{1/v}$). The number of these is
 \[
 \psi(z^u,z^{1/v})  + \sum_{z<p\leq x} \left[ \frac xp \right] \lesssim x \rho(uv)+ x \sum_{z<p\leq z^u} \frac 1p \sim x(\rho(uv)+\log u).
 \]
 Next we sieve ``greedily'' with the primes $\leq z^{1/v}$ so that the number of integers left is
 \[
 \lesssim \prod_{p\leq z^{1/v}}\bigg( 1-\frac 1p \bigg)\cdot x(\rho(uv)+\log u) \sim  v(\rho(uv)+\log u)  \frac{e^{-\gamma} x}{\log z} 
 \]
We now select $v=v_u\in [1,\frac 1{u-1}]$ to minimize $r_u(v):=v(\rho(uv)+\log u)$. Since
\[ 
r_u(v)'=\rho(uv)+\log u +uv\rho'(uv)=\rho(uv)+\log u -\rho(uv-1) ,
\]
we select $v_u$ so that $r_u'(v_u)=0$. 
If $u=1+1/\Delta$ with $1/\Delta=o(1)$ then  
\[
v_u\sim \frac{\log \Delta}{\log\log\Delta}    \text{ and so }  r_u(v_u)\sim \frac{\log \Delta}{\Delta \log\log\Delta}. 
\]
On the other hand if we use the Buchstab function then we cannot obtain a constant smaller than $e^\gamma/2$.   Thus for $1\leq A\leq 2$, we  have
\[
\sigma_-(A) \leq \min\{  e^\gamma/2,  r_A(v_A)  \}
\]
In \cite{MS} this argument is extended to show that $r_A(v_A)$ is the minimum exactly when $1\leq A\leq 1.50046\dots$.
Unfortunately we are only really interested in $\sigma_-(A) $ for $A\geq 2$ in this article.

Now $\omega'(u)$ changes sign in every interval of length 1, so $\omega(u)$ has lots of minima and maxima, which occur whenever $\omega(u)=\omega(u-1)$ (since $u\omega'(u)=\omega(u-1)-\omega(u)$). Nonetheless its global minimum occurs at $u=2$  so that  $\sigma_-(2) \leq e^{\gamma} \omega(2)=\frac {e^\gamma}2$ (and we saw earlier that the linear sieve bounds give $\sigma_-(2)\geq 0$). We are most interested in $\sigma_+(2)$, which is bounded below by $e^{\gamma} \max_{B\geq 2} \omega(B)$.
This maximum occurs at $B\approx 2.75$ with $\omega(B)\approx 0.57$, so that $\sigma_+(2)\geq 1.015\dots $ (and we saw earlier that 
the linear sieve bounds give  $\sigma_+(2)\leq e^\gamma=1.78107\dots$)

In section 1.3 we have 
\[
c_+=\sigma_+(2)
\]
 and took this to be equal to $1.015\dots $ in our computations as this is the best lower bound known on $\sigma_+(2)$. Similarly in section 1.4 we have 
 \[
 c_-=\sigma_-(2)
 \]
  and took this to be equal to $\frac {e^\gamma}2 $ in our model as this is the best upper bound known on $\sigma_-(2)$.  It could be that these are equalities, but there is little evidence either way.

 \section{Very short intervals ($y\leq \log x$)}

If a set of integers $A$ is  inadmissible then  there exists a prime $p$ which divides   $n+a$ for some $a\in A$, for each integer $n$, and so obstructs these from all being simultaneously prime, once $n$ is sufficiently large. On the other hand,  Hardy and Littlewood's prime $k$-tuplets conjecture \cite{HL} states that if $A$ is an admissible set then there are infinitely many integers $n$ for which $n+a$ is prime for every $a\in A$, and this seems to be supported by an accumulation of evidence.

We are interested in $\pi(n,n+y]$,  the number of primes in intervals $(n,n+y]$ of length $y$ (with $y$ small compared to $n$), particularly the minimum,
$m(x,y)$, and the maximum, $M(x,y)$, as $n$ varies between $x$ and $2x$. If the primes in $(n,n+y]$ are $\{ n+a: a\in A\}$ with $n>y$, then $A$ is an admissible set, say of size $k$, and therefore 
\[
\pi(n,n+y]:= \pi(n+y)-\pi(n) = k \leq S(y),
\]
where $S(y)$ is the maximum size of an admissible set $A$ of length $y$. Moreover this implies that if the prime $k$-tuplets conjecture holds then
\[
\max_{n\geq y} \pi(n,n+y] = S(y).
\]
How large is $S(y)$?  One can show that the primes in $(y,2y]$ yield an admissible set and so $S(y)\gtrsim \frac y{\log y}$ (by the prime number theorem). It is believed that 
\[
S(y)\sim \frac y{\log y}
\]
but the best upper bound known is $S(y) \lesssim \frac {2y}{\log y}$ (by the upper bound in \eqref{eq: JR theorem}), and this upper bound seems unlikely to be significantly improved in the foreseable future (as we again run into the Siegel zero obstruction).
Calculations support the believed size of $S(y)$.  One interesting theorem, due to Hensley and Richards, is that if $y$ is sufficiently large then $S(y)>\pi(y)$ and so, if the prime $k$-tuplets conjecture is true then for all sufficiently large $y$ there exist infinitely many intervals of length $y$ that have more primes than the initial interval $[1,y]$.  The known values of $S(y)$ and bounds, can be found on \texttt{http://math.mit.edu/$\sim$primegaps/} and from there we see that $S(3432)\geq 481> \pi(3432)=480$. Therefore we believe that there are infinitely many intervals of length $3432$ containing exactly 481 primes, more than the
480 primes $\leq 3432$ found at ``the start''. However, finding such an interval (via methods based on this discussion) involves finding a prime $481$-tuple, which would be an extraordinary challenge unless one is very lucky.

So assuming the prime $k$-tuplets conjecture we know that $\max_{n\geq y} \pi(n,n+y] = S(y)$ for fixed $y$, and we might expect that $M(x,y) = S(y)$ for $y$ which (slowly) grows with $x$. In sections 4.1 and 8.1 we present two quite different heuristics to suggest that 
\begin{equation} \label{eq: VSmall.Prediction}
M(x,y) = S(y) \text{ for all } y\leq \{ 1-o(1)\} \log x;
\end{equation}
and we saw, in section 2.1, that this is well supported by the data that we have.

By a simple sieving argument Westzynthius showed in the 1930s that for any constant $C>0$ there exist intervals $[x,x+C\log x]$ which do not contain any primes. This argument is easily modified to show that for any $c>0$
\[
m(x,c\log x):=\min_{X\in (x,2x]} (\pi(X+c\log x)-\pi(X)) =0 \text{ if }  x \text{ is sufficiently large.}
\]

We will give two theoretical justifications for our prediction \eqref{eq: VSmall.Prediction}, supporting the conclusions we have drawn from the data represented in the graphs above. The first is explained in the next section and relies on guessing at what point a given admissible set yields roughly as many prime $k$-tuplets as conjectured. The second a more traditional approach is explained in section 8.1, developing the Gauss-Cram\'er heuristic (given in section 6) so that it takes account of divisibility by small primes.

\subsection{An explicit prime $k$-tuplets conjecture}

For a given admissible set of linear forms $b_jn+a_j,\ j=1,\dots, k$, Hardy and Littlewood \cite{HL} conjectured that 
\begin{equation} \label{eq: kTuplets}
\#\{ x<n\leq 2x:\ \text{ Each } b_jn+a_j \text{ is prime}\} \sim \prod_p \bigg( 1-\frac 1p \bigg)^{-k}  \bigg( 1-\frac {\omega(p)}p \bigg) \cdot \frac x{(\log x)^k} ,
\end{equation} 
where $\omega(p)$ is the number of $n\pmod p$ for which $p$ divides $\prod_{j=1}^k (b_jn+a_j)$.\footnote{Here \emph{admissible} can be defined to be those $k$-tuples for which every $\omega(p)<p$. A set $A$ is admissible if and only if the set $\{ n+a: a\in A\}$ of linear forms is admissible.} We wish to know for what $x$ are the two sides of \eqref{eq: kTuplets} equal up to a small factor, and for what $x$ can we obtain a good lower bound on the right-hand side.

This conjecture is known to be true as $x\to \infty$ for  $k=1$ (where we may assume that $1\leq a\leq b-1$). 
There is a lot of data on primes in arithmetic progression and these all suggest that \eqref{eq: kTuplets} holds uniformly for all $x\geq b^{\epsilon}$ for any fixed $\epsilon>0$.\footnote{Surprisingly there is no way known to try to prove this. The best we know how to obtain, assuming the Generalized Riemann Hypothesis, is that if $k=1$ then \eqref{eq: kTuplets} holds for all $x\geq b^{1+\epsilon}$, though this can be obtained ``on average'' unconditionally.   Linnik's Theorem implies that there exists a constant $\lambda$ such that one can obtain a lower bound on the left-hand side of \eqref{eq: kTuplets} once $x\gg b^\lambda$ (and so there is a prime $\ll b^{\lambda+1}$ in each reduced residue class mod $b$). In 2011, Xylouris \cite{Xyl} showed that we can take $\lambda=4$, the smallest $\lambda$ known-to-date}

Let $A$ be an admissible set of size $k=S(y)\sim \frac {\alpha y}{\log y}$ (where we believe $\alpha=1$), a subset of the positive integers $\leq y$. Since there are $\ll \frac y{(\log y)^2}$ integers in $S(y)$ that are $<\frac y{\log y}$ (by the sieve), we deduce that $Q:=\prod_{a\in A} a = e^{(\alpha+o(1))y}=k^{(1+o(1))k}$. 
Now $\omega(p)=k$ for all $p\geq y$ (since no two elements of $A$ can be in the same congruence class mod $p$), 
so that
\[
\prod_{p>y} \bigg( 1-\frac 1p \bigg)^{-k}   \bigg( 1-\frac {\omega(p)}p \bigg)  = \prod_{p>y} \bigg( 1-\frac 1p \bigg)^{-k}   \bigg( 1-\frac {k}p \bigg)  = e^{-o(k^2/y)}.
\]
Otherwise $1\leq \omega(p)\leq \min \{ k,p-1\} $ so that 
\[
e^{o(k)} = \bigg( \frac{\log 2y}{\log k}  \bigg)^k \gg \prod_{y\geq p>k} \bigg( 1-\frac 1p \bigg)^{-k}   \bigg( 1-\frac {\omega(p)}p \bigg)  \geq e^{-o(k)}.
\]
For the primes $\leq k$ we have $p-1\geq \omega(p)\geq 1$ and so
\[
1\geq \prod_{p\leq k}     \bigg( 1-\frac {\omega(p)}p \bigg)  \geq 1\bigg/ \prod_{p\leq k} p=e^{-k+o(k)} .
\]
Therefore, by Mertens' theorem, we have
\[
\prod_p \bigg( 1-\frac 1p \bigg)^{-k}  \bigg( 1-\frac {\omega(p)}p \bigg)  = (e^{O(1)} \log k)^{k}.
\]
So there exists a constant $C>0$ such that the right-hand side of \eqref{eq: kTuplets} is $\geq 1$ when
$(C\log k)^{k}x>(\log x)^k$. This certainly happens when $x=k^{ck}$ for any fixed $c>1$; that is, $x>Q^{1+\epsilon}$. One might guess that there is an error term in \eqref{eq: kTuplets} of size $x^{1/2+o(1)}$, in which case we must take $c>2$,  that is $x>Q^{2+\epsilon}$, to guarantee that the left-hand side of \eqref{eq: kTuplets} is positive.

Now if $\#\{ x<n\leq 2x:\ \text{ Each } n+a \text{ is prime, for each } a\in A\}\geq 1$ then $M(x,y)=S(y)$. From the above we might guess this holds when
$x>Q^{1+\epsilon}$ where $Q=e^{(1+o(1))y}$; that is, if $y\leq (1-o(1))\log x$. Indeed we only need the above heuristic discussion to be roughly correct ``on average'' over all such admissible sets, to support the conjecture in \eqref{eq: VSmall.Prediction}.\footnote{This reasoning suggests that even if we are  pessimistic then we would simply change the range in \eqref{eq: VSmall.Prediction} to $y\leq (c+o(1))\log x$ for some constant $c\in (0,1)$.}

\section{Cram\'er's heuristic} 

Gauss noted from calculations of the primes up to 3 million, that the density of primes at around $x$ is about $\frac 1{\log x}$. Cram\'er used this as his basis for a heuristic to make predictions about the distribution of primes: Consider an infinite sequence of independent random variables $(X_n)_{n\geq 3}$ for which
\[
\text{Prob}(X_n=1) = \frac 1{\log n} \text{ and } \text{Prob}(X_n=0) =1- \frac 1{\log n}.
\]
    By determining what properties are true with probability $1+o(1)$ for the sequence of $0$'s and $1$'s given by $X_3,X_4,\dots$, Cram\'er suggested   that such properties must also be true of the sequence $1,0,1,0,1,0,0,0,1,\dots$ of $0$'s and $1$'s which is characteristic of the odd prime numbers. For example, if $N$ is sufficiently large then
 \[
 S_N:=\sum_{n=3}^N X_n
 \]
 has mean $\int_2^N \frac{dt}{\log t}+O(1)$ and roughly the same variance, which suggests the conjecture that 
 $\pi(N)=\int_2^N \frac{dt}{\log t}+O(N^{1/2+o(1)})$; it is known that  this conjecture is equivalent to the Riemann Hypothesis. So for this particular statistic, Cram\'er's heuristic makes an important prediction and it can be applied to many other problems to make equally suggestive predictions.
 
  However Cram\'er's heuristic does have an obvious flaw: Since it treats all the random variables as independent, we have   $\text{Prob}(X_n=X_{n+1}=1)\approx \frac 1{(\log n)^2} $, so that 
 \[
 \ex \bigg( \sum_{n=3}^{N-1} X_nX_{n+1}  \bigg)= \int_2^N \frac{dt}{(\log t)^2}+O(N^{1/2+o(1)})
 \]
 with probability $1+o(1)$,
 which, Cram\'er's heuristic suggests, implies that there are infinitely many prime pairs $n,n+1$. But we have seen this is not so as $\{ 0,1\}$ is an inadmissible set.  More dramatically this heuristic would even suggest that $M(x,y)=y$ for all values of $y\leq \{ 1+o(1)\} \log x$. From the previous section we know  that this is false because $M(x,y)\leq S(y)$, as every $\pi(n,n+y]$ is restricted by those integers that are divisible by ``small''  primes, that is primes $\leq y^{1+o(1)}$.  This heuristic also suggests that the primes are equi-distributed amongst all of the  residue classes modulo a given integer $q$, rather than just the reduced classes.
 
 It therefore makes sense to modify Cram\'er's probabilistic model for the primes to take account of divisibility by ``small'' primes. The obvious way to proceed is to begin by sieving out the integers $n$ that are divisible by a prime $p\leq z$ (perhaps with $z=y$), and then to apply an appropriate modification of Cram\'er's model to the remaining integers, that is the integers that have no prime factor $\leq z$. The number of such integers up to $x$ is 
 \[
 \sim   \kappa x \text{ where } \kappa=\kappa(z):=\prod_{p\leq z} \bigg( 1 - \frac 1p \bigg) 
 \]
 if $z=x^{o(1)}$,
 and so the density of primes amongst such integers is $\frac 1{\kappa \log x}$.  We therefore proceed as follows:

 Define $P=P(z):=\prod_{p\leq z} p$ so that  $\kappa(z)=\frac{\phi(P)}{P}$. We consider
  an infinite sequence of independent random variables $(X_n)_{n\geq 3}$ for which $X_n=0$ if $(n,P)>1$; and 
\[
\text{Prob}(X_n=1) = \frac 1{\kappa \log n} \text{ and } \text{Prob}(X_n=0) =1- \frac 1{\kappa \log n} \text{ if }  (n,P)=1.
\]
With this model we can again accurately predict the prime number theorem (and the Riemann Hypothesis), as well as asymptotics for primes in arithmetic progressions,  for prime pairs, and even for admissible prime $k$-tuplets (with $k\leq z$). Moreover,   this will allow us to obtain our predictions  for maximal and minimal values of $\pi(x,x+y]$ (including the prediction for $y\ll \log x$ that we already deduced from assuming enough uniformity in the prime $k$-tuplets conjecture in section 4.1).

If $n\in (x,2x]$ with $(n,P)=1$  then  $\text{Prob}(X_n=1) = \frac 1{L}+O( \frac 1{L\log x})$ where $L:=\kappa \log x$, so for convenience we will work with a model where each $\text{Prob}(X_n=1) = \frac 1{L}$. There are, say, $N$ integers in $(X,X+y]$ that are coprime to $P$ where, a priori, $N$ could be any number between $0$ and $y$ (though we can refine that to $0\leq N\leq S^+(y,z) \ll \frac y{\log z}$ by the sieve). We now develop a model where $L$ and $N$ are fixed:

\section{The maxima and minima of a binomial distribution} \label{sec: maxmin}
Suppose that we have a sequence of  independent, identically distributed random variables $X_1,\dots,X_N$ with
 \[
 \mathbb{P}(X_n=1) = \frac 1L \text{ and }  \mathbb{P}(X_n=0) =1- \frac 1L,
 \]
 where $L$ is large. Let
 \[
 \mathbb Y:=\sum_{n\leq N} X_n.
 \]
 Then $ \mathbb Y$ is a \emph{binomially distributed random variable}, which is often denoted $B(N,\tfrac 1L)$.

 \begin{prop} \label{prop: k-values} Suppose that $N\ll L\log x$, and that $L\to \infty$ as $x\to \infty$.
If $k_-=k_-(N,L,x)$ is the largest integer for which
 \[
 \mathbb{P} (    \mathbb Y < k_- )  \leq \frac 1{x}
 \]
 then
 \[
 k_- = \begin{cases} \, \qquad 0 &\text{ if } N\leq  \{ 1+o(1)\} L\log x;\\
 \{ \delta_-(\lambda)+o(1)\} \tfrac NL &\text{ if } N=  \{ \lambda+o(1)\} L\log x \text{ with } \lambda>1;\\
 \end{cases}
 \]
 where $\delta_-=\delta_-(t)$ is the smallest positive solution to  $ \delta    (\log \delta-1)+1 =1/t$.
 
 If $k_+=k_+(N,L,x)$  is the smallest integer  for which
 \[
  \mathbb{P} (    \mathbb Y \geq k_+ )  \leq \frac 1{x}.
 \]
 then
 \[
 k_+= \begin{cases} \, \qquad N &\text{ if }     N\leq   \frac{\log x}{\log L} ;\\
     \{ 1+o(1)\}  \frac {\log x}{ \log \big(\tfrac{L\log x}{N}  \big) } & \text{ if }       \frac{\log x}{\log L} \leq  N=o(L\log x);\\
 \{ \delta_+(\lambda)+o(1)\} \tfrac NL &\text{ if } N=  \{ \lambda+o(1)\} L\log x \text{ with } \lambda>0;\\
 \end{cases}
 \]
 where $\delta_+=\delta_+(t)$ is the largest positive solution to  $ \delta    (\log \delta-1)+1 =1/t$.
 We observe that $k_-\leq k_+\ll \log x$ if $N\ll L\log x$.
   \end{prop}


\begin{proof} From the independent binomial distributions we deduce that if $0\leq k\leq N$ then
\[
\mathbb{P}(\mathbb Y=k)  = \mathbb{P}\bigg(   \sum_{n\leq N} X_n=k\bigg)  = \binom Nk \bigg(  \frac 1L\bigg)^k \bigg( 1- \frac 1L \bigg)^{N-k}.
\]
 
 Therefore $\mathbb{P}(\mathbb Y=N) =1/L^N$ and this is $>1/x$ provided $N\leq \tfrac{\log x}{\log L}$. 
 
 Also $\mathbb{P}(\mathbb Y=0) =   ( 1- \frac 1L )^{N}=e^{-  N/L+O(N/L^2)}$ 
which is   $>\frac 1x$ for $N\leq \{ L+O(1)\} \log x$.\footnote{To be more precise we obtain $N\leq \frac{\log x}{-\log(1-\frac 1L)} = (L-\tfrac 12-\tfrac 1{12L}+O(\tfrac 1{L^2})) \log x$.}
 
We now estimate the terms in our formula for $\mathbb{P}(\mathbb Y=k)$:
 \begin{align*}
\binom Nk &= \frac{N^k}{k!}  \prod_{i=0}^{k-1} \bigg( 1-\frac iN\bigg) =  \frac{N^k}{(k/e)^k} k^{O(1)} \exp \bigg( \sum_{i=0}^{k-1} O\bigg( \frac iN \bigg) \bigg)  \\
& =  \frac{N^k}{(k/e)^k}   \exp \bigg(  O\bigg( \frac {k^2}{N} +\log k\bigg) \bigg) .
 \end{align*}
by Stirling's formula. We also have $( 1- \frac 1L )^{N-k}=\exp( -\frac NL +O(\frac kL+\frac N{L^2}))$, and so
 \[
\mathbb{P}(\mathbb Y=k)    =   \bigg(  \frac {eN}{kL} \bigg)^k   \exp\bigg(- \frac NL +O\bigg( \frac{k^2}N + \log k+ \frac kL + \frac N{L^2}  \bigg)\bigg)
\]

Therefore if $N=o(L\log x)$ and $k=o(\log x)$ then  $k^2/N\leq k=o(\log x)$   so that
\[
\mathbb{P}(\mathbb Y=k)     =  \bigg(  \frac {eN}{kL}\bigg)^k   x^{o(1)},
\]
and this equals $x^{-1+o(1)}$ if and only if
 \[
k \sim \frac {\log x}{ \log (\tfrac{L\log x}{N}  ) }
\]

Finally we deal with the range $N=\lambda L\log x$ with $\lambda>0$. If $k=\delta \lambda \log x$ with $\delta>0$  then, by the above estimate, 
\[
\mathbb{P}(\mathbb Y=k)=   \bigg(  \frac{e\lambda \log x}k \bigg)^k   \exp\bigg( -\lambda \log x +O\bigg( \frac{\log x}L    \bigg)\bigg) = 1/x^{\lambda(1-\delta    \log(e/\delta)) +o(1)},
\]
which equals $1/x^{1+o(1)}$ if  $\delta=\delta_{\pm}(\lambda)$ so that $\lambda(1-\delta \   \log(e/\delta) )=1$. 
\end{proof}

\begin{remark} There are well-known bounds on the tail of the binomial distribution (see, e.g., \cite{Fell}) which can be used to obtain this last result:
 \[
\frac 1{\sqrt{8k(1-\tfrac kN)}}\exp\bigg(  -N\,   \mathbb  D \bigg(  \frac kN \bigg|   \frac 1L      \bigg)   \bigg)  \leq  
\begin{cases} \mathbb{P} (    \mathbb Y \leq k)  &\text{ if } k\leq \frac NL\\
\mathbb{P} (    \mathbb Y \geq k)  &\text{ if } k\geq \frac NL
\end{cases}
\leq \exp\bigg(  -N\,   \mathbb  D \bigg(  \frac kN \bigg|   \frac 1L      \bigg)   \bigg) 
 \]
 where 
 \[
 \mathbb D(a|p):= a \log \frac ap + (1-a) \log \frac {1-a}{1-p}
 \]
which is called  the \emph{relative entropy} in some circles  (this clean upper bound can be obtained by an application of Hoeffding's inequality); the two cases are equivalent since if $k\geq \frac NL$ then 
$\mathbb D(1-a|1-p)=\mathbb D(a|p)$. Using these inequalities we would  determine $\delta=\delta(t,L)$ from the functional equation
\[
L\, \mathbb  D \bigg(  \frac \delta L \bigg|   \frac 1L      \bigg)     = \frac 1 t\bigg( 1 + O \bigg( \frac {\log\log x}{ \log x} \bigg) \bigg) ,
\]
which is slightly different, but yields $\delta(t,L)=\delta(t)+O(\tfrac 1{\log \delta(t)} ( \tfrac 1L+ \tfrac {\log\log x}{\log x}))$, a negligible difference in the ranges we are concerned about.
\end{remark}

\section{Asymptotics} \label{sec: uandv}

In section 1.3 we used the solutions $u=u_-\in (0,t)$ and $u=u_+\in (t,\infty)$ to 
\[
u(\log u -\log t -1)+t=1
\]
where $u(t)=t\delta(t)$, and $\delta=\delta_-\in (0,1)$ and $\delta=\delta_+\in (1,\infty)$ are the solutions to 
\[
f(\delta):=1-\delta \log(e/\delta)=\frac 1t.
\]
To verify these claims, we note that $f(0)=1, f(1)=0$ and $f(\infty)=\infty$
We have  $\frac{df}{d\delta}=   \log \delta $ so $f$ (as a function of $\delta$) has its minimum $f(1)=0$ with $f''(\delta)>0$ for all $\delta>0$. 
Therefore there exists a unique $\delta_-\in (0,1)$ with $f(\delta_-) =1/t$
for all $t>1$ and no such $\delta_-$ otherwise. Moreover $\delta_-(t)$ is an increasing function with limit $1$.
Also, there exists a unique $\delta_+>1$ with $f(\delta_+) =1/t$ for all $t>0$.
Moreover $\delta_+(t)$ is a decreasing function with limit $1$.

  We will now show that $ u_+(t)$ is increasing in $t>0$ and  $ u_-(t)$ is increasing in $t\geq 1$
  Differentiating $f(\delta)=\frac 1t$ we obtain $\log \delta \cdot \frac{d\delta}{dt} = - \frac 1{t^2}$.  Therefore 
 \[
 \frac d{dt} \log u(t)=\frac d{dt} \log t\delta=   \frac{1}{\delta} \frac{d\delta}{dt} +\frac 1t = \frac 1t - \frac 1{t^2\delta \log \delta}
 =  \frac {t\delta \log \delta-1} {t^2\delta \log \delta}= \frac {\delta-1} {t\delta\log \delta} >0
 \]
 for all $\delta>0$.

We can be more precise about the limits:

\subsection{Estimates as $t\to \infty$} Write $\delta= 1+\theta$ so that 
\[
1-1/t=(1+\theta)(1-\log (1+\theta)) =1-\frac{\theta^2}2+\frac{\theta^3}6-\frac{\theta^4}{12}+\dots
\]
Therefore $\theta=\pm \frac{2^{1/2}}{t^{1/2}} + \frac 1{3t}\pm \frac{1}{ 9(2t)^{3/2}} +O(\frac 1{t^2})$ as $t\to \infty$,
so that 
 \begin{align*}
u_+(t)=t \delta_+(t)&=t+ (2 t)^{1/2} + \frac 1{3 }+ \frac{1}{ 9\cdot 2^{3/2}t^{1/2}} +O(\frac 1{t}) \\
 u_-(t)=t  \delta_-(t)&=t- (2 t)^{1/2} + \frac 1{3 }- \frac{1}{ 9\cdot 2^{3/2}t^{1/2}} +O(\frac 1{t}),
 \end{align*}
 for large $t$. So if $t$ is large and $N=t L \log x$ then, in Proposition \ref{prop: k-values},
 \[
 k_\pm = \bigg(t \pm (2 t)^{1/2} + \frac 1{3 }- O\bigg( \frac 1{t^{1/2}}  \bigg)  \bigg)   \log x  \text{ as } t\to \infty.
 \]

\subsection{Approximating the normal distribution}
A random variable given as the sum of enough independent binomial distributions tends to look like the normal distribution, at least at the center of the distribution. However since we are looking here at tail probabilities, the explicit meaning of ``enough'' is larger than we are used to. To be specific,  $\mathbb Y$ has mean $\mu:=\tfrac NL$ and variance $\sigma^2=\tfrac NL(1-\tfrac 1L)$, and we  expect $\mathbb Y$ will eventually be  normally distributed with these parameters. If so, then
\[
\mathbb{P}(\mathbb Y<\mu-\tau\sigma), \mathbb{P}(\mathbb Y>\mu+\tau\sigma) \approx \frac 1{\sqrt{2\pi}} \int_{\tau}^\infty e^{-t^2/2} dt\sim 
\frac{e^{-\tau^2/2}}{\tau \sqrt{2\pi}} 
\]
and if this is $\approx 1/x$ then $\tau\sim \sqrt{2\log x}$. Therefore $\tau\sigma \sim ( 2\tfrac NL \log x)^{1/2}$.
Writing $N=\lambda L\log x$ we have $\tau\sigma \sim ( 2\lambda )^{1/2}\log x$. Therefore we might expect
the maximum and minimum values of $\mathbb Y$ to be $(\lambda\pm (2 \lambda)^{1/2} +o(1))\log x$. We see from section 7.1 that this is correct as $\lambda\to\infty$ (but not for fixed $\lambda$).

We can see this issue more simply:  If $k=\kappa N/L$ with $\kappa>1$ then  the binomial distribution gives
\[
\text{Prob} (\mathbb Y\geq k) \asymp \bigg(1-\frac 1L\bigg)^N \binom Nk \frac 1{(L-1)^k} 
=\exp\bigg( -\frac NL (\kappa(\log \kappa-1)+1+o(1))\bigg)
\]
and the normal distribution (with the same mean and variance) gives
\[
\text{Prob} (\mathbb Y\geq k)=\exp\bigg( -\frac NL ( \tfrac 12(\kappa-1)^2+o(1))\bigg)
\]
and the main terms here are only the same when $\kappa\to 1^+$.

\subsection{Estimates as $t\to 0^+$}

 In the other direction we obtain estimates for $\delta_\pm(t)$  as $t$ gets smaller.
 
If $t\to 0^+$ then we deduce from $\delta_+ ( \log \delta_+-1)+1=1/t$ that 
\begin{equation} \label{eq: delta_+estimate}
\delta_+(t) = \frac{1/t}{ \log \big(   \frac{1/t}{e\log 1/t} \big) }  \bigg( 1 + O \bigg(  \frac{\log\log 1/t}{( \log 1/t)^2}  \bigg) \bigg)
 \end{equation}
 so that 
 \[
 u_+(t)=t \delta_+(t)= \frac{1}{ \log ( 1/t) }  \bigg( 1 + O \bigg(  \frac{\log\log 1/t}{ \log 1/t}  \bigg) \bigg)
 \]
and therefore
\[
k_+ \sim u_+(t) \log x \sim \frac{\log x}{ \log(1/t) }   \text{ as } t\to 0^+.
\]
Combining this with the second estimate for $k_+$ in Proposition 1, we deduce that $k_+(N)$ is a continuous function in $N$ in the range of Proposition 1.

If $t\to 1^+$ then writing $t=1+\eta$ with $\eta>0$ small and $\delta_-=1/B$, we deduce from $\delta_- ( 1-\log \delta_-)+1=1/t$ that $\frac{1+\log B}B=\eta+O(\eta^2)$ and so
\[
1/\delta_-=B = (1/\eta) \log (1/\eta)  \bigg( 1 + O \bigg(  \frac{\log\log 1/\eta}{\log 1/\eta}  \bigg) \bigg) .
 \]
 This implies that 
  \[
 u_-(t)=t \delta_-(t)=  \frac{\eta}{ \log (1/\eta) } \bigg( 1 + O \bigg(  \frac{\log\log 1/\eta}{\log 1/\eta}  \bigg) \bigg) 
 \]
and therefore
\[
k_- \sim u_-(t) \log x \sim   \frac{(t-1)\log x}{\log (\tfrac 1{t-1})}  \text{ as } t\to 1^+,
\]
which $\to 0$ as $ t\to 1^+$.  This suggests that  $k_-=0$ for $N<\{ 1-o(1)\} L\log x$ but grows like
\[
  \frac {N-L\log x} {L \log \tfrac{N} {N-L\log x}}
\]
for a small range near $L\log x$ which we denote by $L\log x<N<\{ 1+o(1)\} L\log x$.

\section{Applying the modified Cram\'er heuristic}

Here is the general set-up. For some $z\leq y$ define $P=P(z):=\prod_{p\leq z} p$ so that $P(z)=e^{(1+o(1))z}$ by the prime number theorem.
For $S(x,y,z) := \#\{ n\in (x,x+y]:\ (n,P(z))=1\}$ (as in section 2.2) we define
\[
 I(N)=\{ X\in (x,2x]: S(X,y,z) =N\} .
 \]
for each integer $N$ in the range $0\leq N\leq S^+(y,z)$.  Our heuristic is that the values
\[
\pi(X,X+y] \text{ for } X\in  I(N),
\]
are distributed like the binomially distributed random variable 
\[
B(N,\tfrac 1L) \text{ where } L=\frac {\phi(P)}P \log x.
\]
We therefore use  Proposition 1 (with $x$ there equal to $\# I(N)$)  to predict the value of
\[
M_N(x,y) := \max_{X\in I(N)}  \pi(X,X+y] 
\]
for each $N$ with $I(N)$ non-empty.  From these predictions we obtain our  predictions for
 \[
 M(x,y) = \max_N  M_N(x,y).
 \]

One can work out the details of this heuristic to make precise conjectures provided we can get a good estimate for $\log \# I(N)$.
This is not difficult when $z\leq \epsilon \log x$: For each $m,0\leq m\leq P-1$ we have 
\[
S(X,y,z)=S(m,y,z) \text{ whenever } X\equiv m \pmod {P(z)},
\]
since $(X+j,P)=(m+j,P)$ for all $j$.  Moreover these intervals $(X,X+y]$ with $X\equiv m \pmod {P(z)}$ are all disjoint so can be considered to be independent.
Therefore if $N=S(m,y,z) $ then $P=P(z)\leq x^{\epsilon+o(1)}$ and so
\[
\#I(N)\geq \#\{ X\in (x,2x]: X\equiv m \pmod {P(z)}\} = x/P+O(1) \geq x^{1-\epsilon+o(1)}.
\]
Hence, when $z$ is this small, the answer given by our heuristic depends only on the extreme values, $S^-(y,z)$ and $S^+(y,z)$.

Getting a good estimate for $\log \# I(N)$ is not straightforward if $z$ (and therefore $y$) is significantly larger than $\log x$. However one expects our heuristic   to be more accurate the larger $z$ is, so we have to find the right balance in our selection of $z$.

\subsection{Very short intervals ($y\ll \log x$)}\label{applyingmodel-veryshort}

If $y\leq \eta \log x$ with $0<\eta<\tfrac 12$ small, then the above discussion suggests taking $z=y$.
Hence $S^+(y,z)=S^+(y,y)=S(y)$. For each $m\pmod P$ we apply Proposition 1 with 
\[
N=S(m,y,y),\  L=\frac {\phi(P)}P \log x, \text{ and } x \text{ replaced by } x^{1-\eta}.
\]
For given $L$ and $x$,
one obtains the largest value of $k_+$ in Proposition 1, when   $N$ is as large as possible. This happens here when $N=S(y)$, which we believe is $\sim \frac{y}{\log y}$ and know is no more than twice this.  Now $ L\asymp \frac{\log x}{ \log y}$ and Proposition 1 then implies that $k_+=N=S(y)$ as long as
$S(y) \leq (1-\eta+o(1)) \frac{\log x}{\log L}$, which should be true for any fixed $\eta < \tfrac 12$ (and at worst for $\eta < \tfrac 13$).

This supports the conjecture \eqref{eq: VSmall.Prediction} in a range   like $y\leq (\tfrac 12-o(1))\log x$. What about for larger $y$?

\subsection{Larger $y$ with a different choice of intervals}

For larger $y$, say $  \log x\ll y <(\log x)^A$ with $A>2$, we need to decide how to select our value for $z$.
One might guess that the right way to do so is   to take   $z=y$.\footnote{We do not wish to sieve with  primes  larger than the length of the interval, since any larger primes cannot divide more than one element in an interval of length $y$, so cannot be helpful in a sieve argument.}
That is, to sieve the intervals of length $y$ with all of the primes $\leq z=y$, and then apply the modified Cram\`er model.  
In this case the sets $\{ j\in [1,y]: (X+j,P)=1\}$
are probably different for every $X\in (x,2x]$ (certainly they do not repeat periodically as in the earlier subsection), which seems difficult   to cope with.  However we do not need to understand these sets so precisely, we only need to understand their size, that is, to have good estimates for $\log \# I(N)$ for each $N$, but even this seems to be out of reach. Therefore this is the less desirable option (though we work through some of the details in Appendix C). In general, we do not know how to get good estimates for $\log \# I(N)$ whenever $z$ is substantially larger than $\log x$. 

These (for now insurmountable) issues, suggest that we should proceed as before, with a smallish value of $z$, so as to recover the sieved sets repeating predictably. Therefore we pre-sieve the intervals of length $y$ with all of the primes $\leq z:=\epsilon \log x$, and then apply the modified Cram\`er model.  
There might be a substantial difference when sieving with the primes $\leq z$, as opposed to $y$, though we hope not.  If there is a substantial difference then this needs further investigation.

\subsection{Larger $y$; Predictions by pre-sieving up to $z=o( \log x)$}

We pre-sieve with the primes up to $z=\epsilon \log x$ where $\epsilon\to 0$ very slowly as $x\to \infty$. 
In this case we have seen that we may cut to the chase by taking
\[
N_+=S^+(y,z)=:e^{-\gamma} \frac{y}{\log\ z}c_+ \text{ and } L=\frac{\phi(P)}P \log x\sim e^{-\gamma} \frac{\log x}{\log\log x}
\]

  \medskip
  
\noindent \textbf{Prediction:\ Pre-sieving up to $z=\epsilon \log x$}:  
\textsl{ If $\log x\ll y= o( (\log x)^2)$ then
\[
M(x,y) = \min\bigg\{ S^+(y,z),   \{ 1+o(1)\}  \frac {\log x}{ \log \big(\tfrac{(\log x)^2}{y}  \big) }\bigg\} .
\]
If $y=\lambda (\log x)^2$ with $\lambda>0$ then 
 \[
 M(x,y) \sim u_+(\lambda c_+)   {\log x} \text{ and }  m(x,y) \sim u_-(\lambda c_-)   {\log x} .
 \]
 }
 \smallskip
 
 If $y\asymp \log x$ then this might predict that $M(x,y)=S^+(y,z)>S(y)$ which is obviously false (though not by much) --   in this range it therefore makes sense to sieve up to $z=y$, which  will assure the feasible prediction $M(x,y)=S(y)$ (as we work out in Appendix C).

   If $\lambda$ is large and $y=\lambda (\log x)^2$ then
 \[
 u_+(\lambda c_+)   = \lambda c_+ + \sqrt{ 2\lambda c_+ } + O(1) ,
 \]
 and so $M(x,\lambda (\log x)^2)\sim c_+ \frac y{\log x}$ as $\lambda\to \infty$; and analogously 
$m(x,\lambda (\log x)^2)\sim c_- \frac y{\log x}$.

 \begin{proof}[Deduction from the predictions of Proposition 1] We apply Proposition 1 to predict, for each $ 0\leq j\leq P-1$ where $P=P(z)$,
\[
M_j(x,y):=\max_{ \substack{X\in (x,2x]\\ X\equiv j \pmod P}}   \pi(X+y)-\pi(X)
\]
and then we guess that $M(x,y)=\max_j M_j(x,y)$. We observe that 
  \[
   \#  \{ X\in (x,2x]: \ X\equiv j \pmod P\}= \frac xP+O(1)=x^{1-o(1)}\]
    for each $j$, so we apply Proposition 1 to a set of this size, and the result follows directly.  The analogous proof works for $m(x,y)$.
   \end{proof}

 \section{Which choices should we make?}
 
 We will now distill these discussions, which each yield slightly different predictions.
 
 \subsection{Very short intervals ($y\ll \log x$)}
 In section 1.1, we predicted that if $y\leq c \log x$ then $M(x,y)=S(y)$. This was confirmed by one heuristic in section 4.1,
and by a very different heuristic in section 8.1, giving us some confidence in this conclusion.
 
 From all three discussions it is not obvious what explicit constant one should take in place of the inexplicit ``$c$''.
  Our guess is that for any $\epsilon>0$ one has
 \[
 M(x,y)=S(y) \text{ for } y\leq (1-\epsilon) \log x,
 \]
if $x$ is sufficiently large, as well
 \[
 M(x,y)\sim \frac{\log x}{\log\log x} \text{ for }  (1-\epsilon) \log x\leq y\leq (1+o(1)) \log x.
 \]
 The ``$o(1)$'' is inexplicit and our methods do not  pinpoint the transition more accurately. The data represented in figure 1 appear to more-or-less confirm these predictions.  However  these small $x$-values do suggest that $c>1$ which we do not believe, since that would force contradictions to our predictions for $M(x,y)$ for larger $y$.   
 
  \subsection{Intermediate length intervals ($\log x\leq y=o( (\log x)^2)$)} 
 In the range
 $\log x\leq y= o( (\log x)^2)$ we have predicted \eqref{eq: Intermediate intervals} no matter whether we presieve up to $z$ or up to $y$.

 One can revisit the heuristic arguments above to try to get a more accurate approximation: By \eqref{eq: delta_+estimate} we believe that if $y=\lambda (\log x)^2$ with $\lambda\to 0$ then
 \[
 M(x,y) \text{ is better approximated by }   \frac{\log x}{ \log \big(   \frac{1/\lambda}{e\log 1/\lambda} \big) }  .
  \]
However the data for this prediction is no more compelling then for the less precise prediction $L(x,y)$  in this range, presumably because $x$ is so small.

\subsection{Comparatively long intervals ($y/(\log x)^2\to \infty$ with $y\leq x$)}

Here we write $y=(\log x)^A$ with $A\geq 2$ and understanding that if $A=2$ then $y/(\log x)^2\to \infty$.
If \eqref{eq: asymp for S's} holds then Proposition 1 suggests that 
\[
M(x,y) \sim \sigma_+(A) \frac y{\log x} \text{ and } m(x,y) \sim \sigma_-(A) \frac y{\log x} 
\]
which is what we believe.

If we were to pre-sieve up to $y$ then Proposition 1 suggests that one should make a similar prediction but with 
$\sigma_+(A) $ replaced by  
\[
  \max_{x<X\leq 2x} \#\{ j\leq y:\ (X+j,P(y))=1\} \bigg/\frac{\phi(P(y))}{P(y)} y.
\]
(and $\sigma_-(A) $ by the analogous expression with the min). However we have no idea how  to study this ratio  in this restricted range
for $X$.

  \subsection{Longish intervals ($y\asymp (\log x)^2$) } 
  
  In section 1.3 we saw that if  $y=\lambda (\log x)^2$ then we should expect that 
 \[
 M(x,y) \sim u_+(  c_+\lambda)    \cdot  {\log x}
 \] 
Now $u_+(  c_+\lambda) \sim c_+\lambda$ as $\lambda\to \infty$ and so $M(x,y) \sim c_+ \frac y{\log x}$.
 This implies, letting $\lambda\to \infty$ and comparing this prediction to that in the last subsection, that $c_+=\sigma_+(2) $.
 
  Following the same heuristic  but now focusing on the minimum we see that if  $y=\lambda (\log x)^2$ then we should expect that 
 \[
 m(x,y) \sim u_-(  c_-\lambda)    \cdot  {\log x}
 \] 
for some constant $c_->0$. This analogously yields that $c_-=\sigma_-(2) $.

\subsection{More precise guesses for the maximal gap between primes}
   
 We can be more precise about our prediction for gaps between primes using the footnote in the proof of Proposition 1. The estimate there
   $N\leq (L-\tfrac 12+o(1))\log x$ with $L= \frac  {\phi(P)}P \log x$ which would suggest that 
 \[
 \max_{x<p_n\leq 2x} p_{n+1}-p_n \approx c_-^{-1} \log x \bigg( \log x - \frac 12\frac P{\phi(P)} \bigg) \approx
 c_-^{-1} \log x \bigg( \log x - \tfrac 12 \log\log x \bigg) .
 \]
 Here $P=P(z)$ and $c_-$   depend on $z$.
 
 Cadwell \cite{Cad} presented a variant of Cram\'er's model.  He took the viewpoint that certain aspects of the distribution of $H:=\pi(2x)-\pi(x)$ primes in $(x,2x]$ can be assumed to be like the distribution of $H$ randomly selected integers in $(x,2x]$. He very elegantly proved that the expected largest gap has length $\frac x{H+1} (\frac 11+ \frac 12 +\dots +\frac 1{H+1})$. This can be used to predict that\footnote{Cadwell's conjecture of $\log x(\log x -\log\log x)$ for the largest prime gap $\leq x$ was briefly mentioned in section 1.4. However since $x/\pi(x)$ is more accurately approximated by $\log x -1$, a famous correction of Legendre's prediction by Gauss, he should have deduced $(\log x-1)(\log x -\log\log x)$ from his model!  Here we are looking at gaps in $(x,2x]$ rather than up to $x$, which explains the difference in the constants.}
 \[
 \max_{x<p_n\leq 2x} p_{n+1}-p_n \approx  \log (4x/e) (\log x -\log\log x +\gamma).
 \]
 It is not clear how to incorporate divisibility by small primes into this argument, particularly working only with those intervals with an unexpectedly small number of integers left unsieved.   
 
 There are some similarities in these two conjectural formulas but it is not clear which to choose and on what basis.
 We did see in Figure 5 that the data suggests that one should subtract a larger multiple of $\log\log x$ in the formulas above but we have not found a believable heuristic to do so, though finding a way to combine the two heuristics would be a good start.

\section{Short arithmetic progressions}   We can proceed similarly with the distribution of $\pi(qy;q,a)$,  
the number of primes among the smallest $y$ positive integers in the  arithmetic progression $\equiv a \pmod q$,
as we vary over reduced residue classes $a \pmod q$ and where $y$ is small compared to $q$.  As before we sieve out with the primes $\leq z$ (that do not divide $q$) before trying to find primes.
If $P_q(z):=\prod_{p\leq z,\ p\nmid q} p$ then the probability that a random such integer of size $q^{1+o(1)}$ is prime is
\[
\sim \frac{ qP_q(z)}{\phi(qP_q(z))} \frac 1{\log q}
\]

Now the number of unsieved integers in such an interval of length $y$ is expected to be 
\[
\frac{\phi(P_q(z))} { P_q(z)}  y,
\]
and so the ``expected'' number of primes is 
\[
\sim \frac{ q}{\phi(q)} \frac y{\log q}
\]
(which is what suggested by the prime number theorem for arithmetic progressions).
This set up allows us to proceed much as in the questions about primes on short intervals.
We shall explore this in detail, with copious calculations, in a subsequent article.

  \appendix
  
\section{The largest prime gap conjecture in computing range}

In section 1.4, particularly in figure 5, we saw that our predictions for $\displaystyle\max_{p_n\leq x} (p_{n+1}-p_n)$ appear to be significantly too large. The technique we used to make our prediction involves several asymptotic predictions for the distribution of primes and for the sieve and so any of these may be sufficiently far out for small integers that this might have led to the difference from the data that we have seen.  Our belief is that the main issue is the sieving and not the probabilistic argument and so we test that in this section.  We take an example near to the upper limit of what is currently computable:

We take $\log x=40$: The largest prime gap up to $x$ is $1248$ immediately following $218034721194214273$.
The Cram\'er prediction is $1600$ and ours is $1797$.  We follow the argument in this paper:

We want to determine the maximal gap $y$ which should be (at a first guess) around $(\log x)^2=1600$ (at least according to Cram\'er), so we will now study sieving all intervals of length $1600$ with the primes $\leq z=\tfrac 12 \log x=20$. Define $P=P(20)$ and
\[
R(n):=\#\{ X \pmod {P}: S(X,y,z)=n\} \text{ where } n=:c_n \tfrac{\phi(P)}{P} y.
\]
In the notation of Proposition 1, we want   $N=n$ to be as large as possible so that $k_-=0$ where $L=\frac{\phi(P)}{P}\log x$ and $x$ (there) equals
$R(n)x/P$ here. Proposition 1 then suggests that we should take $N\sim L \log (R(n)x/P)$. Referee \#3 observed that the proof of Proposition 1 indicates that 
replacing $L$ in this formula by $\frac 1{-\log(1-\frac 1L)}$ is more accurate, and indeed is about 7.5\% better for this value of $L$. This then suggests that \[
y \approx \   \max_n \frac {37}{c_n} (23.91 +\log R(n)).
\]
as $\log (x/P)\approx 23.91$ (where we had ``40'', which is $\log x$, instead of ``37'', before the referee's suggestion).
We can easily determine this function for each $n$ on a computer, and from this we obtain a prediction of $y=1420$,\footnote{It was $y=1536$ before the referee's intervention.} significantly smaller than either previous prediction, but still unaccountably larger than the truth. The data   for each $n$  is given in the following:

\begin{figure}[H]\centering{

\centerline{\vbox{\offinterlineskip \halign{\vrule #&\ \ #\ \hfill
&& \hfill\vrule #&\ \ \hfill # \hfill\ \ \cr \noalign{\hrule} \cr
height5pt&\omit && \omit && \omit & \cr &$ \ \ n $&& $R(n)$ &&
 \hfill$\frac {37}{c_n} (23.91 +\log R(n))$\hfill & \cr
height5pt&\omit && \omit && \omit & \cr \noalign{\hrule} \cr
height3pt&\omit && \omit && \omit & \cr
&  234 && 24 && 1040.2 & \cr
           height3pt&\omit && \omit && \omit & \cr
& 235 && 784 && 1169.0  & \cr
           height3pt&\omit && \omit && \omit & \cr
& 236 && 6392 && 1244.0 & \cr
           height3pt&\omit && \omit && \omit & \cr
& 237 && 32404 && 1300.3 & \cr
           height3pt&\omit && \omit && \omit & \cr
& 238 && 123540 && 1345.4 & \cr
           height3pt&\omit && \omit && \omit & \cr
& 239 && 342796 && 1378.1 & \cr
           height3pt&\omit && \omit && \omit & \cr
& 240 && 737536 && 1401.0 & \cr
           height3pt&\omit && \omit && \omit & \cr
& 241 && 1263416 && 1415.3 & \cr
           height3pt&\omit && \omit && \omit & \cr
& 242 && 1714444 && 1420.8 & \cr
           height3pt&\omit && \omit && \omit & \cr
& 243 && 1841372 && 1417.6 & \cr
           height3pt&\omit && \omit && \omit & \cr
& 244 && 1569650 && 1405.9 & \cr
           height3pt&\omit && \omit && \omit & \cr
& 245 && 1075420 && 1386.3 & \cr
           height3pt&\omit && \omit && \omit & \cr
& 246 && 594076 && 1359.0 & \cr
           height3pt&\omit && \omit && \omit & \cr
& 247 && 265624 && 1324.2 & \cr
           height3pt&\omit && \omit && \omit & \cr
& 248 && 95356 && 1281.8 & \cr
           height3pt&\omit && \omit && \omit & \cr
& 249 && 28584 && 1233.1 & \cr
           height3pt&\omit && \omit && \omit & \cr
& 250 && 6652 && 1175.8 & \cr
           height3pt&\omit && \omit && \omit & \cr
& 251 && 1320 && 1113.2 & \cr
           height3pt&\omit && \omit && \omit & \cr
& 252 && 268 && 1051.9 & \cr
           height3pt&\omit && \omit && \omit & \cr
& 253 && 32 && 972.3 & \cr
           height3pt&\omit && \omit && \omit & \cr
 \noalign{\hrule}}} }
 
\caption{Data when $y=1420$}
}
\end{figure} 

\noindent We see that there are about $1.71$ million intervals mod $P(20)$ of length $1420$ which contain exactly $242$ integers that are coprime to $P(20)$. The probabilistic argument then suggests that some of the corresponding intervals in $(x,2x]$ contain no primes at all.
If instead we work with $P(25)$ then our prediction reduces a little but not much, and indeed we tried all the obvious possibilities but could not manipulate the variables to construct a prediction that would reduce $1420$ to anywhere near the truth, namely $1248$.

\section{Is the model valid?}

\subsection{A first example, $x=10^8, y=340, z=11$}

For $ x=10^8$  we are going to study the distribution of primes in intervals of length $y=340\approx (\log x)^2$, which lie between $x$ and $2x$, grouping them according to the value of $S(X,y,z)$ where $z=11$.

A quick calculation reveals that $S(X,340,11)$ takes each value between $68$ and $73$.  
Let $C(N):=\#\{ m\pmod P: S(m,y,z)=N\}$.
As discussed in section 8, we have $S(X,y,z)=S(m,y,z)$ whenever $X\equiv m \pmod {P(z)}$, so that
\[
  I(N)=\bigcup_{ m \in C(N)}   \{ X\in (x,2x]:  X\equiv m \pmod P\},
\]
and therefore $\#I(N)= \frac xP \#C(N) +O(P)$.  A simple calculation yields that $P(11)=2310$ with
\[
\#C(68)=28,  \#C(69)=228,  \#C(70)=784,  
\]
\[
\#C(71)=820,  \#C(72)=386,  \#C( 73)=64.
\]

For each $N\in [68,73]$ we define, for each integer $h$,
\[
I(N,h):= \{  X\in I(N): \pi(X+y)-\pi(X)=h\}.
\]
Then we create the bar graph where the column rooted at $h$ on the vertical axis has height $\# I(N,h)$.

We wish to compare this to our assumptions, and to the  binomial distribution.
The first thing we might want to look at is how the sieving effects the probability of being prime. Thus if $\mu(N)$ is the calculated mean number of primes in an interval in $I(N)$, then we are interested in the probability of an unsieved integer being prime, namely $1/L(N)$ where
$L(N)=N/\mu(N)$.  In our model we would take $L=\frac {\phi(P)}P \log x=3.82767\dots$; but to compare this to small data we need to be more precise, noting that a better approximation to 
\[
\frac{x}{\pi(2x)-\pi(x)} \text{ is given by } \log 4x/e,
\]
and using this we have $L=\frac {\phi(P)}P \log 4x/e=3.90794 \dots$
Our data yields
\[
L(68)= 3.8665\dots, 
L(69) = 3.8847\dots, 
L(70) = 3.8977\dots, 
\]
\[
L(71) = 3.9133\dots, 
L(72) = 3.9265\dots,  
L(73) =  3.9418\dots,
\]
which are all reasonably close to $L$ (no more than about 1\% out).
The $L$-values here appear to be growing, more or less linearly, which deserves an explanation.
A `best fit' approximation yields that $L(N)\approx L + 0.01478(N-70.69)$.

Next we compare what the binomial distribution predicts to the actual counts for primes when $S(X,y,z)=N$.
Here $N$ runs from $68$ to $73$ and we graph $I(N,h)$ compared to the prediction 
\[
 \binom Nh    \frac 1{L^h}\  \bigg( 1- \frac 1{L} \bigg)^{N-h}  
\]
from the binomial distribution. We also mark the mean $\mu(N)$ number of primes in these intervals, as well as $m_N(x,y), M_N(x,y)$, the minimum and maximum number of primes in such intervals, and $m(x,y), M(x,y)$, the global minimum and maximum.

\begin{figure}[H]\centering{
\includegraphics[scale=.5]{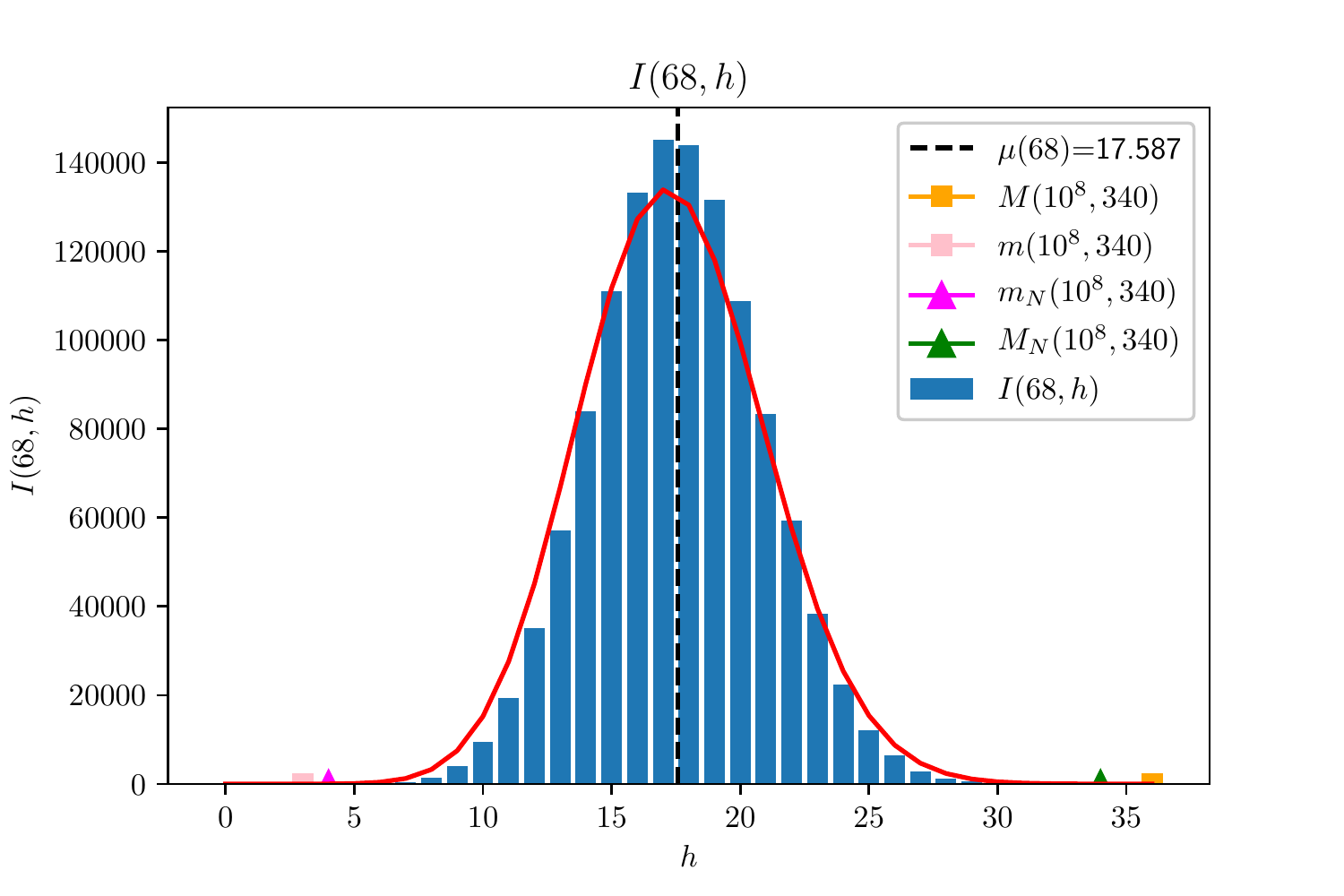}
\includegraphics[scale=.5]{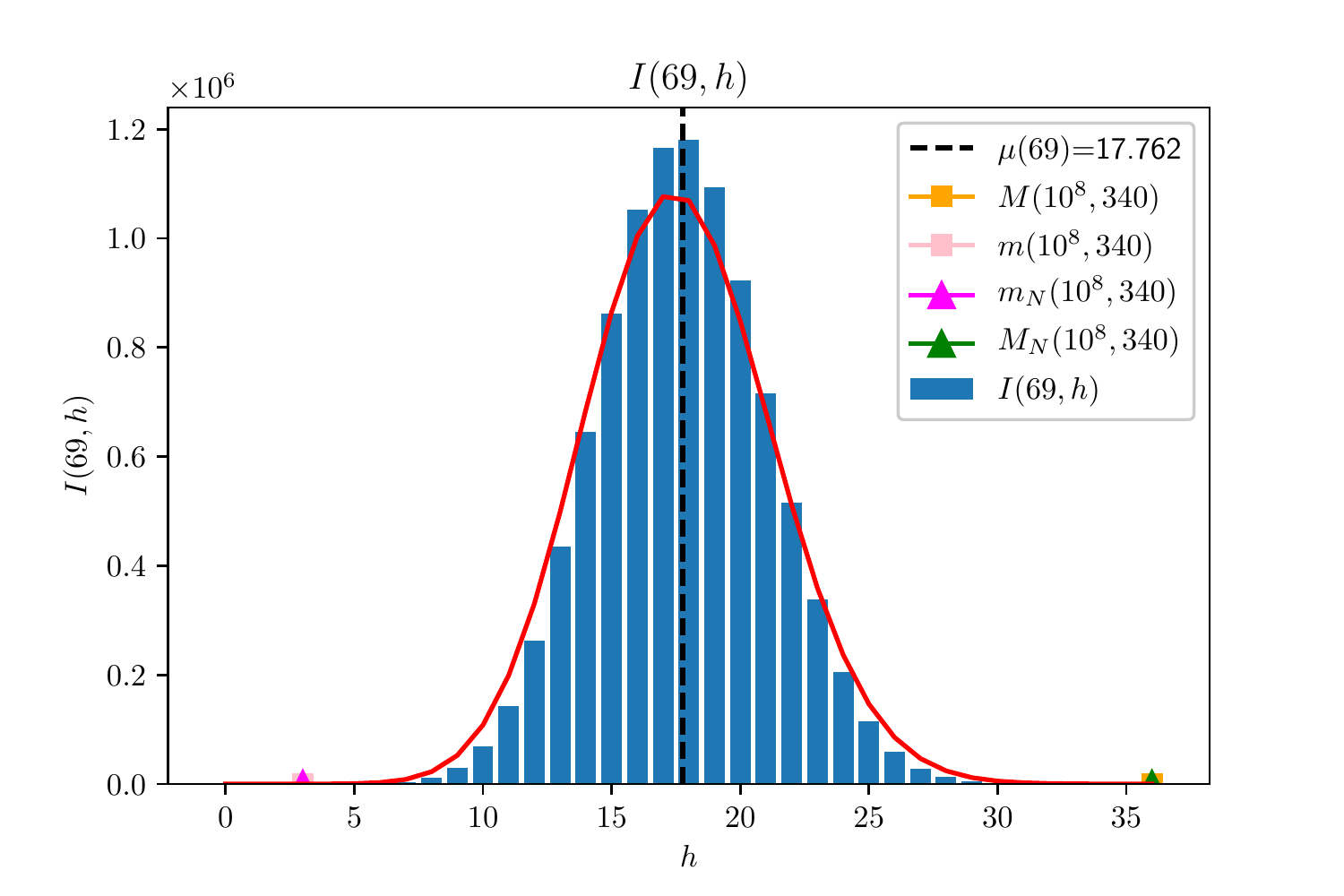}
\newline

\includegraphics[scale=.5]{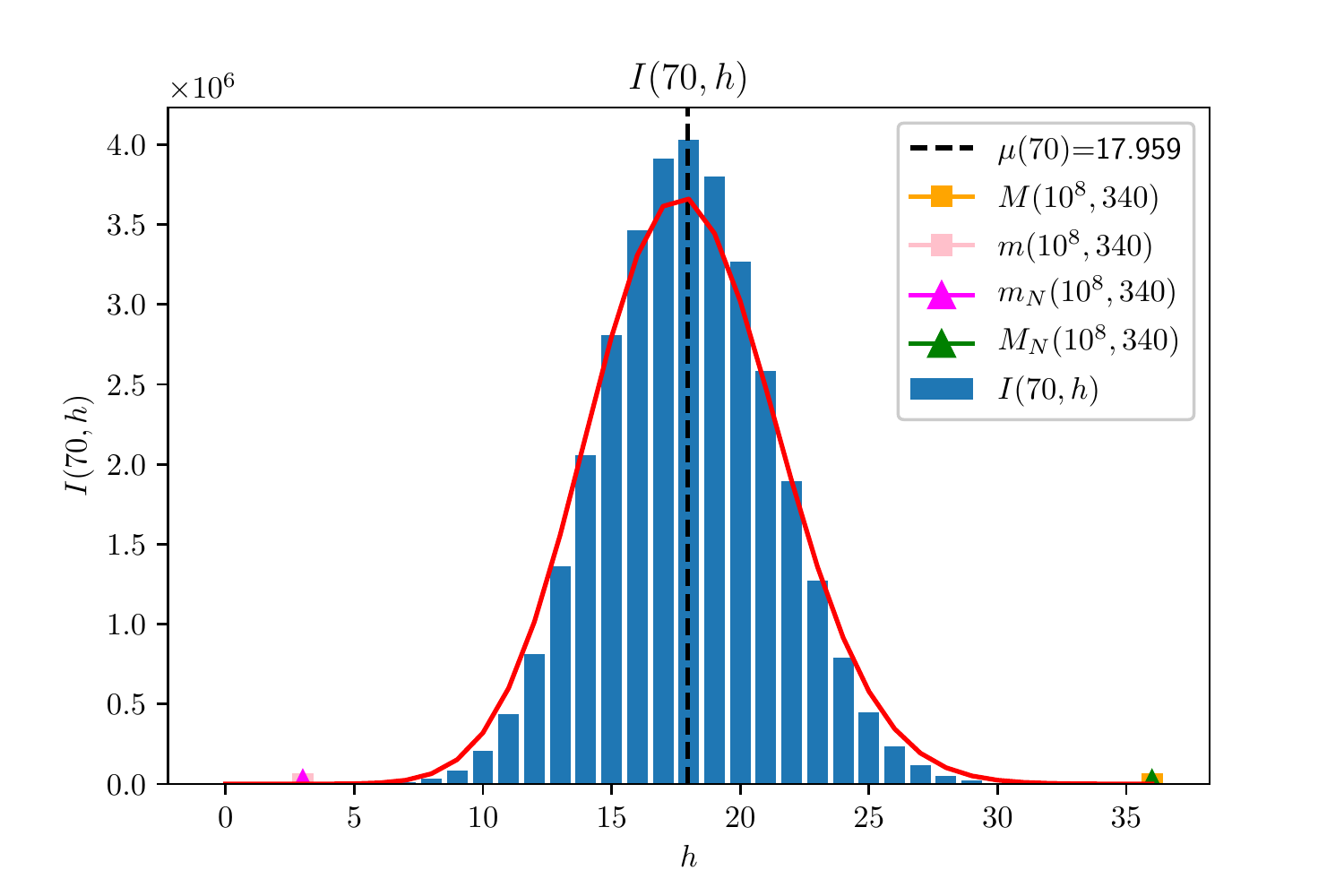}
\includegraphics[scale=.5]{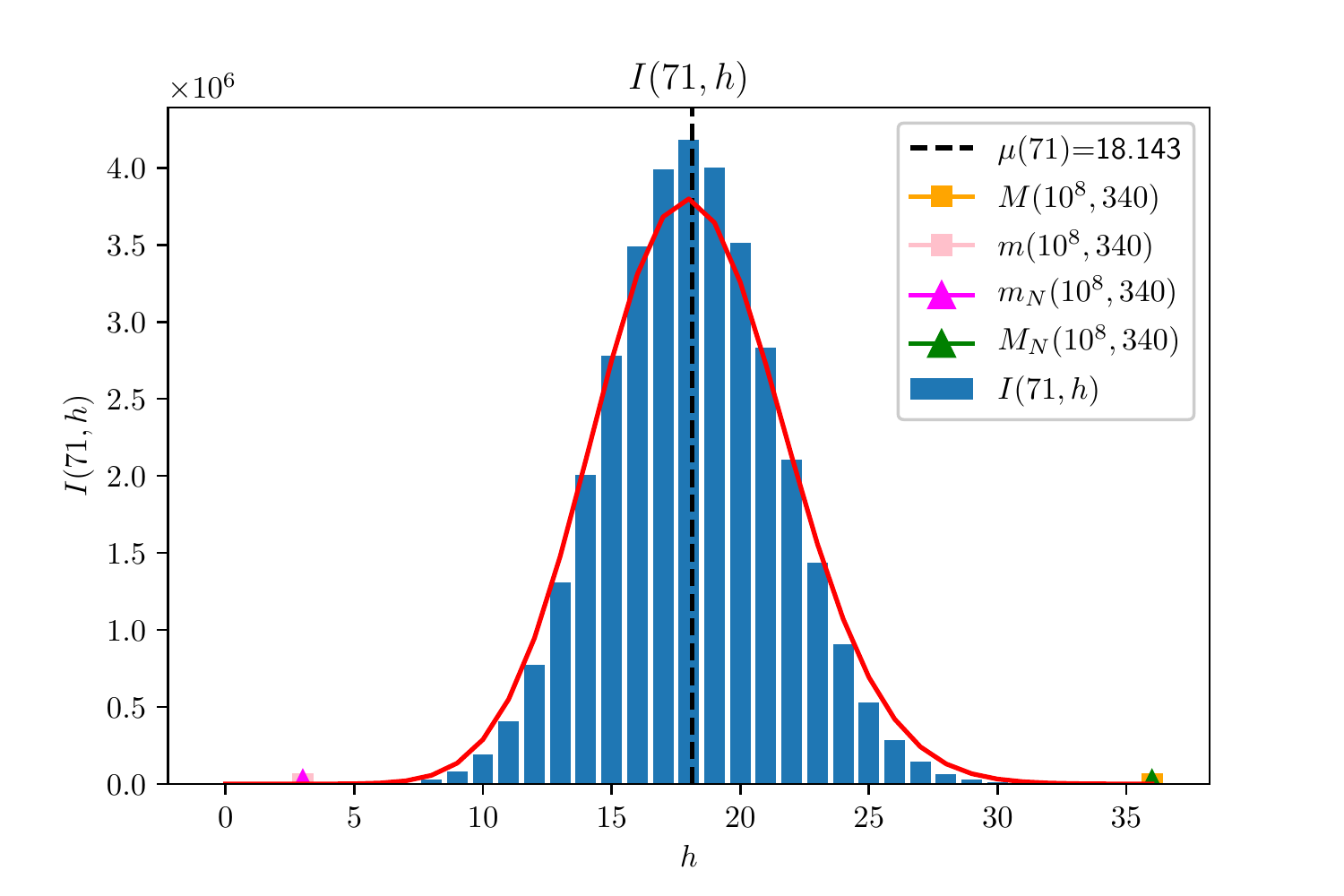}
\newline

\includegraphics[scale=.5]{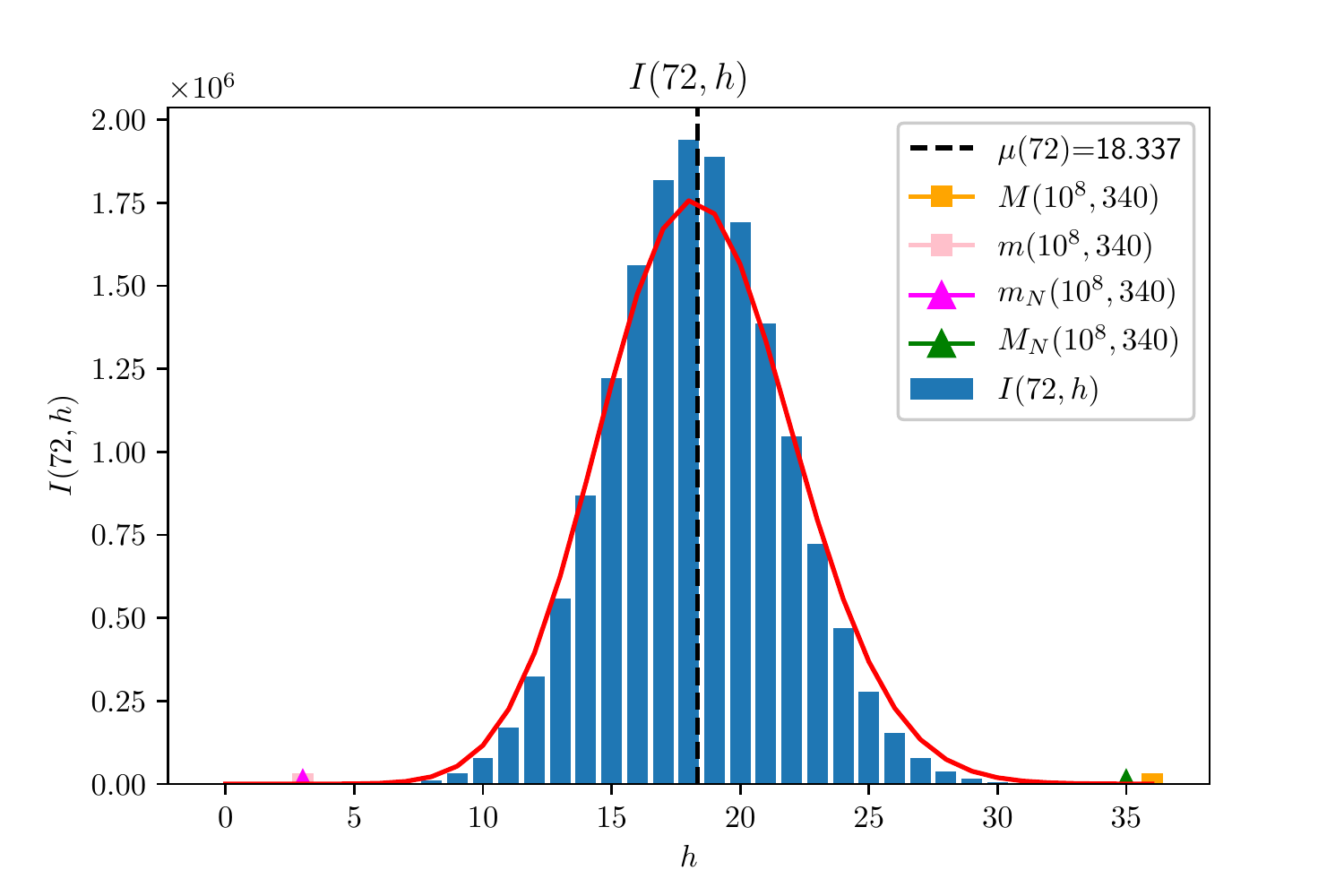}
\includegraphics[scale=.5]{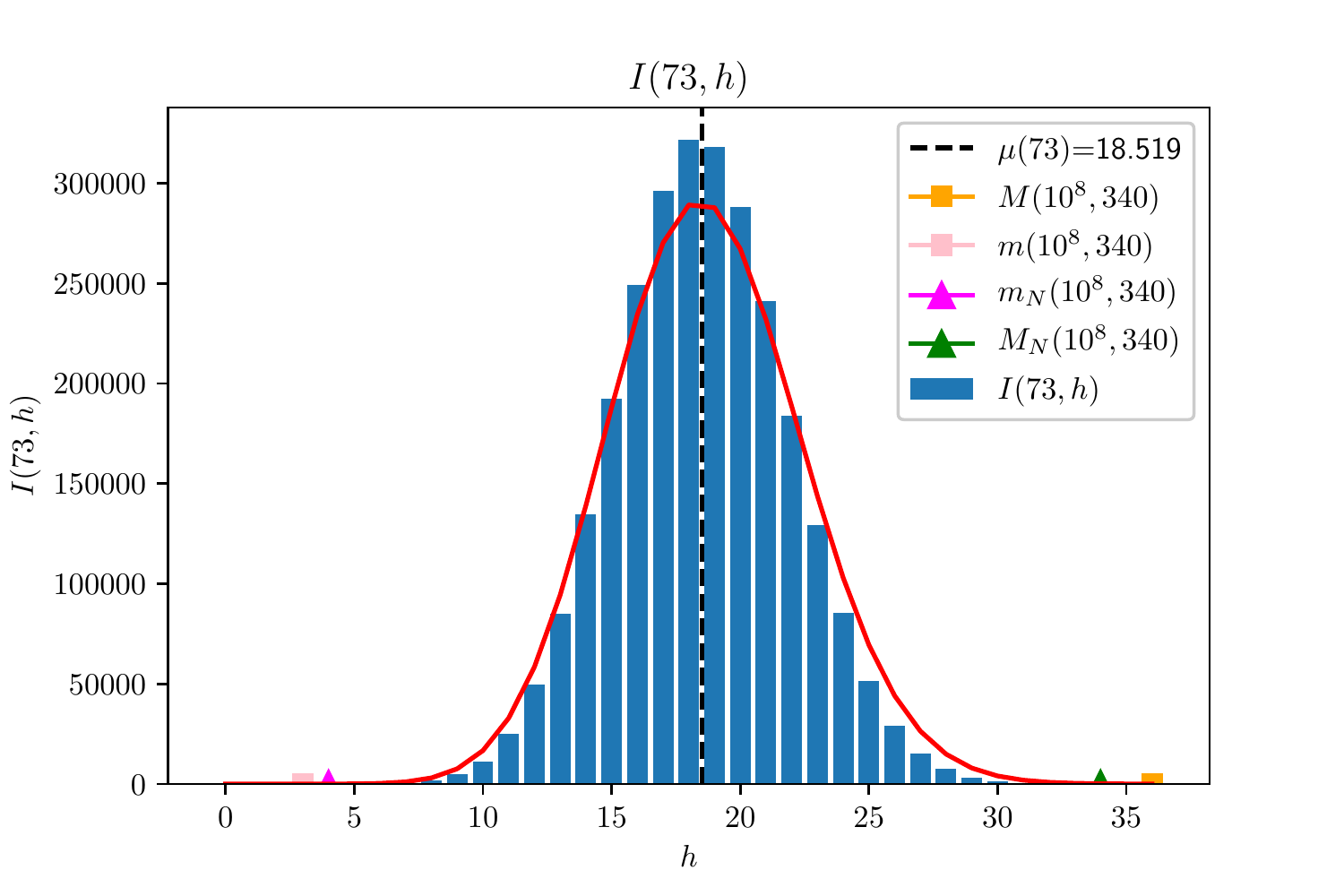}  }
\caption{Testing the distributions, $h$ vs $I(N,h)$, for each $N$ in our range.}
\end{figure} 

In each case we see that our prediction has the same basic shape as the data (a Bell curve) but is wider than the data, with less density around the mean.
We can analyze this by simply looking at the mean and variance compared to what is expected from our model.

\[ 
\begin{array}{rcccccc}
 N:  & \underline{68} & \underline{69} & \underline{70} & \underline{71} & \underline{72} & \underline{73}  \\
\mbox{Expected mean:} & 17.40 & 17.66 & 17.91 & 18.17 & 18.42 & 18.68 \\
\mbox{Actual mean:} & 17.59 & 17.76 & 17.96 & 18.14 & 18.34  & 18.52 \\
\mbox{Expected variance:} & 12.95 &  13.14 &  13.33 &  13.52 & 13.71 & 13.90 \\
\mbox{Actual variance:} & 10.82  & 10.93  & 11.06  & 11.17  & 11.25 & 11.34 
\end{array}
\] 
  
 Although both the actual and expected means increase with $N$ we see that the actual mean increases more slowly than the expected.
 More striking is that the actual variance, that is the variance given by the data, is far smaller than in our prediction. 
 
 According to  Montgomery and Soundararajan \cite{HLM} we should have 
 \[
 \sum_{X=x}^{2x} (\psi(X+y)-\psi(X)-y)^{2k} \sim y^k \cdot \int_{t=x}^{2x} \bigg(\log \frac{e^{-\gamma}t}{2\pi y} +1  \bigg)^k dt
 \]
 for $\log x\leq y\leq x^{1/2k}$. Therefore the variance here (for the primes) is, more-or-less
 \[
\frac y{x(\log x)^2}\cdot \int_{t=x}^{2x} \bigg(\log \frac{e^{-\gamma}t}{2\pi y} +1  \bigg) dt =
\frac y{\log x}\cdot   \frac{\log \frac{2 e^{-\gamma}x}{\pi y} }{\log x} .
 \]
 Thus a first approximation gives 
 mean $\frac y{\log x}\approx 18.46$ and variance $\approx 11.586$.
 If we replace $\log x$ by $\log 4x/e$ (since this gives a more accurate description of the density of primes in $[x,2x]$) then we get $\approx 18.08$ and $\approx 11.11$, respectively.
 This corresponds very well to the data.

\subsection{A second example, $x=10^8, y=500, z=17$}

Here $S(X,500,17)$ takes each value between $84$ and $97$.
Now $P(17)=510510$ and the $C$-values are given by
\[
\begin{array}{c|ccccccc}
h& 84 & 85 & 86 & 87 & 88 & 89 & 90  \\
\hline
\#C(h) & 52 &  576 &  3764 &  15836 &  47186 &  91432 &  125688  \\
\end{array}
\]
\[
\begin{array}{c|ccccccc}
h&   91 & 92 & 93 & 94 & 95 & 96 & 97\\
\hline
\#C(h) &    115800 &  70096 &  29428 & 8050 &  1520 &  212 &  28\\
\end{array}
\]
We see that there are very few such intervals for the outlying $h$-values, and indeed the data for these $h$-values does not conform to the patterns that we observe.

We have that $L=\frac{\phi(P(z))}{P(z)}\log(4x/e)= 3.39513\dots$ and our data yields the following $L$-values to four decimal places
\[
\begin{array}{c|ccccccc}
h& 84 & 85 & 86 & 87 & 88 & 89 & 90  \\
\hline
L(h) & 3.3853 &  3.3805& 3.3845 & 3.3843& 3.3873 & 3.3906 & 3.3938 \\
\end{array}
\]
\[
\begin{array}{c|ccccccc}
h&   91 & 92 & 93 & 94 & 95 & 96 & 97\\
\hline
L(h) & 3.3974 & 3.4011   & 3.4043 & 3.4062 & 3.4082 &  3.4156& 3.4450\\
\end{array}
\]
Again it is usually within 1-2\% of the true $L$-value, but is slightly increasing.
Our best linear approximation is $L(N)\approx L + .003054(N-90.09)$. The corresponding graphs are given by

\begin{figure}[H]\centering{
\includegraphics[scale=.3]{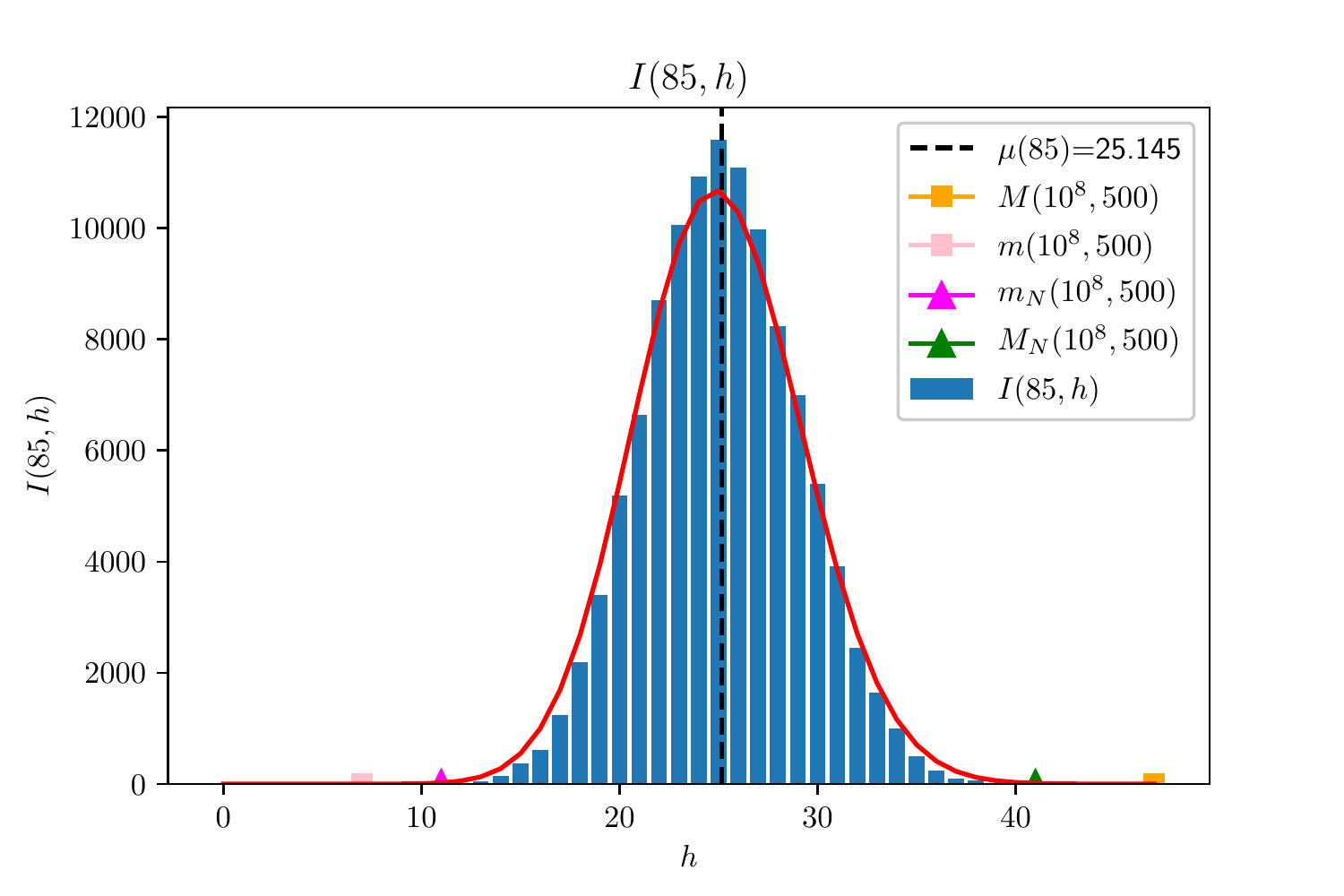}
\includegraphics[scale=.3]{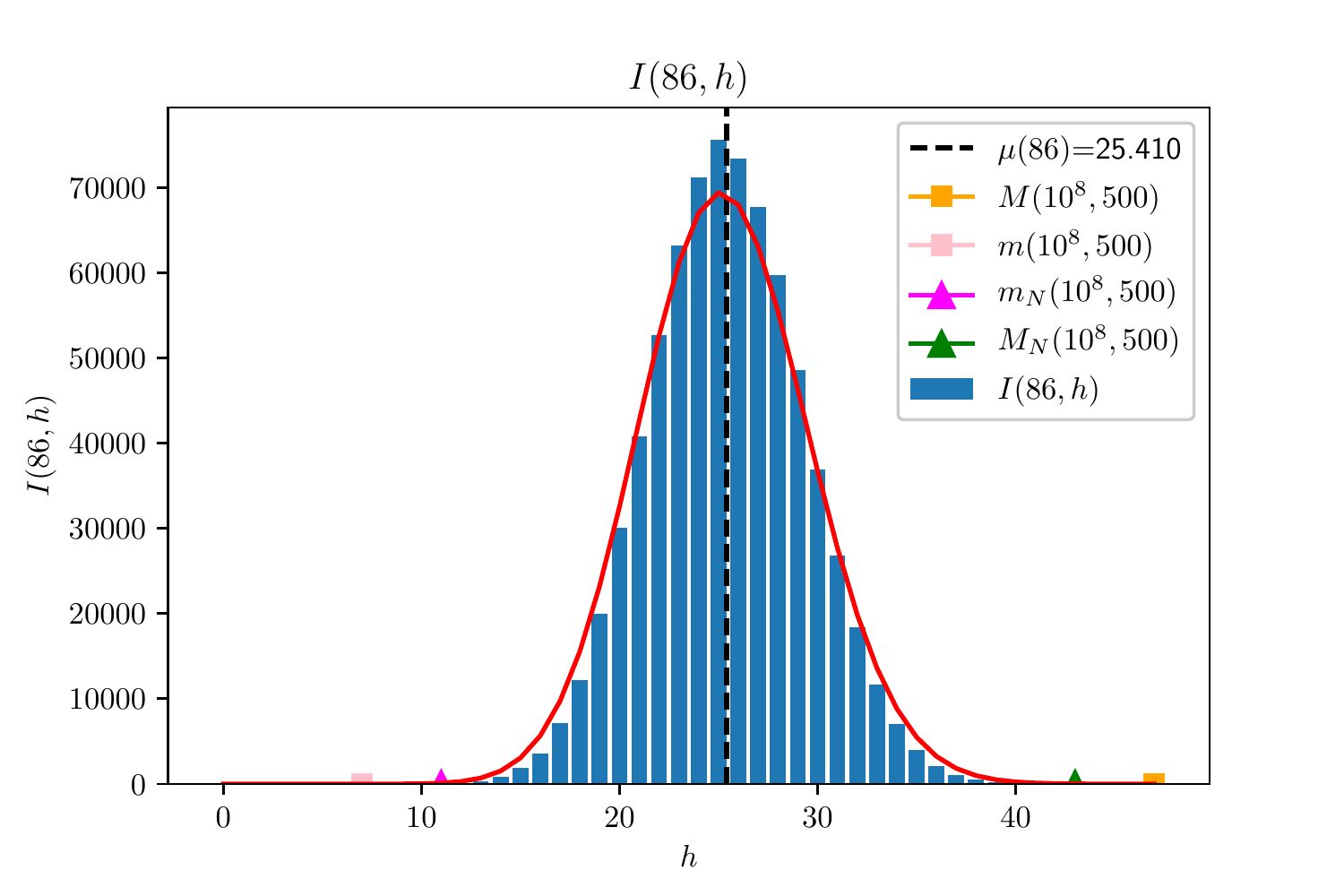}
\includegraphics[scale=.3]{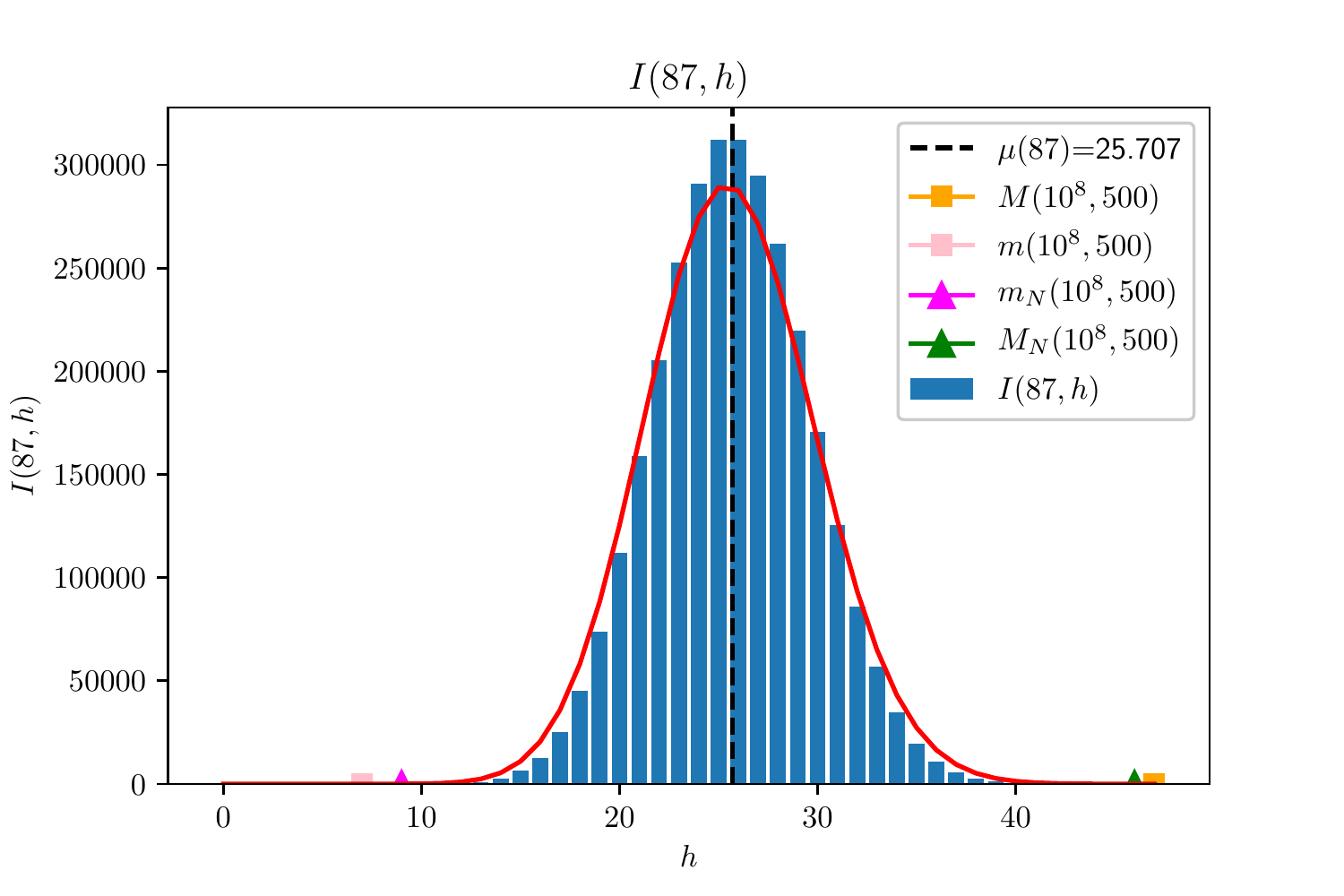}
\newline

\includegraphics[scale=.3]{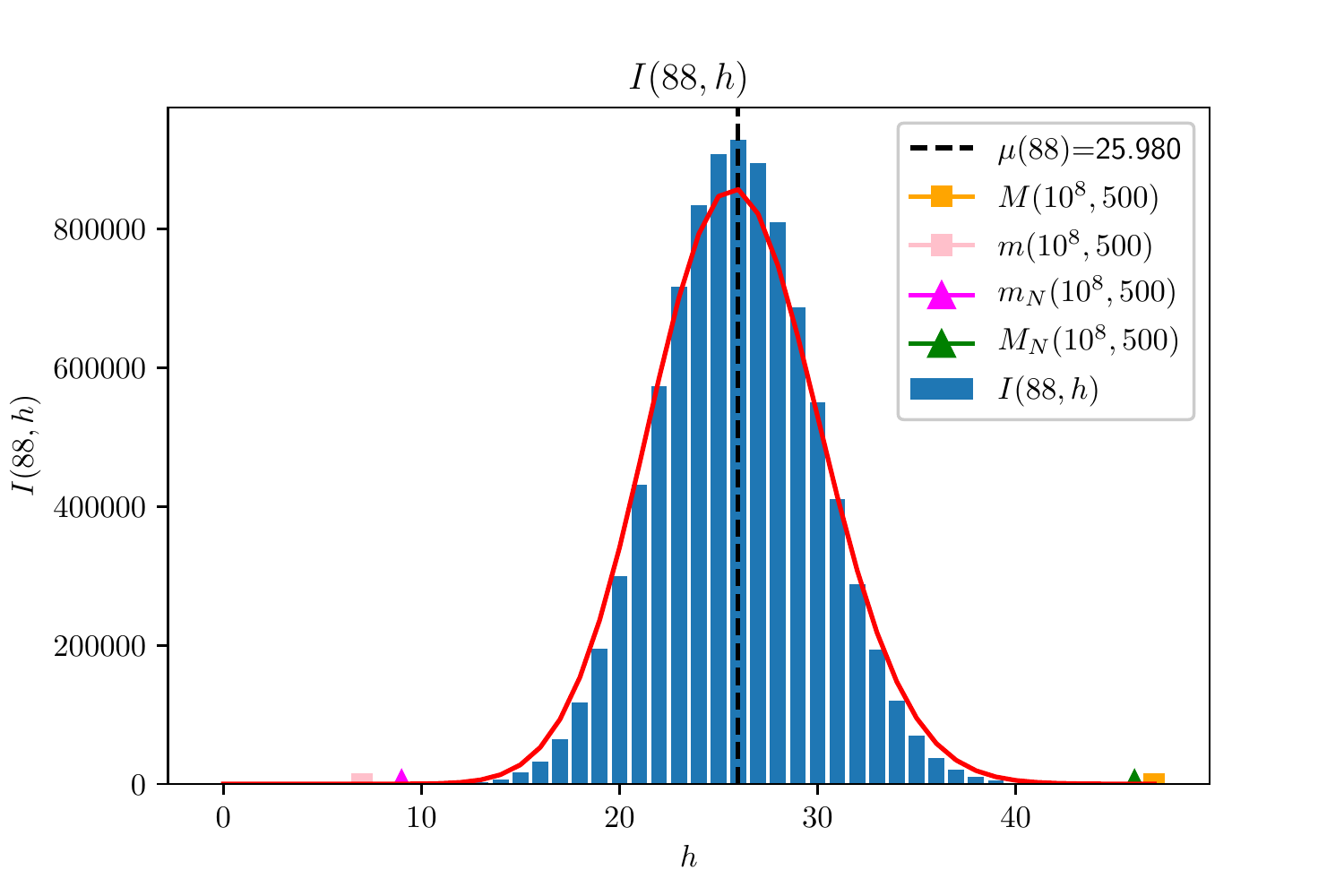}
\includegraphics[scale=.3]{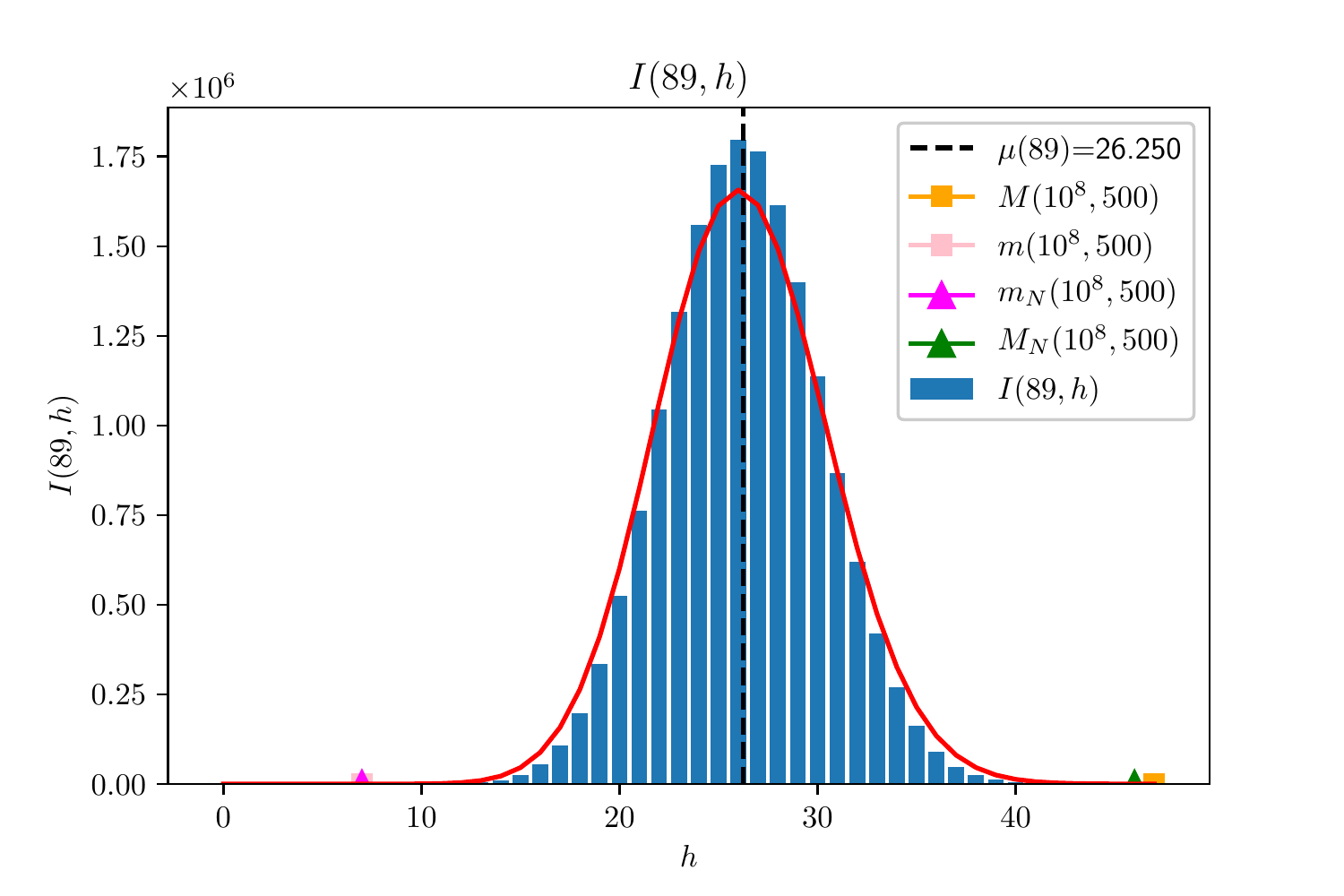}
\includegraphics[scale=.3]{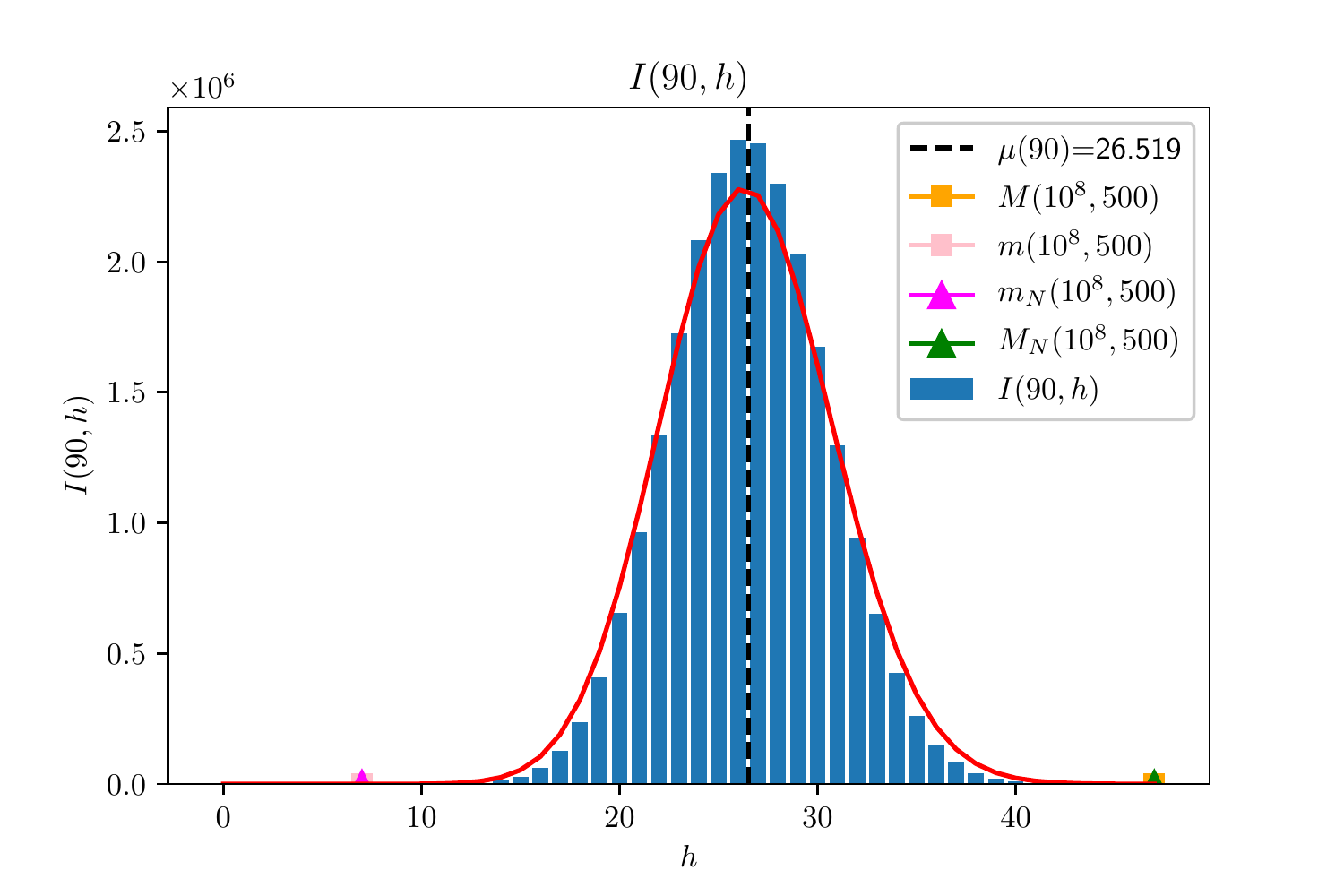}
\newline

\includegraphics[scale=.3]{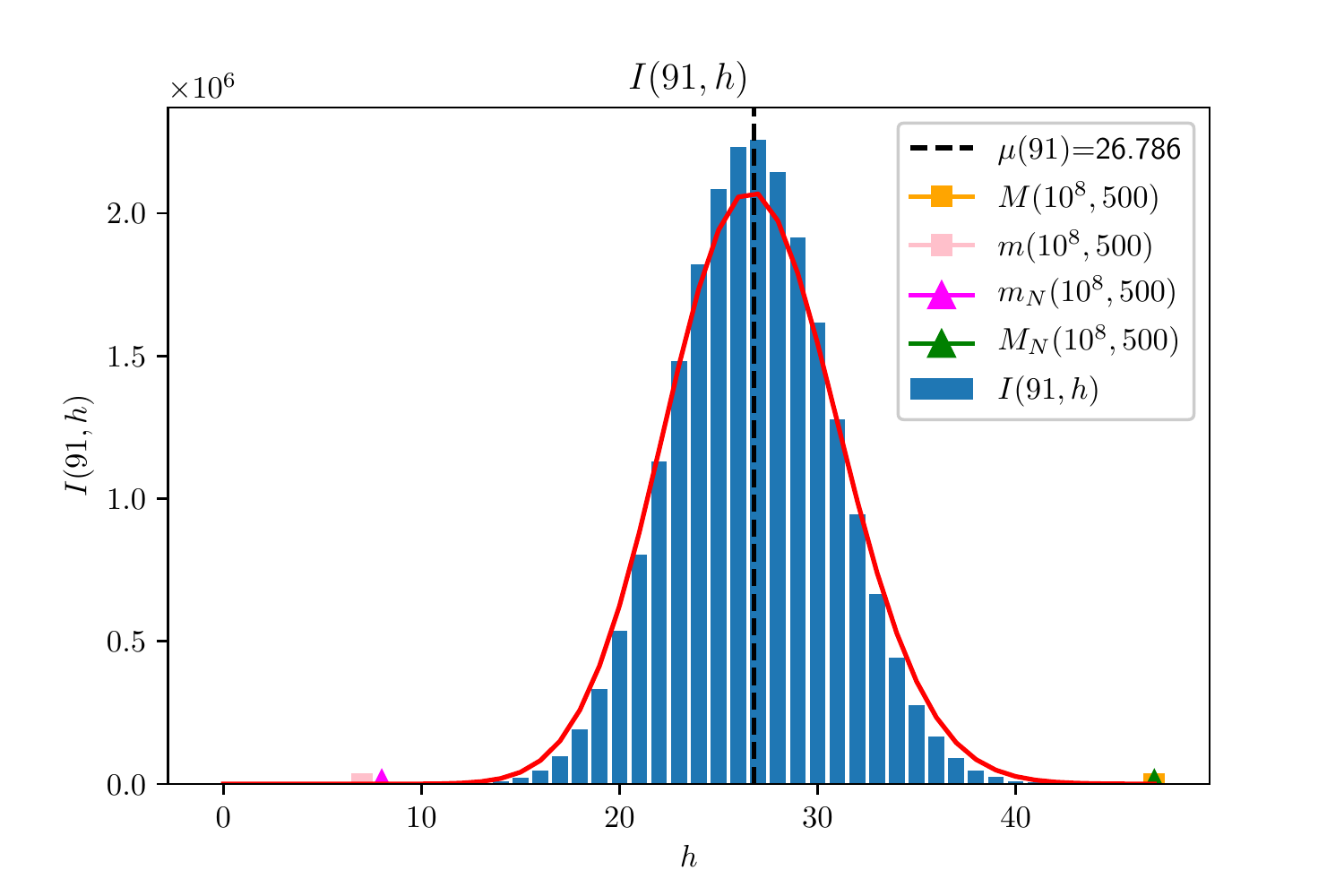}
\includegraphics[scale=.3]{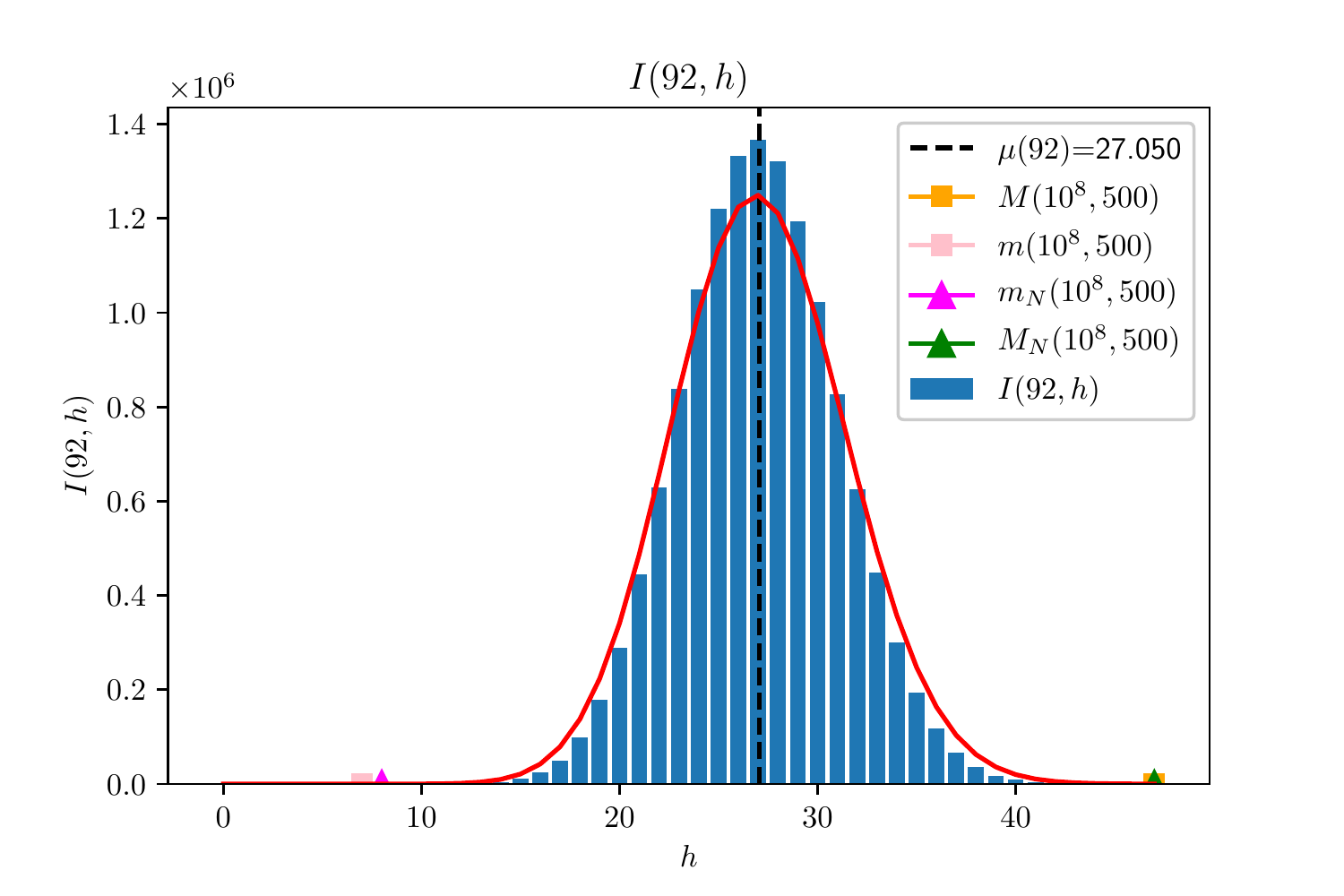}
\includegraphics[scale=.3]{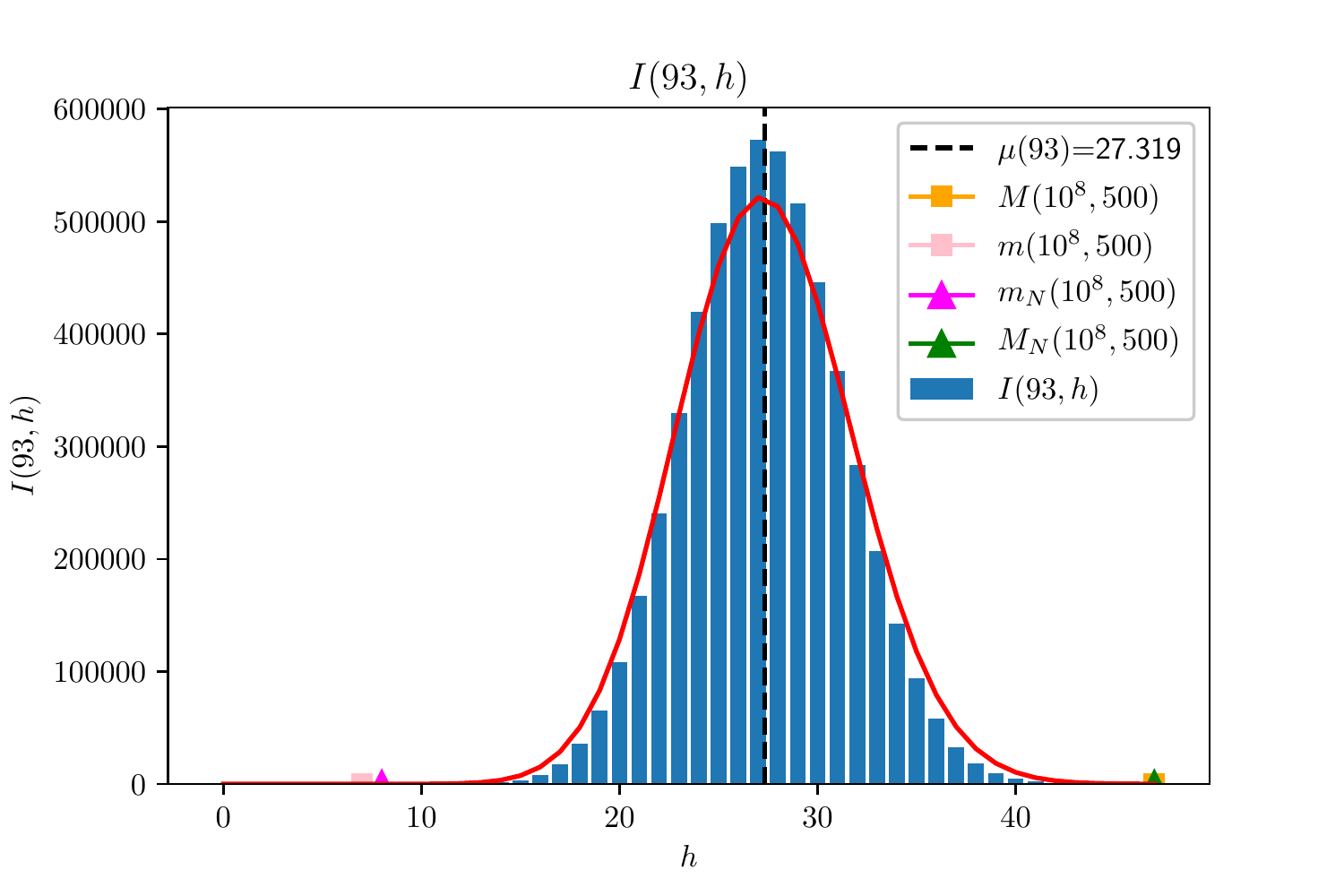}
\newline

\includegraphics[scale=.3]{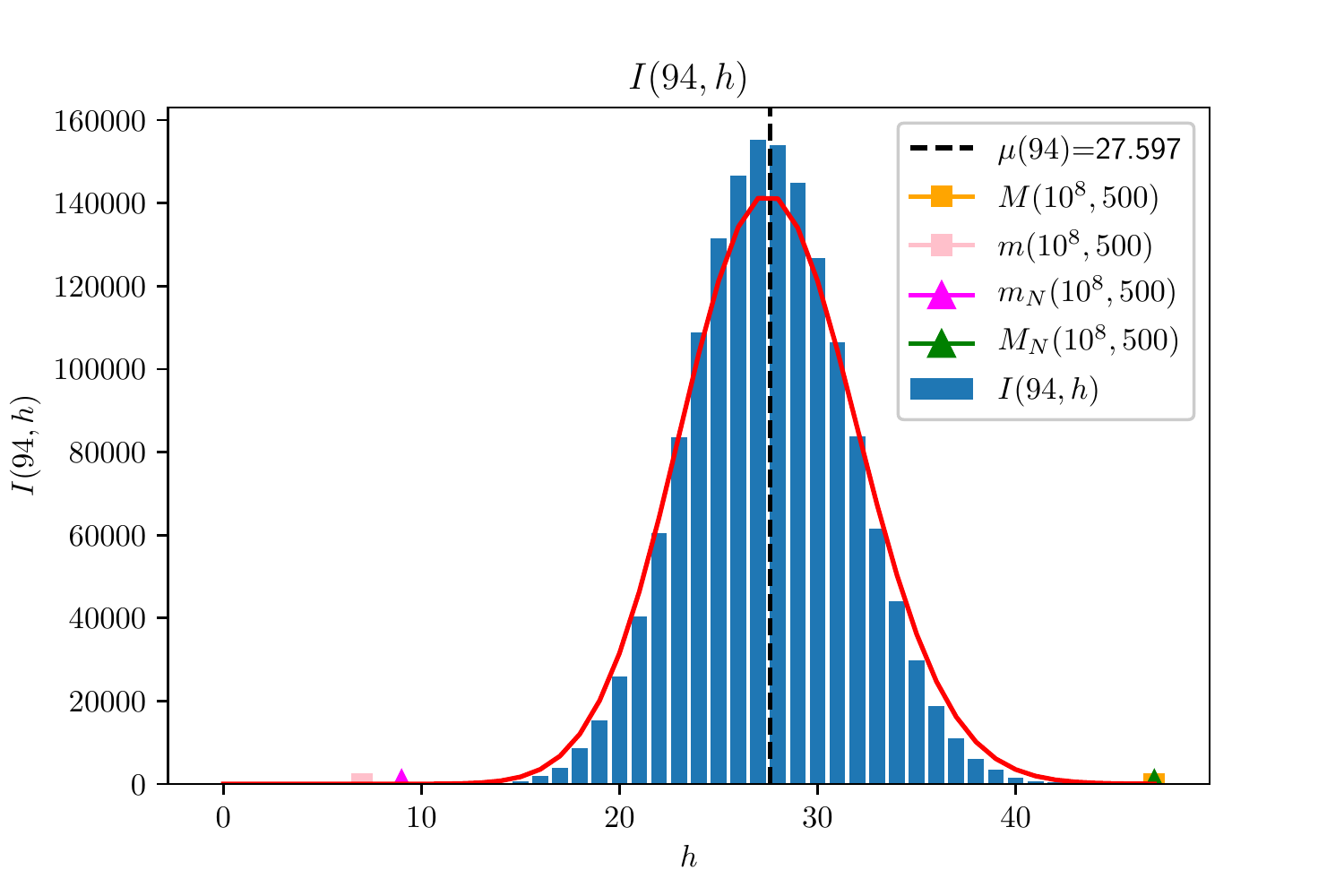}
\includegraphics[scale=.3]{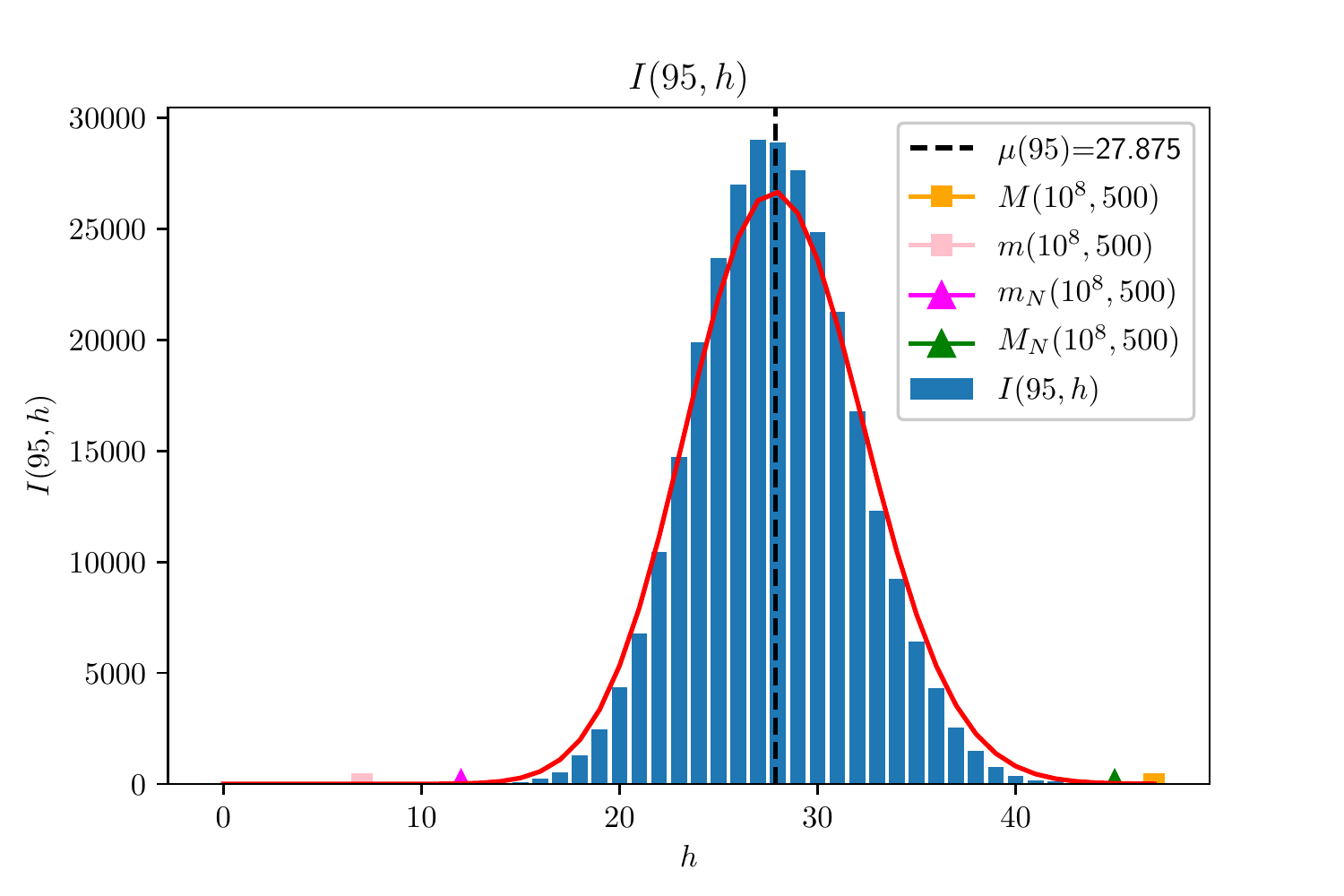}
\includegraphics[scale=.3]{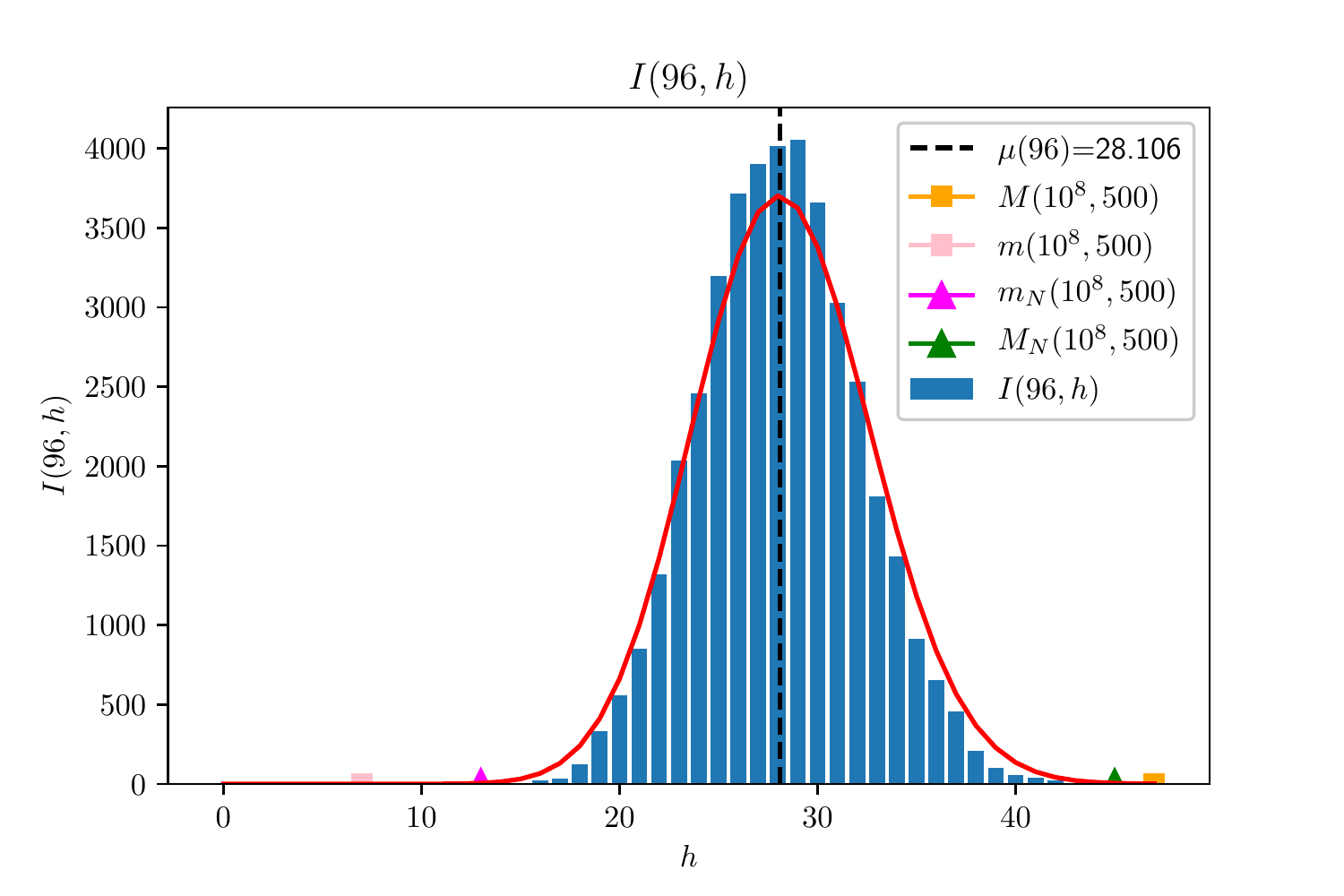}
\newline
\caption{Testing the distributions, $h$ vs $I(N,h)$, for $85\leq N\leq 96$.} }
\end{figure} 

{\small
\[
\begin{array}{c|cccccccccccc}
h& 85 & 86 & 87 & 88 & 89 & 90 &   91 & 92 & 93 & 94 & 95 & 96  \\
\hline
\text{Data Mean} & 25.15 & 25.42 & 25.71 &   25.99 &   26.25 &   26.52 &   26.79 &   27.06 &   27.32 &   27.60 &   27.88 &   28.11 \\ 
 \hline
 \text{Exp Mean}  & 25.04 &  25.33  & 25.62 & 25.92 & 26.21 & 26.51 & 26.80 & 27.1 & 27.39 & 27.69 & 27.98 & 28.28 \\
 \hline
\text{Data Var} & 15.26 & 15.21 & 15.29 & 15.44 & 15.56 & 15.67 & 15.80 & 15.94 & 16.02 & 16.18 & 16.32 & 16.20\\
\hline
\text{Exp Var} &  17.71 & 17.91 & 18.12 & 18.32 & 18.51 & 18.71 & 18.91 & 19.10 & 19.30 & 19.50 & 19.70 & 19.88\\
\end{array}
\]
}

 Replacing  $\log x$ by $\log 4x/e$ as i the first example the overall expected mean is $26.5858\dots$ and the new expected variance is
 $15.8003\dots$, which again is a pretty good fit with this data.

The data in this appendix makes a compelling case that one should develop a different model, stemming from the binomial distribution, but in which the $X_n$ are not independent. Instead,   their dependence must imply that the number of primes in short intervals of length $y$ between $x$ and $2x$ satisfies the normal distribution with the variance predicted by Montgomery and Soundararajan, and then perhaps we might see what this new model might give for tail probabilities. We would thus revise our predictions for $M(x,y), m(x,y)$ and the largest gaps between consecutive primes.\footnote{Though hopefully only in the secondary terms, so as not to invalidate the conjectures in this paper!} We hope to return to this key topic in a further paper.

\section{Pre-sieving intervals of length $y$ by the primes up to $y$}


Fix $x$ and $y$, let $P=P(y)$ and assume that $S(y)\sim \frac y{\log y}$.
Recall that
$ I(N)=\{ X\in (x,2x]:\ S(X,y,y) =N\} $, where $0\leq N\leq S(y)$, and let $\# I(N)=:x^{\theta_N}$. Now 
 \[
 \max_N \# I(N)\geq x/(S(y)+1)\geq x/y \geq x^{1-o(1)},
 \]
so there exist $N$-values for which $\theta_N=1+o(1)$. It is not hard to show that 
$S(X,y,y) = \frac {\phi(P)}Py+O(y^{1/2+o(1)})$ for almost all $X \mod P$; but we cannot assume that the distribution of $\# I(N)$ is comparable in the
restricted interval  $X\in(x,2x]$, with the distribution in the much larger set $[0,P)$.

We will use Proposition 1 with $L=\frac{\phi(P)}{P} \log x\sim e^{-\gamma} \frac{\log x}{\log y}$ and $x$ (there) equal to $\# I(N)$ to predict the values of 
\[
M_N(x,y) := \max\{ \pi(X,X+y]:\ x<X\leq 2x \text{ and } S(X,y,y)  =N\}
\]
for each $N$ with $I(N)$ non-empty.   From these predictions we obtain our  predictions for
 \[
 M(x,y) = \max_N  M_N(x,y).
 \]

In section \ref{applyingmodel-veryshort}, the independence hypothesis of Proposition 1 was satisfied as the intervals were disjoint. Here the intervals in $I(N)$ might overlap, so we replace $I(N)$ by $I'(N)$, the  largest subset of $I(N)$  of disjoint intervals.  
Evidently $\# I(N)\geq \# I'(N)\geq \# I(N)/y$ so  $ \# I'(N)=x^{\theta_N}/y^{O(1)}$; the  $y^{O(1)}$-factor is irrelevant in applying Proposition 1  when $\theta_N>0$.

We will focus our heuristic on those integers $N$ for which $\theta_N=1+o(1)$ (working with other $N$ will only affect our heuristic in the range with $Y\ll\log x$, as we discuss   in a footnote).  Therefore we let $N_*=N_*(x,y)$  be the largest integer $N$  for which $\theta_N=1+o(1)$ and $c_*:=c_{N_*}$, where we define $c_N$ by  $N=:c_N\frac {\phi(P)}Py$ for each $N$.

  \medskip
  
\noindent \textbf{Predictions, by  pre-sieving up to $y$}:  {\sl
 If $\log x\ll y\leq (e^\gamma/c_*) \log x$ then\footnote{Had we included  $N$-values for which $\theta_N<1$ in our calculation then instead we would have predicted that
 $M(x,y)   \sim c_N\cdot e^{-\gamma} \frac{y}{\log y}$ in the range $\log x\ll y\leq (\theta_N/c_N)e^\gamma \log x$ where $c_N$ is chosen as large as possible so that  $c_N y\leq e^\gamma \theta_N \log x$. This makes sense since $y$ is fixed, and our job is to select the optimal $N$-value. However, if $\theta_N<1$ then this new prediction leads to  complications:  At the smallest $x$-value in this range we have the prediction
 $M(x,y)   \sim \theta_N  \frac{\log x}{\log y}$, whereas the next range begins with the prediction $M(x,y)   \sim  \frac{\log x}{\log y}$. Since there can be no
 discontinuity in these predictions that means that there must be at least one other $(x,y)$-range with a different $N$-value in-between, etc.  Because this gets so complicated we made the choice to make the simplifying assumption (Occam's razor) that we select from those $N$ with $\theta_N=1+o(1)$ in our heuristic.}
 \[
 M(x,y)   \sim c_*\cdot e^{-\gamma} \frac{y}{\log y}.
 \]
 If $(e^\gamma/c_*) \log x\leq y = o( (\log x)^2)$ then 
 \[
 M(x,y) \sim   \frac {  \log x }{ \log(\frac {(\log x)^2}y)} .
 \]
 Finally if $y=\lambda (\log x)^2$ with $\lambda>0$ then 
 \[
 M(x,y) \sim \max_N   c_N \delta_+(c_N\lambda/\theta_N) \cdot \frac y{\log x}.
 \]
 }
 \medskip
 
 If $\lambda$ is large and $y=\lambda (\log x)^2$ then
 \[
 c_N \delta_+(c_N\lambda/\theta_N) = c_N + \sqrt{ \frac{2\theta_Nc_N}\lambda } + O\bigg( \frac {1}{\lambda} \bigg) ,
 \]
 and so $M(x,\lambda (\log x)^2)\sim c_\dag \frac y{\log x}$ as $\lambda\to \infty$ where $c_\dag = \max_N c_N$ where the maximum is taken over all those $c_N$ with $\theta_N\gg 1$.
 
 These predictions are substantially more complicated than those obtained when pre-sieving up to $\epsilon \log x$. By Occam's razor, we choose to follow the other path though it is feasible that both will yield the same prediction if only we could at least partly resolve the relevant sieve questions
(that is, determine the values of $c_\dag, c_*$ and $\max_N  \{  c_N :\ c_N  \leq u \theta_N\}$ for each $u>0$).
 
 \begin{proof}[Deduction of the above predictions from Proposition 1]
 
 Evidently $\sum_N \# I(N)=x$, each $N\leq S(y)$ and 
 \[
 \sum_N  N \# I(N)= y \#\{ n\in (x,2x]:\ (n,P)=1\} +O(y^2) \sim \frac{\phi(P)}P x y,
 \] 
 so that $\#\{ n\in (X,X+y]:\ (n,P)=1\}$ averages $\sim \frac{\phi(P)}P y$ over all $X\in (x,2x]$. 
 We can restrict attention in both sums to those $N$ with $\theta_N=1+o(1)$ with only a negligible error term, and so by taking the average over such $N$ we deduce that  $c_*\geq 1$.

We take the largest subset of the  intervals in $I(N)$ that begin at least $y$ apart (so there are $\# I(N) y^{O(1)}$ such intervals).
We can employ Proposition 1 with $L\sim e^{-\gamma} \frac{\log x}{\log y}$, so that $\log L\sim \log\log x$. This yields that
  \[
 M_N(x,y)\sim \begin{cases} \, \qquad N &\text{ if }     N\leq   \frac{\log  \# I(N)}{\log\log x} ;\\
      \frac {\log  \# I(N)}{ \log \big(\tfrac{L\log  \# I(N)}{N}  \big) } & \text{ if }       \frac{\log  \# I(N)}{\log\log x} \leq  N=o(L\log  \# I(N));\\
   \delta_+(\lambda) \tfrac NL &\text{ if } N=    \lambda  L\log  \# I(N) \text{ with } \lambda>0.
 \end{cases}
 \]
The first range is $c_Ny\lesssim e^\gamma  \theta_N \log x$, and therefore the maximum occurs when $N=N^*$ provided $y \leq c_*^{-1}  e^\gamma  \log x$.
For those $N$ with $\theta_N<1$ the first range might be applicable for larger $y$. However for these $y$ the predicted value of $M_{N^*}(x,y)$ in the second range will be larger than those $M_N(x,y)$.  Obtaining the results in the other two ranges is straightforward.
 \end{proof}

 \bibliographystyle{plain}

\begin{thebibliography}{99}

\bibitem{BFT} William Banks, Kevin Ford and Terence Tao,  
\emph{Large prime gaps and probabilistic models},
(preprint).

\bibitem{Cad}  J.H. Cadwell,  
\emph{Large intervals between consecutive primes},
Math. Comp. \textbf{25} (1971), 909--913.

\bibitem{Cra}Harald Cram\'er, 
\emph{On the order of magnitude of the difference between consecutive prime numbers},
Acta Arithmetica. \textbf{2} (1936), 23--46. 

\bibitem{Fell} William Feller, 
\emph{An introduction to probability theory and its applications}, Vol. II. (2nd ed.)
Wiley, New York, 1971

\bibitem{FGK}  Kevin Ford, Ben Green, Sergei Konyagin and Terence Tao,
\emph{Large gaps between consecutive prime numbers},
Ann. of Math. \textbf{183} (2016),  935--974.
 
 
 \bibitem{FGM}  Kevin Ford, Ben Green, Sergei Konyagin, James Maynard,  and Terence Tao,
 \emph{Long gaps between primes},
J. Amer. Math. Soc.  \textbf{31} (2018),   65--105. 

\bibitem{FG} John Friedlander and Andrew Granville,  
\emph{Limitations to the equi-distribution of primes. I},
Ann. of Math.  \textbf{129} (1989),   363--382.

\bibitem{FI} J.B.~Friedlander and H.~Iwaniec, 
\emph{Opera de Cribro},  
AMS Colloquium Publications \textbf{57}  
American Mathematical Society, 2010.



\bibitem{Gr1} Andrew Granville,
\emph{Harald Cram\'er and the distribution of prime numbers},
Harald Cram\'er Symposium (Stockholm, 1993).
Scand. Actuar. J. \textbf{1} (1995),  12--28.

\bibitem{Gr2} Andrew Granville,
\emph{Primes in intervals of bounded length},
Bull. Amer. Math. Soc. \textbf{52} (2015), 171--222.

\bibitem{Gr3} Andrew Granville,
\emph{Sieving intervals and Siegel zeros},
(preprint)

 \bibitem{HL} G. H. Hardy and J. E. Littlewood, 
\emph{Some problems of ``Partitio Numerorum'', III: On the expression of a number as a sum of primes}, 
Acta Math. \textbf{44} (1923), 1--70.
 
 \bibitem{Iw1}  Henryk Iwaniec, 
\emph{On the problem of Jacobsthal}, 
Demonstratio Math. \textbf{11} (1978), 225--231.

 \bibitem{JR}  W.B. Jurkat and H.-E. Richert,  
\emph{An improvement of Selberg's sieve method. I.}
Acta Arith. \textbf{11} (1965), 217--240.

\bibitem{Mai} Helmut Maier,  
\emph{Primes in short intervals},
Michigan Math. J. \textbf{32} (1985),   221--225.

\bibitem{MS} Helmut Maier,   and Cam L. Stewart,  
\emph{On intervals with few prime numbers},
J. Reine Angew. Math. \textbf{608} (2007), 183--199.

\bibitem{May1} James Maynard, 
\emph{Small gaps between primes},  
Ann. of Math.   \textbf{181} (2015),  383--413. 

\bibitem{May2} James Maynard, 
\emph{Large gaps between primes},  
Ann. of Math.   \textbf{183} (2016),  915--933.

\bibitem{May3} James Maynard, 
\emph{Sums of two squares in short intervals}, arXiv:1910.13384

\bibitem{HLM}  Hugh L. Montgomery and K. Soundararajan
\emph{Primes in short intervals},  Comm. Math. Phys. \textbf{252}  (2004),  589--617. 

 \bibitem{Se1}  A. Selberg,  
\emph{Sieve methods},
ch. 36 of Collected Works, Vol I, Springer-Verlag, New York 1989.
\textsl{published originally as:} Proc. Sympos. Pure Math \textbf{20} (1971), 311--351.

\bibitem{Tao} Terence Tao, \emph{Polymath8b: Bounded intervals with many primes, after Maynard}, Blog note. \url{https://terrytao.wordpress.com/2013/11/19/polymath8b-bounded-intervals-with-many-primes-after-maynard/}

\bibitem{Xyl}  Triantafyllos  Xylouris, 
\emph{\"Uber die Nullstellen der Dirichletschen $L$-Funktionen und die kleinste Primzahl in einer arithmetischen Progression},
Ph.D. thesis, Universit\"at Bonn, Mathematisches Institut,  Bonner Mathematische Schriften  \textbf{404} (2011), 110pp.



\bibitem{Zha} Yitang Zhang,\,
\emph{Bounded gaps between primes},
Ann. of Math.  \textbf{179} (2014),  1121--1174.

\end{thebibliography}

\end{document}